\documentclass[12pt]{amsart}
\usepackage[margin=1in]{geometry}  
\usepackage[pdftex, dvipsnames]{xcolor}
\usepackage{graphicx}              

\usepackage{multirow}
\usepackage{amssymb}
\usepackage{amsmath}               
\usepackage{amsfonts}              
\usepackage{amsthm}                
\usepackage{caption} 
\usepackage[lofdepth,lotdepth]{subfig}
\usepackage{tikz}
\usetikzlibrary{cd}
\usetikzlibrary{decorations.pathreplacing,intersections}
\usepackage[shortlabels]{enumitem}
\usepackage{overpic}

\usepackage[pdftex,%
    final,%
    colorlinks=true,%
    linkcolor=violet,%
    citecolor=violet,%
    filecolor=violet,%
    menucolor=violet,%
    urlcolor=violet,%
    bookmarks=true,%
    bookmarksdepth=3,%
    bookmarksnumbered=true,%
    bookmarksopen=true,%
    bookmarksopenlevel=2,%
]{hyperref}

\definecolor{specialteal}{RGB}{0, 128, 128}
\definecolor{specialorange}{RGB}{255, 153, 85}
\usepackage[textsize=tiny]{todonotes}
\usepackage{thm-restate}
\newcommand{\Z}{\mathbb Z}
\DeclareMathOperator{\Aut}{Aut}
\DeclareMathOperator{\im}{Im}

\newtheorem{thm}{Theorem}[section]
\newtheorem{lem}[thm]{Lemma}

\newtheorem{prop}[thm]{Proposition}
\newtheorem{cor}[thm]{Corollary}

\theoremstyle{remark}
\newtheorem{remark}[thm]{Remark}

\theoremstyle{definition}

\newcommand{\ZZ}{\mathbb{Z}}      

\usepackage{cleveref}
\usepackage[backend=bibtex,maxbibnames=99,firstinits=true,style=ieee,sorting=nty,style=numeric]{biblatex}
\DeclareFieldFormat{postnote}{#1}
\allowdisplaybreaks
\begin{document}

\title{On $\chi$-slice pretzel links}

\author[Fanelle]{Sophia Fanelle}
\author[Huang]{Evan Huang}
\author[Huenemann]{Ben Huenemann}
\author[Shen]{Weizhe Shen}
\author[Simone]{Jonathan Simone}
\author[Turner]{Hannah Turner}

\maketitle 

\begin{abstract}
A link is called $\chi-$slice if it bounds a smooth properly embedded surface in the 4-ball with no closed components and Euler characteristic 1.
If a link has a single component, then it is $\chi-$slice if and only if it is slice. One motivation for studying such links is that the double cover of the 3-sphere branched along a nonzero determinant $\chi-$slice link is a rational homology 3-sphere that bounds a rational homology 4-ball.
This article aims to generalize known results about the sliceness of pretzel knots to the $\chi-$sliceness of pretzel links. In particular, we completely classify positive and negative pretzel links that are $\chi-$slice, and obtain partial classifications of 3-stranded and 4-stranded pretzel links that are $\chi-$slice. As a consequence, we obtain infinite families of Seifert fiber spaces that bound rational homology 4-balls.
\end{abstract}

\section{Introduction}\label{intro}
A knot in $S^3$ is called slice if it bounds a smooth properly embedded disk in $B^4$. 
One motivation behind studying slice knots is that the double cover of $S^3$ branched along a slice knot is a rational homology 3-sphere that bounds a rational homology 4-ball. The question of which rational homology 3-spheres bound rational homology 4-balls is a well-known open problem in low-dimensional topology.

There are two generalizations of sliceness to links---the classical notion of \textit{slice links} and the more recently defined notion of \textit{$\chi-$slice links} \cite{donaldowens}.
A link $L\subset S^3$ is \textit{$\chi$-slice} if it bounds an embedded surface in $B^4$ with no closed components and Euler characteristic $\chi=1$. We call a smooth properly embedded Euler characteristic 1 surface in $B^4$ without closed components whose boundary is $L$ a \textit{$\chi-$slice surface} for $L$. 
Moreover, if $F$ is a $\chi-$slice surface for $L$ and $F$ can be smoothly isotoped rel boundary so that the radial distance function restricted to $F$ has no local maxima, then $L$ is called \textit{$\chi-$ribbon} and $F$ is called a \textit{$\chi-$ribbon surface}.
Figure \ref{fig:chi-slice-surface} shows an example of a $\chi-$ribbon link $L$ along with a $\chi-$ribbon surface for $L$.
Note that a knot is $\chi-$slice (resp. $\chi-$ribbon) if and only if it is slice (resp. ribbon).
This generalization of sliceness to links is natural from the 4-manifold perspective as the following result shows. 

\begin{figure}[b!]
    \centering
    \includegraphics[width=.28\textwidth]{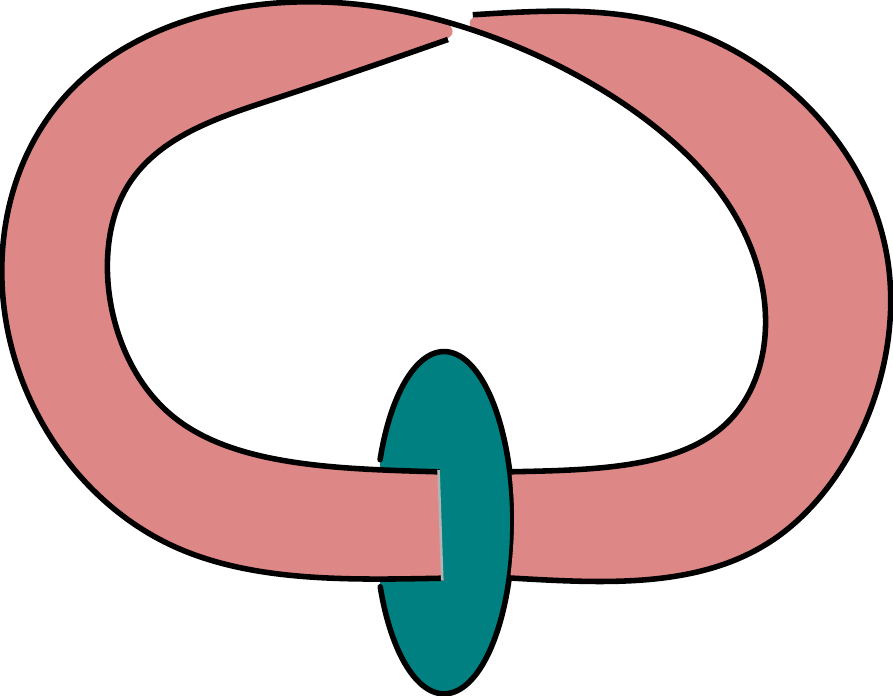}
    \caption{An example of a $\chi-$slice surface $F$ bounded by the link $T(2,4)$. $F$ is the union of a  disk and a Mobius band.}
    \label{fig:chi-slice-surface}
\end{figure}

\begin{prop}[Proposition 2.6 in \cite{donaldowens}] If $L$ is a $\chi-$slice link with \emph{det} $L\neq 0$, then the double cover of $S^3$ branched along $L$ is a rational homology 3-sphere that bounds a rational homology 4-ball.
\label{prop:donaldowens}
\end{prop}

Generically, an infinite family of nonzero determinant $\chi-$slice links, gives rise to an infinite family of rational homology 3-spheres bounding rational homology 4-balls.

\begin{remark}
Recall that an $n$ component link $L\subset S^3$ is \textit{slice} if $L$ bounds the disjoint union of $n$ smooth properly embedded disks in $B^4$. It is known that if a link $L$ is slice, then $\det L=0$; consequently the double cover of $S^3$ branched along $L$ is not a rational homology 3-sphere. Hence the notions of nonzero determinant $\chi-$slice links and slice links (with more than 1 component) are disjoint.
\end{remark}

In this paper, we will explore the $\chi$-sliceness of pretzel links with the aim of extending known results regarding the the sliceness of pretzel knots in \cite{greenejabuka}, \cite{lecuona}, \cite{miller}, and \cite{kimleesong}. We first use ribbon moves to construct infinite families of $\chi-$slice pretzel links. Then, to obstruct other pretzel links from being $\chi-$slice, we use Proposition \ref{prop:donaldowens} and expand on the sliceness obstructions developed in \cite{greenejabuka} stemming from Donaldson's Diagonalization Theorem and Heegaard Floer correction terms; moreover, we apply several geometric obstructions specific to links of multiple components.

The $m$-stranded pretzel link $P(p_1,\ldots,p_m)$ is the link depicted in Figure \ref{fig:pretzel}, where the boxes labeled $p_i$ indicate $p_i$ half-twists. 
We will focus on three classes of pretzel links: positive (and negative) pretzel links, 3-stranded pretzel links, and 4-stranded pretzel links.
Note that $P(p_1,\ldots,p_m)$ is isotopic to $P(p_m,\ldots,p_1)$ and if $(p_{i_1},\ldots, p_{i_m})$ is a cyclic reordering of $(p_1,\ldots,p_m)$, then $P(p_{i_1},\ldots, p_{i_m})$ is isotopic to $P(p_1,\ldots,p_m)$. 
We set:
$$P\langle p_1,\ldots,p_m\rangle:=\{P(p_{i_1},\ldots,p_{i_m})\,:\,(p_{i_1},\ldots,p_{i_m})\text{ is any reordering of }(p_1,\ldots,p_m)\}$$
We denote the mirror of a link $L$ by $-L$. Note that $-P(p_1,\ldots,p_m)=P(-p_1,\ldots,-p_m)$. For convenience, a sequence $(\dots,\overbrace{x,\ldots,x}^t,\ldots)$ will be denoted by $(\dots,x^{[t]},\ldots)$. 
Finally, we take the convention that a pretzel link \begin{center}
\textit{$P(p_1,\ldots,p_m)$ satisfies $m\ge3$ and $p_i\neq0$ for all $i$.}
\end{center}

\begin{figure}
\centering
\begin{overpic}
    [
width=.9\textwidth]{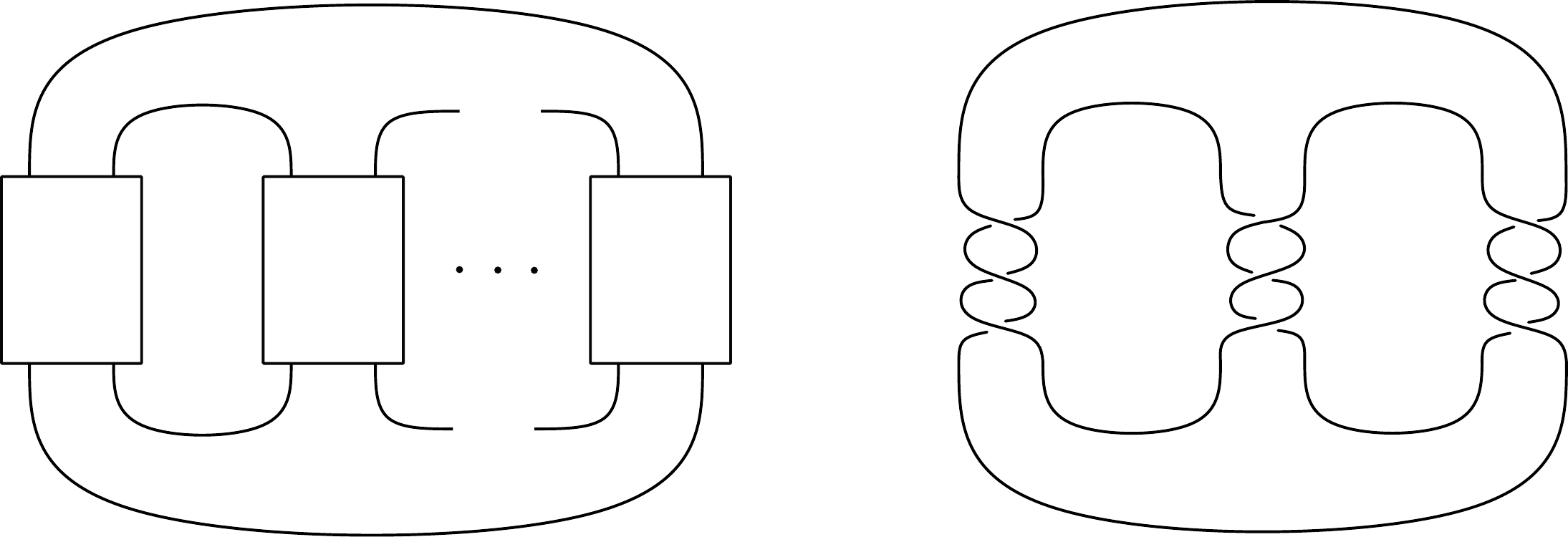}
    \put(3,16){\large$p_1$}
    \put(20,16){\large$p_2$}
    \put(41,16){\large$p_m$}
    \end{overpic}
    \caption{The $m-$stranded pretzel link $P(p_1,\ldots,p_m)$, where the boxes labeled $p_i$ indicate $p_i$ half-twists and the 3-stranded pretzel knot $P(-3,3,-3)$.}
    \label{fig:pretzel}
\end{figure}

\subsection{Positive and negative pretzel links}
A pretzel link $P(p_1,\ldots,p_m)$, where $m\ge3$, is called \textit{positive} (resp. \textit{negative}) if $p_i>0$ (resp. $p_i<0$) for all $1\le i\le m$.
Our first result classifies $\chi-$slice positive pretzel links and $\chi-$slice negative pretzel links. Note that since $P(-p_1,\ldots,-p_m)$ is simply the mirror of $P(p_1,\ldots,p_m)$, $P(-p_1,\ldots,-p_m)$ is $\chi-$slice if and only if $P(p_1,\ldots,p_m)$ is $\chi-$slice.

\begin{restatable}{thm}{positive}
A positive (resp. negative) pretzel link $L$ is $\chi-$slice if and only if $L$ (resp. $-L$) is isotopic to $P(2,2,2,2,2,2,2,2)$ 
or
$P(1^{[z]},2^{[2k]},k+z\pm2)$, where $k,z\ge0$ are integers satisfying $z+2k\ge2$ and $k+z\pm2\ge1$. 
Moreover, $L$ is $\chi-$slice if and only if L is $\chi-$ribbon.
\label{thm:mainpositivepretzellinks}
\end{restatable}

\subsection{3-stranded pretzel links}
Note that a 3-stranded pretzel link $P(p,q,r)$ is a knot if and only if either all parameters are odd or exactly one of the parameters is even. Let 
$$\mathcal{E}=\left\{P\left(a,-(a+2),-\frac{(a+1)^2}{2}\right)\,:\,a \text{ is odd and } a\equiv 1,97\mod{120}\right\}.$$ 

\noindent We first summarize known results from the literature regarding the sliceness of 3-stranded pretzel knots.

\begin{thm}[\cite{liscalensspaces},\cite{greenejabuka},\cite{lecuona},\cite{miller},\cite{kimleesong}] Assume that $K$ is a 3-stranded pretzel knot such that $K,-K\not\in\mathcal{E}$. Then $K$ is slice if and only if $K$ or $-K$ is isotopic to one of the links in the following sets:
\begin{itemize}
    \item $\{P(1,1,4),P(1,-2,-3),P(1,-3,-6)\}$;
    \item $\{P(1,a,-(a+4))\,|\, a\ge1\text{ is odd}\}$; or
    \item $\{P(a,-a,b)\,|\, a\ge1 \text{ is odd and } b\ge1 \}$. 
    \end{itemize}
    Moreover, $K$ is slice if and only if $K$ is ribbon.
\label{thm:3strandedknots}
\end{thm}

\begin{remark}
It is worth mentioning that there is evidence that every pretzel knot in $\mathcal{E}$ is not slice. 
In \cite{lecuona}, Lecuona points out that if $a\le 1597$, then $P(a,-a-2,-\frac{(a+1)^2}{2})$ is not slice (checked via computer).
\end{remark}

Our next result can be thought of as a generalization of Theorem \ref{thm:3strandedknots} to the $\chi-$sliceness of pretzel links. We first define a set that contains the set $\mathcal{E}$. Let

\[\mathcal{F}=\left\{\begin{array}{c c}
    \!\!\multirow{2}{*}{$P(a,-a-x_1^2-x_2^2,-a-y_1^2-y_2^2)\,\,:$ }   & \!\!\Big|\!\det\begin{bmatrix}x_1&y_1\\x_2&y_2\end{bmatrix}\!\Big|\le 4, \,\,\, x_1y_1+x_2y_2=-a, \vspace{5pt}\\ 
   
    & \text{and exactly two parameters are even}
\end{array}\right\}\cup \mathcal{E}.\]
\vspace{.1cm}

\begin{restatable}{thm}{threestrandeda}
Suppose $L$ is 3-stranded pretzel link with nonzero determinant and $L,-L\not\in\mathcal{F}$. Then $L$ is $\chi-$slice if and only if $L$ or $-L$ is isotopic to one of the links in the following sets:
\begin{itemize}
    \item $\{P(1,1,4),P(1,-2,-3),P(1,-3,-6),P(1,-2,-6),P(2,2,3),P(2,2,-5)\}$; 
    \item $\{P(1, a, -(a+4))\,|\,a\ge1\}$; or
    \item $\{P(a,-a, b)\,|\,a,b\geq 1\}$.
\end{itemize}
Moreover, $L$ is $\chi-$slice if and only if $L$ is $\chi-$ribbon.
\label{thm:main3strandeda}
\end{restatable}

\begin{remark}
If $K\in\mathcal{E}$, then $K$ is isotopic to a knot whose parameters are given by the formula in the first set comprising $\mathcal{F}$. In particular, if $a=120k+b$, where $k\in\ZZ$ and $b\in\{1,97\}$, then $x_1=x_2=-1$, $y_1=60k+b$, and $y_2=60k$.
It is worth noting that the obstruction used to obtain the set $\mathcal{F}$ obtains the same formula in the case of $P(p,q,r)$ having exactly one even parameter. However, this list of potentially slice pretzel knots is reduced further in \cite{lecuona},\cite{miller}, and \cite{kimleesong} by using Casson-Gordon invariants and the Fox-Milnor condition. These computations could potentially be carried out in the case when there are two even parameters, however they would only help in obstructing the existence of \textit{orientable} $\chi-$slice surfaces (\cite{donaldowens}, \cite{florens}). 
\end{remark}

\subsection{4-stranded pretzel links}
We recall the following result regarding the sliceness of 4-stranded pretzel knots due to Lecuona. 

\begin{thm}[\cite{lecuona}] 
Suppose $K$ is a 4-stranded pretzel link. If $K$ is slice, then there exist odd integers $a,b\ge1$ such that $K$ or $-K$ is isotopic to a link in the set $P\langle a,-(a+1),b,-b\rangle.$
Moreover, if $K$ or $-K$ is isotopic to $P(a,-(a+1),b,-b)$ for some odd integers $a,b\ge1$, then $K$ is $\chi-$ribbon.
\label{thm:4strandedknots}
\end{thm}

\begin{remark}
    It unknown in general whether the 4-stranded pretzel knots of the form  $P(a,b,-(a+1),-b)$, where $a$ and $b$ are odd, are slice. 
\end{remark}

We are able to extend this result by determining many infinite families that are $\chi-$slice and obstruct ``most" remaining 4-stranded pretzel links from being $\chi-$slice.

\begin{restatable}{thm}{fourstrandedchislice}
Let $a\ge1$ and $b\neq 0$.
The following 4-stranded pretzel links (and their mirrors) are $\chi-$ribbon:

\begin{itemize}
    \item $P(1,1,1,1),P(1,1,1,5),P(2,2,-3,-6),P(1,1,-2,-6)$,
    \item $P(a,2,2,-1),P(a,2,2,-a),P(a,2,2,-(a+4))$,
    \item $P(a,3,1,-a),P(a,3,1,-(a+3)),P(a,1,-2,-2)$,
    \item $P(a,b,-b,-(a+1))$, and $P(a,b,-b,-(a+4)).$
\end{itemize}
\label{thm:main4strandedchislice}
\end{restatable}

\begin{restatable}{thm}{fourstrandednotchislice} If a 4-stranded pretzel link $L$ with nonzero determinant is $\chi$-slice, then $L$ or $-L$ is isotopic to one of the pretzel links listed in Theorem \ref{thm:main4strandedchislice} or to one of the pretzel links in the following sets (where $a,b\ge1$):
\begin{itemize}
    \item $\{P(2,a,2,-a) ,P(2,a,2,-(a+4)), P(a,b,-(a+1),-b),P(a,b,-(a+4),-b)\}$;
    \item $\{P\langle a,b,-b-x^2,-b-y^2\rangle\,:\, xy=-b\text{ and }a\ge\frac{(|x|+|y|-2)b-2}{2}\}$; or
    \item $\{P\langle a,-a-x_1^2-x_2^2-x_3^2,-a-y_1^2-y_2^2-x_3^2,-a-z_1^2-z_2^2-z_3^2\rangle\},$
where $$\Big|\det\begin{bmatrix}
x_1  & y_1 & z_1 \\
     x_2  & y_2 & z_2\\
    x_3 & y_3 & z_3\end{bmatrix}\Big|\le 8$$
and $x_1y_1+x_2y_2+x_3y_3=x_1z_1+x_2z_2+x_3z_3=y_1z_1+y_2z_2+y_3z_3=-a.$
\end{itemize}
\label{thm:main4strandednotchislice}
\end{restatable}

Note that many of the $\chi-$slice pretzel links in Theorem \ref{thm:main4strandedchislice} have the same (unordered) parameters as the potentially $\chi-$slice pretzel links given in the first bullet of Theorem \ref{thm:main4strandednotchislice} (e.g. $P(a,2,2,-a)$ compared to $P(2,a,2,-a)$).
This is due to the obstructions used in the proof of Theorem \ref{thm:main4strandednotchislice}---Donaldson's Diagonalization Theorem \cite{donaldson}, lattice embeddings, and Heegaard Floer $d-$invariant calculations. Our proof strategy largely follows the strategy employed in \cite{greenejabuka}, but adapted to the case of links. In general, these obstructions cannot differentiate between different orderings of the parameters. We are able to differentiate between some of the orderings (e.g. we are able to show that $P(2,-3,2,-6)$ is not $\chi-$slice even though $P(2,2,-3,-6)$ is $\chi-$slice) using geometric arguments (see Section \ref{subsec:geometry}).

\begin{remark}
There are knot invariant obstructions that can tell the difference between different orderings of the parameters, but these would only obstruct the existence of \textit{orientable} $\chi-$slice surfaces. For example, in \cite{donaldowens}, it is shown that if the signature of a link is nonzero, then it cannot bound an \textit{orientable} $\chi-$slice surface; and in \cite{florens}, it is shown that if the Alexander polynomial of an oriented link satisfies the Fox-Milnor condition, then it cannot bound an \textit{orientable} $\chi-$slice surface. 
\label{rem:4stranded}
\end{remark}

\subsection{Seifert fiber spaces bounding rational homology 4-balls}
Following as in \cite{lecuona}, let $Y(p_1,\ldots,p_m)$ denote the double cover of $S^3$ branched over the pretzel link $P(p_1,\ldots,p_m)$. This 3-manifold can be realized as the boundary of the 4-dimensional plumbing shown in Figure \ref{fig:dbc}. In particular, $Y(p_1,\ldots,p_m)$ is a Seifert fiber space.
The next result follows from Proposition \ref{prop:donaldowens} and Theorems \ref{thm:mainpositivepretzellinks}, \ref{thm:main3strandeda}, and \ref{thm:main4strandedchislice}; it provides infinite families of Seifert fiber spaces bounding rational homology 4-balls, some of which are known in the literature (see, e.g., \cite{lecuona}, \cite{lecuonamontesinoslinks}).

\begin{figure}
    \centering
    \begin{tikzpicture}[dot/.style = {circle, fill, minimum size=1pt, inner sep=0pt, outer sep=0pt}]
\tikzstyle{smallnode}=[circle, inner sep=0mm, outer sep=0mm, minimum size=2mm, draw=black, fill=black];

\node[smallnode, label={90:$0$}] (0) at (0,0) {};
\node[smallnode, label={180:$p_1$}] (p1) at (-2,0) {};
\node[smallnode, label={180:$p_2$}] (p2) at (-1.4,-1.4) {};
\node[smallnode, label={0:$p_m$}] (pk) at (2,0) {};
\node (a) at (-.9,-1.1) {};
\node (b) at (1.5,-.2) {};
\draw[-] (0) -- (p1);
\draw[-] (0) -- (p2);
\draw[-] (0) -- (pk);
\draw[loosely dotted, thick] (a) [out=-40, in=-100] to (b);
\end{tikzpicture}
\caption{The double cover of $S^3$ branched along $P(p_1,\ldots,p_m)$}
    \label{fig:dbc}
\end{figure}
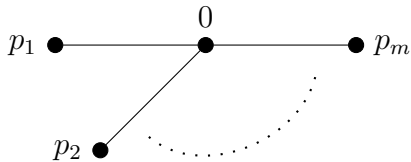

\begin{thm} The following Seifert fiber spaces bound rational homology 4-balls.
\begin{itemize}
\item $Y(1,1,4),Y(1,-2,-3),Y(1,-3,-6),Y(1,-2,-6),Y(2,2,3),Y(2,2,-5)$
\item $Y(1, a, -(a+4))$, where $a\ge1$
\item $Y(a,-a, b)$, where $a,b\geq 1$
\item $Y(1,1,1,1),Y(1,1,1,5),Y(2,2,-3,-6),Y(1,1,-2,-6)$
\item $Y(a,2,2,-1),Y(a,2,2,-a),Y(a,2,2,-(a+4))$, where $a\geq 1$
\item $Y(a,3,1,-a),Y(a,3,1,-(a+3)),Y(a,1,-2,-2)$, where $a\geq 1$
\item $Y(a,b,-b,-(a+1)),Y(a,b,-b,-(a+4))$, where $a\ge1,\, b\neq 0$
\item $Y(2,2,2,2,2,2,2,2)$
\item $Y(1^{[z]},2^{[2k]},k+z\pm2)$, where $k,z\ge0$ such that $z+2k\ge2$ and $k+z\pm2\ge1$
\end{itemize}
\label{thm:seifertspaces}
\end{thm}

\subsection{Organization}
In Section \ref{sec:ribbon} we prove that certain pretzel links are $\chi-$slice via ribbon moves, proving Theorem \ref{thm:main4strandedchislice} and one direction of Theorems \ref{thm:mainpositivepretzellinks}  and \ref{thm:main3strandeda}. In Section \ref{sec:background}, we will recall important facts about pretzel links and provide overviews of the main obstructions used to finish the proofs of Theorems \ref{thm:mainpositivepretzellinks} and \ref{thm:main3strandeda} and to prove Theorem \ref{thm:main4strandednotchislice}. Section \ref{sec:positive} contains the proof of Theorem \ref{thm:mainpositivepretzellinks}, Section \ref{sec:3} contains the proof of Theorem \ref{thm:main3strandeda}, and Section \ref{sec:4} contains the proof of Theorem 
\ref{thm:main4strandednotchislice}. 

\subsection{Acknowledgments} This work originated in an REU at Georgia Tech. Both the REU and HT were supported by NSF DMS 1745583.

\pagebreak


\section{Ribbon Moves}\label{sec:ribbon}

In this section we show that several families of pretzel links are $\chi-$ribbon (and hence $\chi-$slice) by finding ribbon moves in their diagrams to explicitly describe a $\chi-$ribbon surface for the the link. In particular we prove Theorem \ref{thm:main4strandedchislice} and one direction of Theorem \ref{thm:main3strandeda}. We first review the classification of $\chi-$slice 2-bridge links due to Lisca in \cite{liscalensspaces}, which will be used in the aforementioned proofs.

\subsection{$\chi-$slice 2-bridge links}\label{sec:2bridge} In \cite{liscalensspaces}, Lisca classified $\chi-$slice 2-bridge links (although the term ``$\chi-$slice" was not coined until after Lisca's findings). We will rely on this classification for proving the $\chi-$sliceness of particular families of pretzel links.

For integers $x_1,\ldots,x_y\neq0$, we define the following two continued fraction expansions:
$$[x_1,\ldots,x_y]^+=x_1+\frac{1}{\displaystyle x_2+\frac{1}{\displaystyle\cdots+\frac{1}{x_y}}}\,\,\,\,\text{ and }\,\,\,\,[x_1,\ldots,x_y]^-=x_1-\frac{1}{\displaystyle x_2-\frac{1}{\displaystyle\cdots-\frac{1}{x_y}}}.$$
Up to mirroring, every 2-bridge link is isotopic to a link of the form $C(d_1,\ldots,d_m)$ shown in Figure \ref{fig:2bridge}, where $d_i\neq 0$ for all $i$. Moreover, it is known that every 2-bridge link is isotopic to a link of the form $C(d_1,\ldots,d_m)$ with $d_i>0$ for all $i$; see e.g \cite[Chapter 2]{kaw-survey}. Given a such 2-bridge link $C(d_1,\ldots,d_m)$, we call the reduced fraction $\frac{p}{q}=[d_1,\ldots,d_m]^+$ the \textit{classifying fraction} of $C(d_1,\ldots,d_m)$ and denote $C(d_1,\ldots,d_m)$ with the simpler notation $K(p,q)$. It follows that any 2-bridge link is isotopic to $K(p,q)$ for some $p>q>0$. Finally, if $C(e_1,\ldots,e_k)$ (with $e_i\neq 0$ arbitrary) satisfies $[e_1,\ldots,e_k]^+=\frac{p}{q}$, then $C(e_1,\ldots,e_k)$ is isotopic to $K(p,q)$.

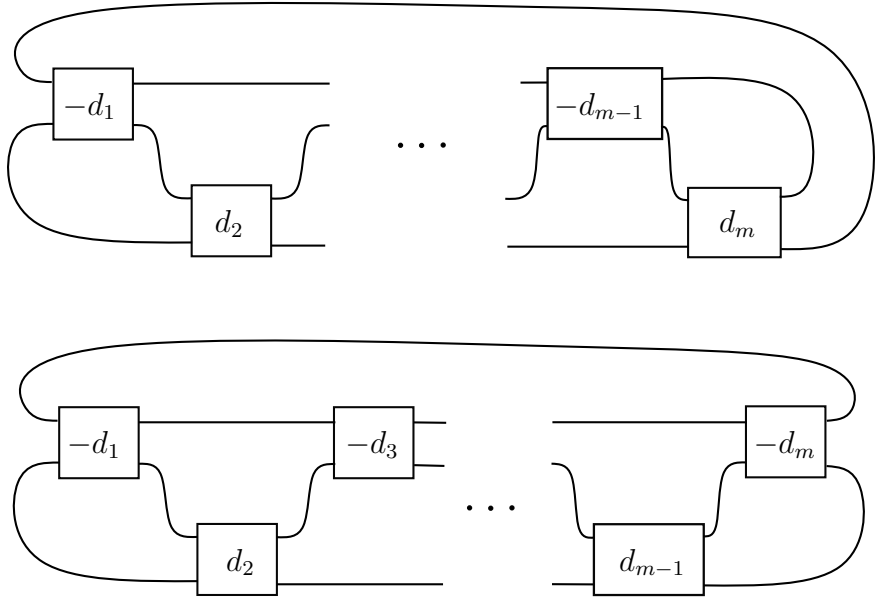
\begin{figure}[h]
\begin{overpic}
[
width=.7\textwidth]{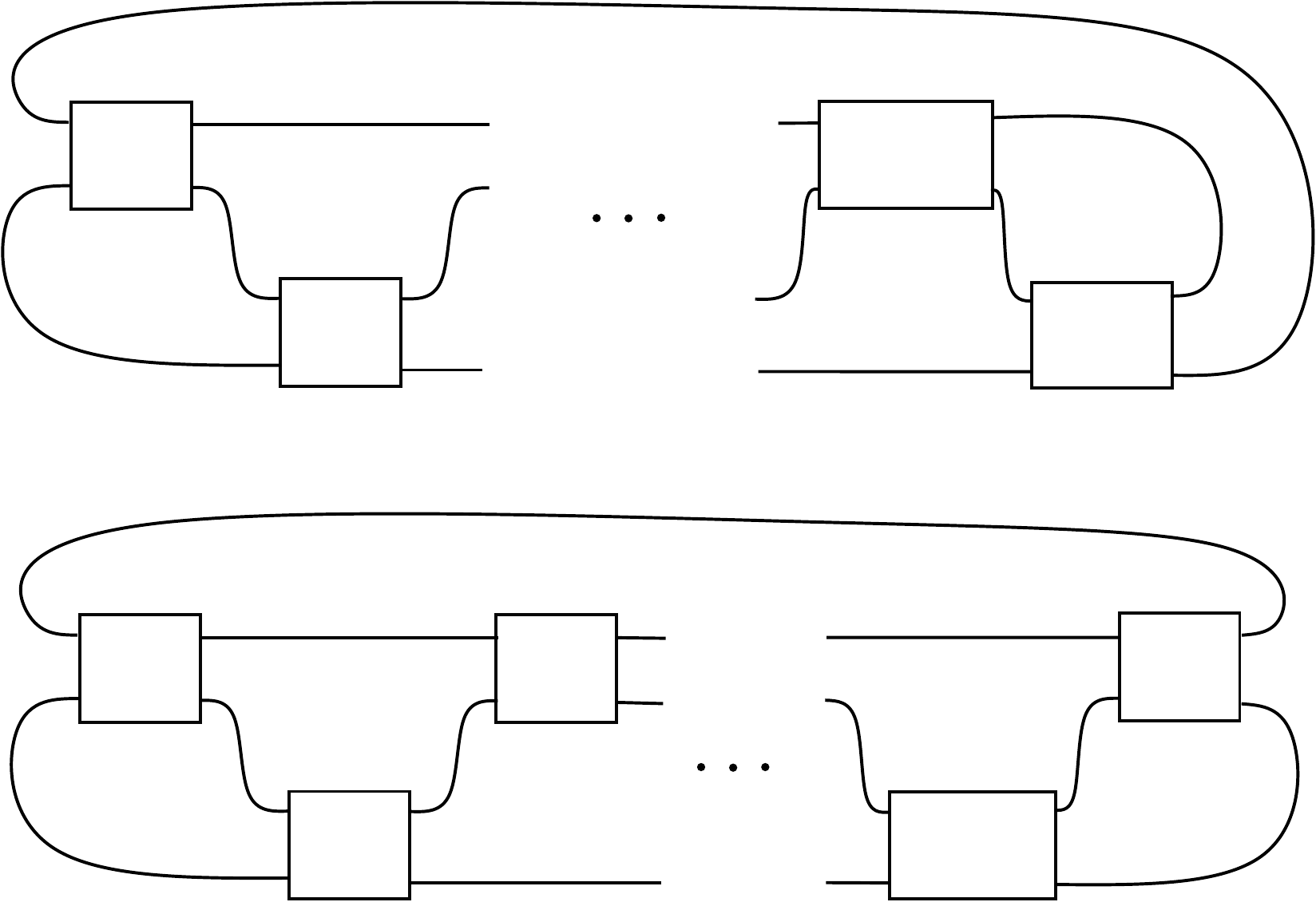}
\put(7,16.5){$-d_1$}
\put(25.3,3){$d_2$}
\put(39, 16.5){$-d_3$}
\put(70.5,3){$d_{m-1}$}
\put(86,16.5){$-d_{m}$}
\put(6.5,55.5){$-d_1$}
\put(24,42){$d_2$}
\put(63,55.5){$-d_{m-1}$}
\put(82,42){$d_{m}$}
\end{overpic}
\caption{The two-bridge link $C(d_1,\ldots,d_m)$ is indicated in the case that $m$ is even (above) and when $m$ is odd (below).}\label{fig:2bridge}
\end{figure}

\pagebreak

Given relatively prime integers $p>q>0$, 
there exist unique integers $a_1,\ldots,a_n\ge2$ such that $\frac{p}{q}=[a_1,\ldots,a_n]^-$.
Let $(b_1,\ldots,b_k)$ be an ordered string of integers such that $b_i\ge2$ for all $i$. If $b_i\ge 3$ for some $i$, then we can write this string in the form 
$$(2^{[m_1]},3+n_1,2^{[m_2]},3+n_2,\ldots,2^{[m_{j-1}]},3+n_{j-1},2^{[m_j]},2+n_j),$$
where $m_i,n_i\ge 0$ for all $i$ and $(\dots,x^{[t]},\ldots)$ denotes $(\dots,\overbrace{x,\ldots,x}^t,\ldots)$. The string $(c_1,\ldots,c_l)=(2+m_1, 2^{[n_1]},3+m_2, 2^{[n_2]},\ldots,3+m_{j-1}, 2^{[n_{j-1}]},3+m_j,2^{[n_j]})$ is called the \textit{dual string} of $(b_1,\ldots,b_k)$. If $b_i=2$ for all $1\le i\le k$, then we define its dual string to be $(k+1)$. Finally, we say $(c_n,\ldots,c_1)$ is the \textit{reverse} of $(c_1,\ldots,c_n)$.

\begin{thm}[Lemmas 7.1-7.3 in \cite{liscalensspaces} and Remark 3.2 in \cite{lecuonamontesinosknots}] Let $\frac{p}{q}=[a_1,\ldots,a_n]^-$. $K(p,q)$ is $\chi-$slice \emph{(}and, indeed, $\chi-$ribbon\emph{)} if and only if, up to reversal, $(a_1,\ldots,a_n)$ or its dual is of one of the following forms:
\begin{enumerate}[(a)]
\item $(b_1,\ldots,b_k,2,c_l,\ldots,c_1)$, where $(b_1,\ldots,b_k)$ and $(c_1,\ldots,c_l)$ are dual;
\item $(b_1,\ldots ,b_{k-1},b_k+1,2,2,c_l+1,c_{l-1},\ldots,c_1)$, where $(b_1,\ldots,b_k)$ and $(c_1,\ldots,c_l)$ are dual;
\item $(2^{[x]}, 3, 2+y, 2+x, 3, 2^{[y]})$; 
\item $(2^{[x]}, 3+y,2,2+x,3,2^{[y]})$; 
\item $(2+x,2+y,3,2^{[x]},4, 2^{[y]})$; 
\item $(2+x,2,3+y,2^{[x]},4,2^{[y]})$; or 
\item $(3+x,2,3+y,3,2^{[x]},3,2^{[y]})$.
\end{enumerate}
\label{thm:2bridgechislice}
\end{thm}

\begin{cor}
The torus link $T(2,a)$, where $a\neq0$, is $\chi-$slice \emph{(}and $\chi-$ribbon\emph{)} if and only if $a\in\{\pm1,\pm4\}$.
\label{cor:torus}
\end{cor}

\begin{proof} The torus knots $T(2,1)$ and $T(2,-1)$ are isotopic to the unknot, so they are slice. Assume $a\neq \pm1$. The torus link $T(2,a)$ is isotopic to the 2-bridge link $C(a)$. Suppose $a>1$. Then $C(a)$ has classifying fraction $\frac{a}{1}=[a]^-$. Notice that the string $(a)$ has dual $(2^{[a-1]}).$ The former string does not appear in the list of strings in Theorem \ref{thm:2bridgechislice}. The latter string, however, is in the list of strings if and only if $a=4$. Now since the mirror of $T(2,a)=C(a)$ is $T(2,-a)=C(-a)$, the result follows.    
\end{proof}

 We now use the classification of $\chi-$slice 2-bridge link to classify the $\chi-$sliceness of particular 2-bridge pretzel links.

\begin{lem}
$P(1^{[z]},a)$, where $a\ge1$, is $\chi-$slice (and, indeed, $\chi-$ribbon) if and only if $a=z\pm2$, where $z\pm2\ge1$.
\label{lem:2-bridgechi1}
\end{lem}

\begin{proof}
First note that a pretzel link of the form $P(1^{[z]},a)$ is the 2-bridge link $C(z,a)=K(az+1,a)$. 
It is easy to verify that $[z+1,2^{[a-1]}]^-=\frac{az+1}{a}$. 
Consider the list of strings in Theorem \ref{thm:2bridgechislice} and note that the dual of $(z+1,2^{[a-1]})$ is $(2^{[z-1]},a+1)$. Since, up to reversal, these strings are not of the form $(b)-(g)$, it follows that $(z+1,2^{[a-1]})$ or $(2^{[z-1]},a+1)$ is of type $(a)$. The string
$(z+1,2^{[a-1]})$ is of type $(a)$ if and only if $(z+1)$ and $(2^{[a-2]})$ are dual strings; that is, $z+2=a$. On the other hand, $(2^{[z-1]},a+1)$ is of type $(a)$ if and only if $(2^{[z-2]})$ and $(a+1)$ are dual strings; that is, $z-2=a$.
\end{proof}

The next result was noted (without proof) in Remark 1.3 in \cite{greenejabuka}. For completeness, we will prove it here.

\begin{lem}
Suppose $L$ is a 3-stranded pretzel link containing a parameter equal to 1 and with nonzero determinant. Then $L$ is $\chi-$slice (and, indeed, $\chi-$ribbon) if and only if $L$ or $-L$ is isotopic to one of the links in the following sets:
\begin{itemize}
    \item $\{P(1,-3,-6),P(1,-2,-6),P(1,-2,-3),P(1,1,4)\}$ 
    \item $\{P(1,-1,a)\,|\,a\ge1\}$
    \item $\{P(1,a,-a)\,|\,a\ge1\}$
    \item $\{P(1,a,-(a+4))\,|\, a\ge1\}$
\end{itemize}
\label{lem:2-bridgechi2}
\end{lem}

\begin{proof}
Let $L=P(1,a,b)$ be a 3-stranded pretzel link with nonzero determinant.
Note that if $a=-1$ or $b=-1$, then $L$ is isotopic to $P(1,-1,a)$, which is unknotted; this are the links in the second bullet point. If $a=1$ or $b=1$, then $L$ is a twist knot. The only slice twist knots are the unknot and Stevedore's knot, which occur when $b=0$ (which is not possible, by assumption) and $b\in\{4,-5\}$, respectively; hence $L$ is $\chi-$slice if and only if $L$ is isotopic to $P(1,1,-1)$, $P(1,1,-5)$, or $P(1,1,4)$. The first two links belong to the last two bullet points listed in the statement of the lemma; the last link belongs to the first bullet point.

We now assume that  $a,b\not\in\{-1,1\}$.
It is easy to see that $L$ is isotopic to the two-bridge link $C(a+1,b+1)=K(-ab-a-b,b+1)$; see Figure \ref{fig:2-bridge}. 

\begin{figure}[h!]
\begin{overpic}
[
width=.9\textwidth]{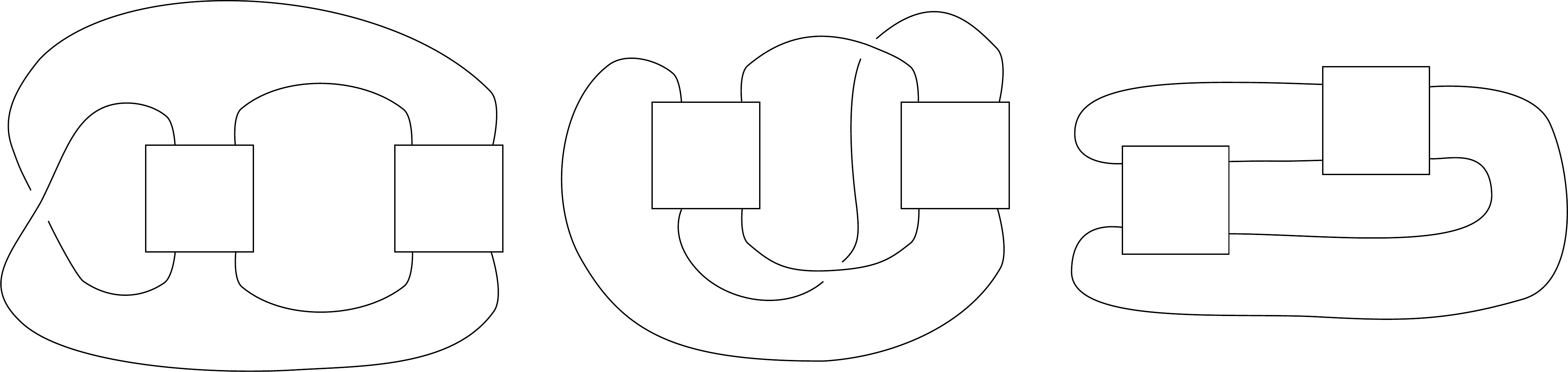}
\put(12.1,10){$a$}
\put(28,10){$b$}
\put(44.5,13){$a$}
\put(60.5,12.5){$b$}
\put(71.75,10){$a+1$}
\put(84.75,15.5){$b+1$}

\end{overpic}
    \caption{The pretzel link $P(1,a,b)$ is isotopic to the 2-bridge link $C(a+1,b+1)=K(-ab-a-b,b+1)$.}
    \label{fig:2-bridge}
\end{figure}
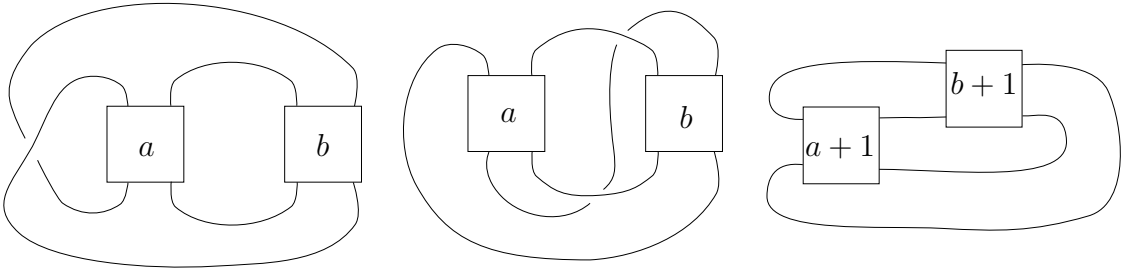

First suppose $a$ and $b$ have opposite sign. Up to mirroring and isotoping, we may assume that $a<-1$ and $b>1$. Then $\frac{-ab-a-b}{b+1}=[-a,2^{[b]}]^-$. Following as in the proof of Lemma \ref{lem:2-bridgechi1}, it is clear that $L$ is $\chi-$slice if and only if either $-a=b$ or $-a=b+4$, which yield the links in the third and fourth bullet points in the statement of the lemma.

Next assume $a$ and $b$ have the same sign. Up to mirroring an isotoping, we may assume that $a,b<-1$. First suppose $a=-2$ or $b=-2$; up to isotopy, we may assume the latter. Since $\det L\neq 0$, it follows that $a<-2$. If $a=-3$, then it is easy to see that $L$ is the unknot; hence $P(1,-3,-2)=P(1,-2,-3)$ is $\chi-$slice. Assume $a<-3$. Then $\frac{-ab-a-b}{b+1}=-(a+2)=[-(a+2)]^-$ and $(-(a+2))$ has dual $(2^{[-(a+3)]})$. Following as in the proof of Lemma \ref{lem:2-bridgechi1}, it is clear that $L$ is $\chi-$slice if and only if $a=-6$; that is, $L$ is isotopic to $P(1,-6,-2)=P(1,-2,-6)$. Finally suppose $a,b<-2$. Then $\frac{-ab-a-b}{b+1}=[-(a+1),-(b+1)]^-$ and $(-(a+1),-(b+1))$ has dual $(2^{[-(a+3)]},3,2^{[-(b+3)]})$. Arguing as above, it is clear that $L$ if $\chi-$slice if and only if $a=-6$ and $b=-3$ (or vice versa); that is, $L$ is isotopic to $P(1,-3,-6)$.
\end{proof}

\begin{figure}[h!]
\begin{overpic}
[
width=1\textwidth]{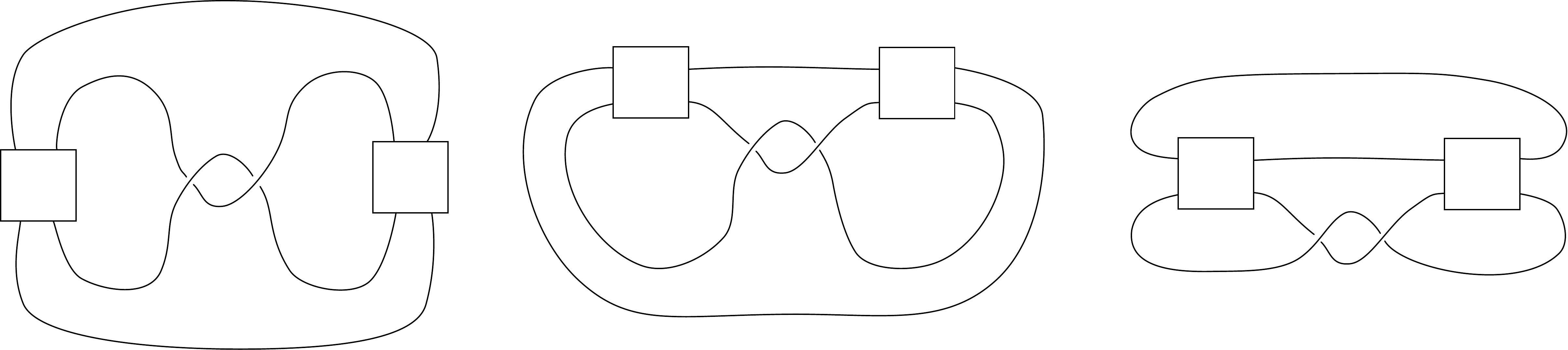}
\put(1.7,10){$b$}
\put(25.5,10){$a$}
\put(41,16.5){$b$}
\put(58,16.5){$a$}
\put(77,10.5){$b$}
\put(94,10.5){$a$}

\end{overpic}
    \caption{The pretzel link $P(a,1,1,b)$ on the left, is isotopic to the 2-bridge link $C(-a,-2,-b)$ on the left.}
    \label{fig:4-strand-2-bridge}
\end{figure}
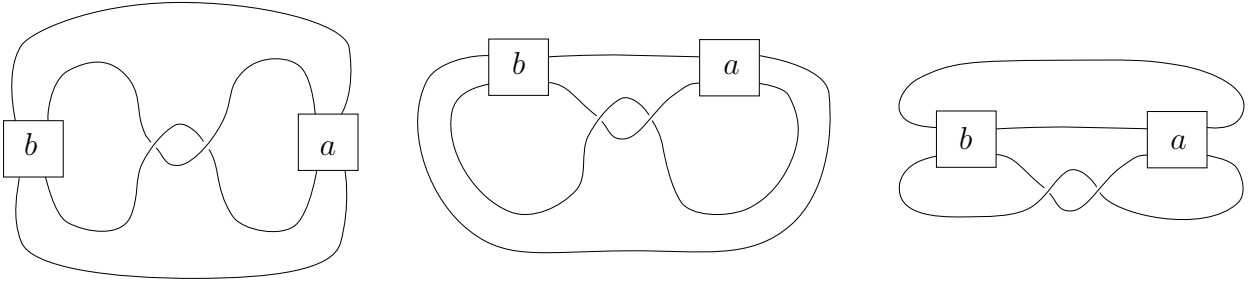

\begin{lem}
    $P(1,1,a,b)$, where $a<0$ and $b>0$, is $\chi-$slice (and $\chi-$ribbon) if and only if $(a,b)\in\{(-1,3), (3,-4)\}$.
    \label{lem:4strands}
\end{lem}

\begin{proof}
First suppose $b=1$. Then it is easy to see that $P(1,1,a,1)=P(1,1,1,a)$ is the 2-bridge link $C(3,a)=K(3a+1,a)$ and that $[3,-a]^-=\frac{3a+1}{a}$. Notice that the string $(3,-a)$ has dual $(2,3,2^{[-a-2]})$. We can see that neither of these strings fall under Theorem \ref{thm:2bridgechislice}; hence $P(1,1,1,a)$ is not $\chi$-slice for all $a\le-1$. 

Next assume $b\ge2$. Figure \ref{fig:4-strand-2-bridge} shows that $P(1,1,a,b)=P(b,1,1,a)$ is the 2-bridge link $L=C(-b,-2,-a)$. First suppose $a=-1$; then it is easy to see that $L$ is isotopic to the torus link $T(2,b+1)$. By Corollary \ref{cor:torus}, $L$ is $\chi-$slice (and, indeed $\chi-$ribbon) if and only if $b=3$. Next suppose $a\le -2$. and consider $-L=C(b,2,a)=K(2ab+a+b,2a+1)$.
It is routine to check that $[b+1,3,2^{[-a-2]}]^-=\frac{2ab+a+b}{2a+1}$. Note that the dual of $(b+1,3,2^{[-a-2]})$ is $(2^{[b-1]},3,-a)$.
Consider the list of strings in Theorem \ref{thm:2bridgechislice}. Up to reversal, neither $(b+1,3,2^{[-a-2]})$ nor $(2^{[b-1]},3,-a)$ are of the form $(a)-(d)$ or $(f)$. They are of the form $(e)$ and $(f)$ precisely when $x=y=0$, yielding $(a,b)=(-4,3)$.
\end{proof}

\subsection{$\chi-$slice pretzel links}

We now describe the standard method for showing that a link is $\chi-$slice. A ribbon (or band) move is an operation on a link diagram illustrated in Figure \ref{fig:ribbonmovedefinition}. Given a ribbon (or band), an embedding  $f:I\times I\to S^3-L$ with $f(\{0,1\}\times I)\subset L$, a new link is formed as $L'=(L-f(\{0,1\}\times I))\cup f(I\times \{0,1\}).$ We say that $L'$ is obtained from $L$ by a ribbon (or band) move. It is an elementary argument to show the following proposition. 

\begin{figure}[h!]
  \begin{overpic}
[
width=.8\textwidth]{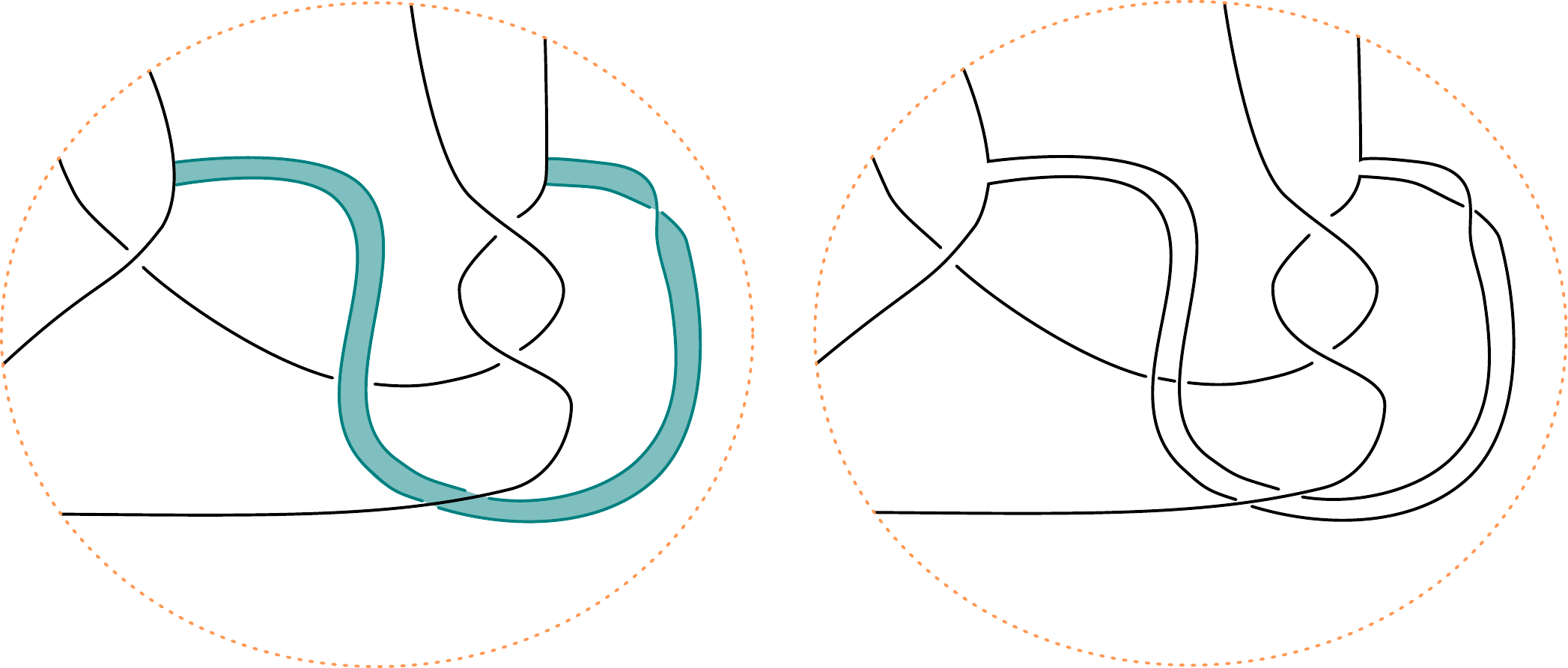}
\put(16,34){\textcolor{specialteal}{$B$}}
\put(13,11){$L$}
\put(65, 11){$L'$}
\end{overpic}
    \caption{Obtaining a new link $L'$ from $L$ via a ribbon (band) move using the band $B$.}
    \label{fig:ribbonmovedefinition}
\end{figure}

\begin{prop}
If $L$ admits a sequence of $n$ ribbon moves which convert $L$ to the split union of an $n$-component unlink and a $\chi-$slice link $L'$, then $L$ is $\chi-$slice. Moreover, if $L'$ is $\chi-$ribbon, then $L$ is also $\chi-$ribbon.
\label{prop:ribbonmoves}
\end{prop}

We first prove that the positive pretzel links listed in Theorem \ref{thm:mainpositivepretzellinks} are indeed $\chi-$ribbon.

\begin{prop}
The positive pretzel links $P(2,2,2,2,2,2,2,2)$ and $P(1^{[z]},2^{[2k]},k+z\pm2)$ (and their mirrors) are $\chi-$ribbon.
\label{prop:pospretzelslice}
\end{prop}

\begin{proof}
    Performing $k$ band moves to the pretzel link $P(1^{[z]},2^{[2k]},k+z\pm2)$, as indicated in Figure \ref{fig:P(2,2,...,2,k)}, yields the split union of the unlink of $k$ components and the pretzel link $P(1^{[k+z]},k+z\pm2)$, which is $\chi-$ribbon by Lemma \ref{lem:2-bridgechi1}.
    Similarly, performing four band moves to $P(2,2,2,2,2,2,2)$, as indicated in Figure \ref{fig:P(2,2,...,2,k)} yields the split union of the unlink of four components and the torus link $T(2,-4)$, which is $\chi-$ribbon by Corollary \ref{cor:torus}. 
    The result now follows from Proposition \ref{prop:ribbonmoves}.
\end{proof}

\begin{figure}
	\vspace{.2cm}
   \begin{overpic}
    [
width=.6\textwidth]{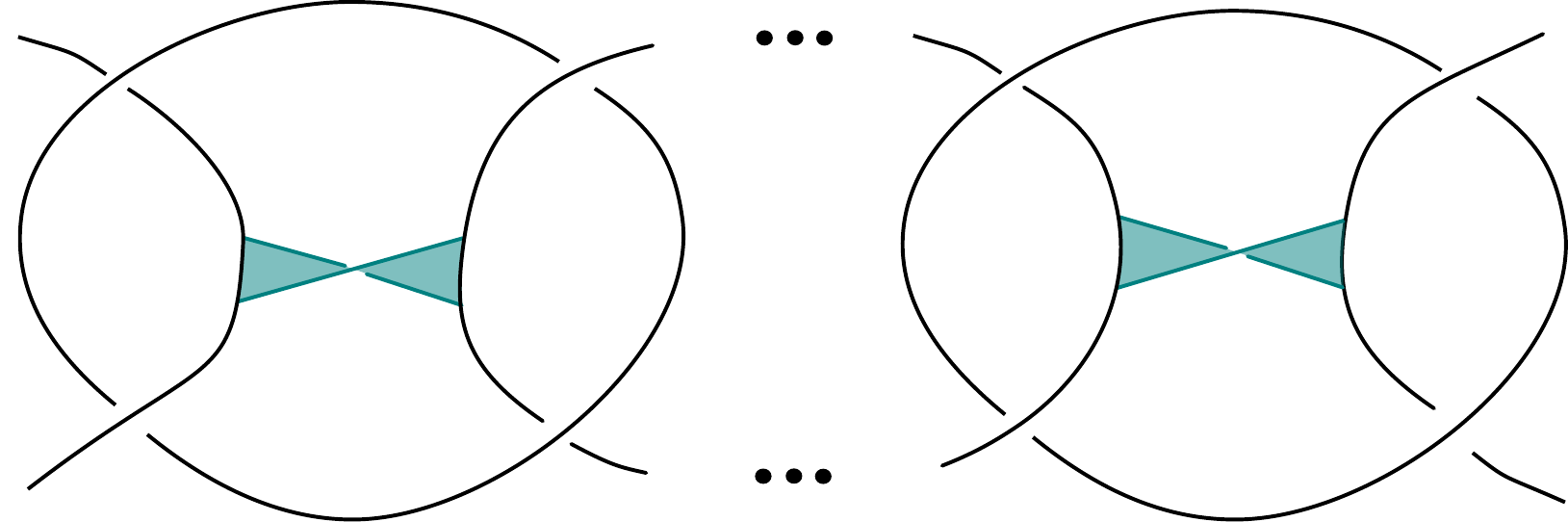}

\end{overpic}
    \caption{A tangle in the pretzel link $P(a_1,\ldots,a_m,2^{[2k]},b_1,\ldots,b_n)$. Applying the indicated $k$ band moves yields the split union of the $k$-component unlink and the pretzel link $P(a_1,\ldots,a_m,1^{[k]},b_1,\ldots,b_n)$.}
    \label{fig:P(2,2,...,2,k)}
\end{figure}

We next show that the 3-stranded pretzel links listed in Theorem \ref{thm:main3strandeda} are $\chi-$ribbon.

\begin{prop}
The following 3-stranded pretzel links (and their mirrors) are $\chi-$ribbon.
\begin{itemize}
    \item $\{P(1,1,4),P(1,-2,-3) ,P(1,-3,-6),P(1,-2,-6),P(2,2,3),P(2,2,-5)\}$ 
    \item $\{P(1, a, -(a+4))\,|\,a\ge1\}$
    \item $\{P(a,-b, b)\,|\,a,b\neq0\}$
\end{itemize}
\label{prop:3-strand}
\end{prop}

\begin{proof}
The links with at least one parameter equal to 1 are $\chi-$ribbon by Lemma \ref{lem:2-bridgechi2}. Moreover, they were shown to admit ribbon moves to the unlink in \cite{liscalensspaces}. It remains to show that links $P(2,2,3)$, $P(2,2,-5)$, and $P(a,-b,b)$, where $a,b\ge1$ are $\chi-$ribbon. The first two links admit ribbon moves to the links $T(2,4)$ and $T(2,-4)$, which are $\chi-$ribbon by Corollary \ref{cor:torus}; see Figure \ref{fig:3-strandribbonmove:b}. Hence, they are $\chi-$ribbon by Proposition \ref{prop:ribbonmoves}.
Applying the ribbon move indicated in Figure \ref{fig:3-strandribbonmove:a} converts $P(a,-b,b)$ to an unlink of two components; hence it is $\chi-$ribbon.
\end{proof}

\begin{figure}[h]
    \centering    
    
  \subfloat[Obtaining the 2-component unlink from $P(a,-b,b)$. \label{fig:3-strandribbonmove:a}]{%
      \begin{overpic}[
      height=0.2 \textheight]{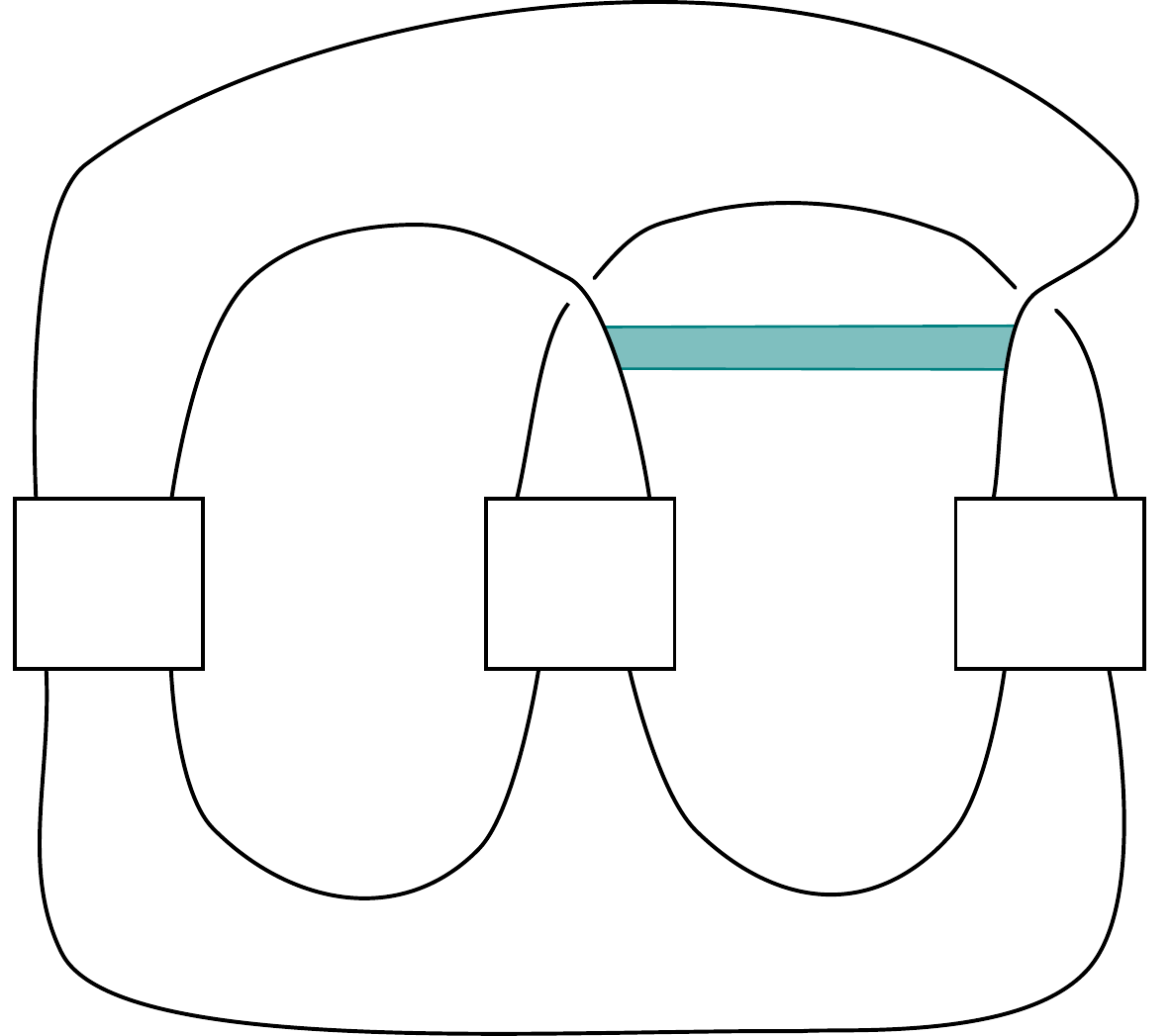}
        \put(7.7,36.5){$a$}
        \put(43,36.5){$-b$}
        \put(87,36.5){$b$}
      \end{overpic}
         }
    \quad\quad
     \subfloat[Obtaining the split union of an unknot and $T(2,a+1)$ from $P(2,2,a)$. \label{fig:3-strandribbonmove:b}]{%
      \begin{overpic}[
      height=0.2 \textheight]{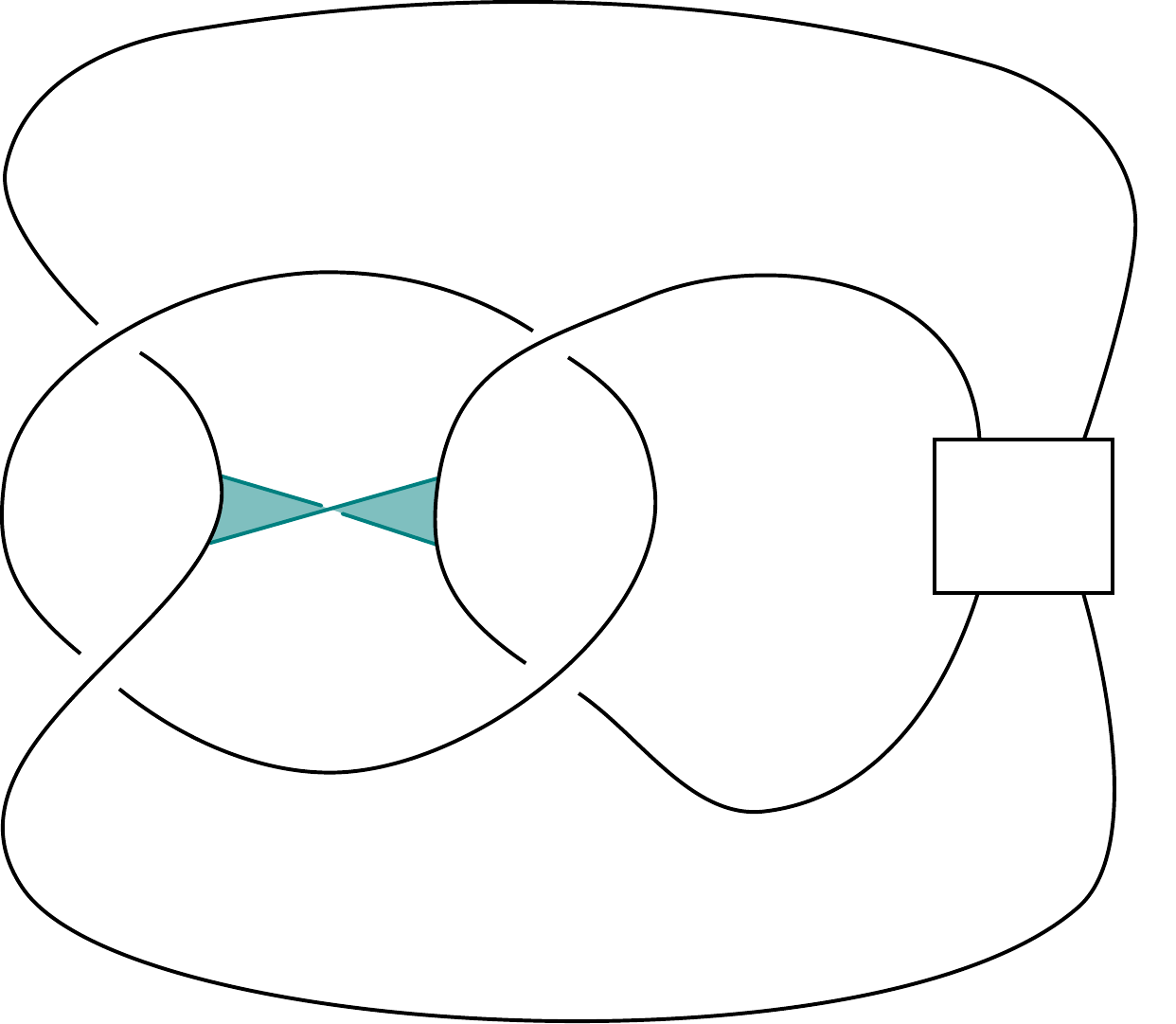}
        \put(85,40.5){$a$}
      \end{overpic}
  }

    \caption{Applying ribbon moves on families of 3-strand pretzel links to yield simpler links.}
    \label{fig:3-strandribbonmove}

\end{figure}

Finally, we prove Theorem \ref{thm:main4strandedchislice}, which we recall here for convenience.

\pagebreak

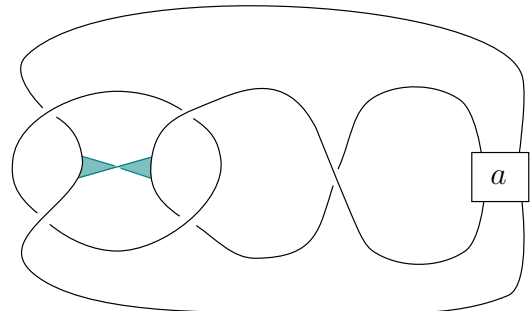
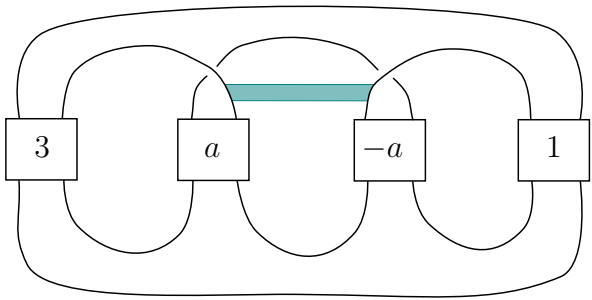
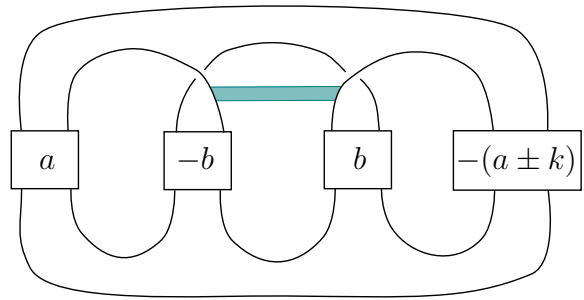
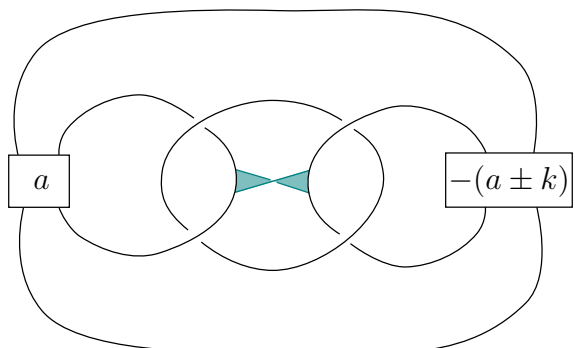
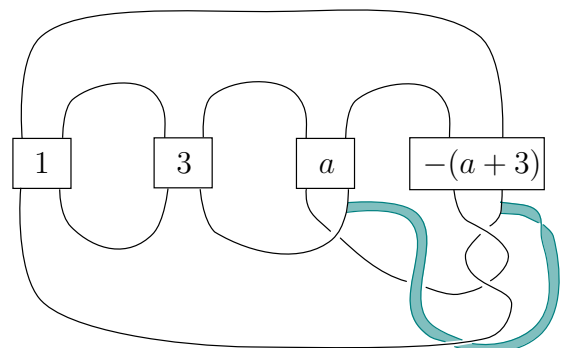
\begin{figure}[h!]
    \centering
\vspace{40pt}
    \subfloat[Subfigure A][Obtaining the split union of unknot and $P(1,-3,-6)$ from $P(2,2,-3,-6)$. \label{fig:ribbonmove:a}]{%
      \begin{overpic}[
      height=0.2 \textheight]{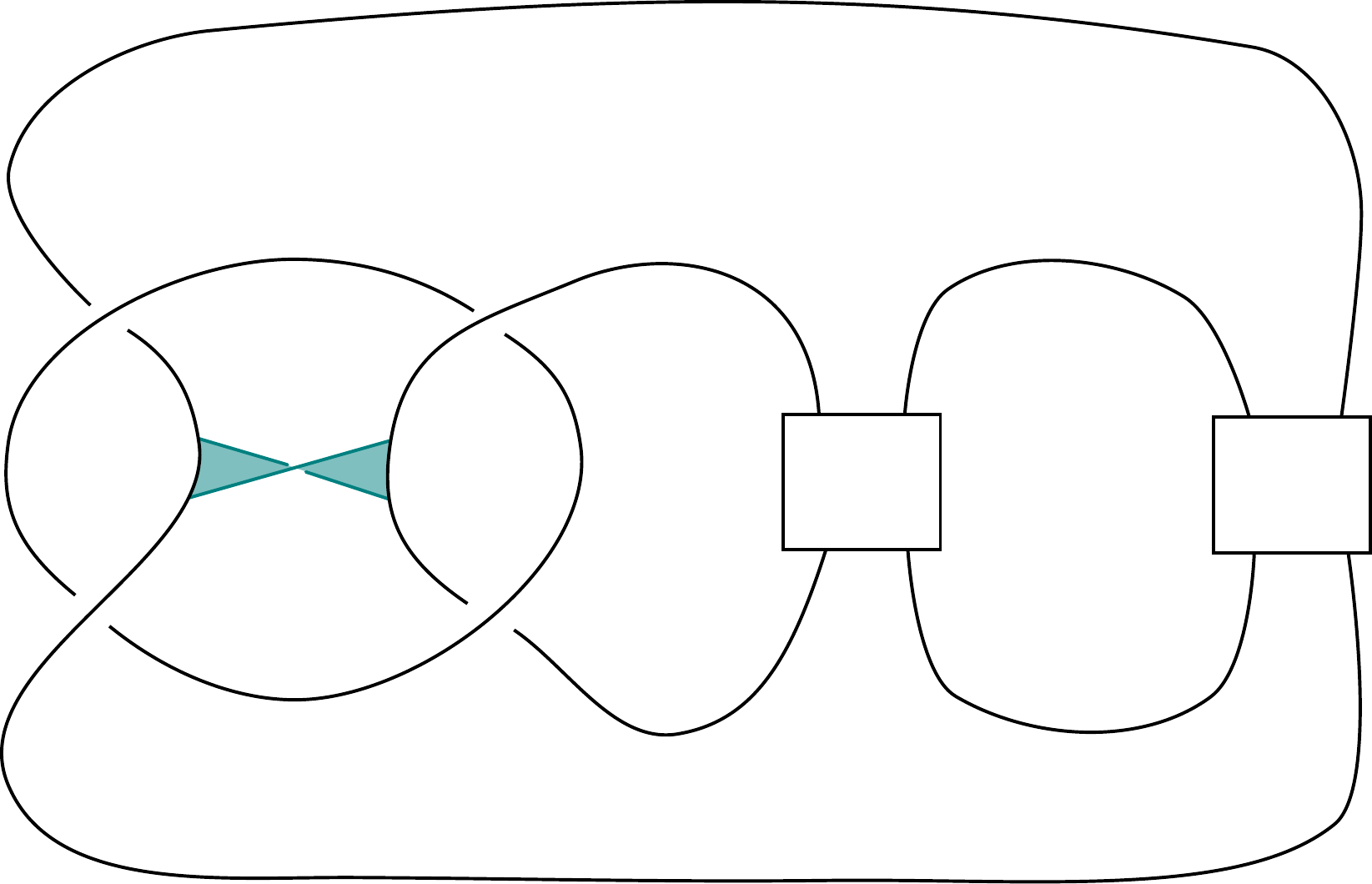}
        \put(59.5,27){$-3$}
        \put(90,27){$-6$}
      \end{overpic}
  }
    \quad
    \subfloat[Subfigure B][Obtaining the two component unlink from $P(2,2,-1,a)$.\label{fig:ribbonmove:b}]{%
      \begin{overpic}[
      height=0.18 \textheight]{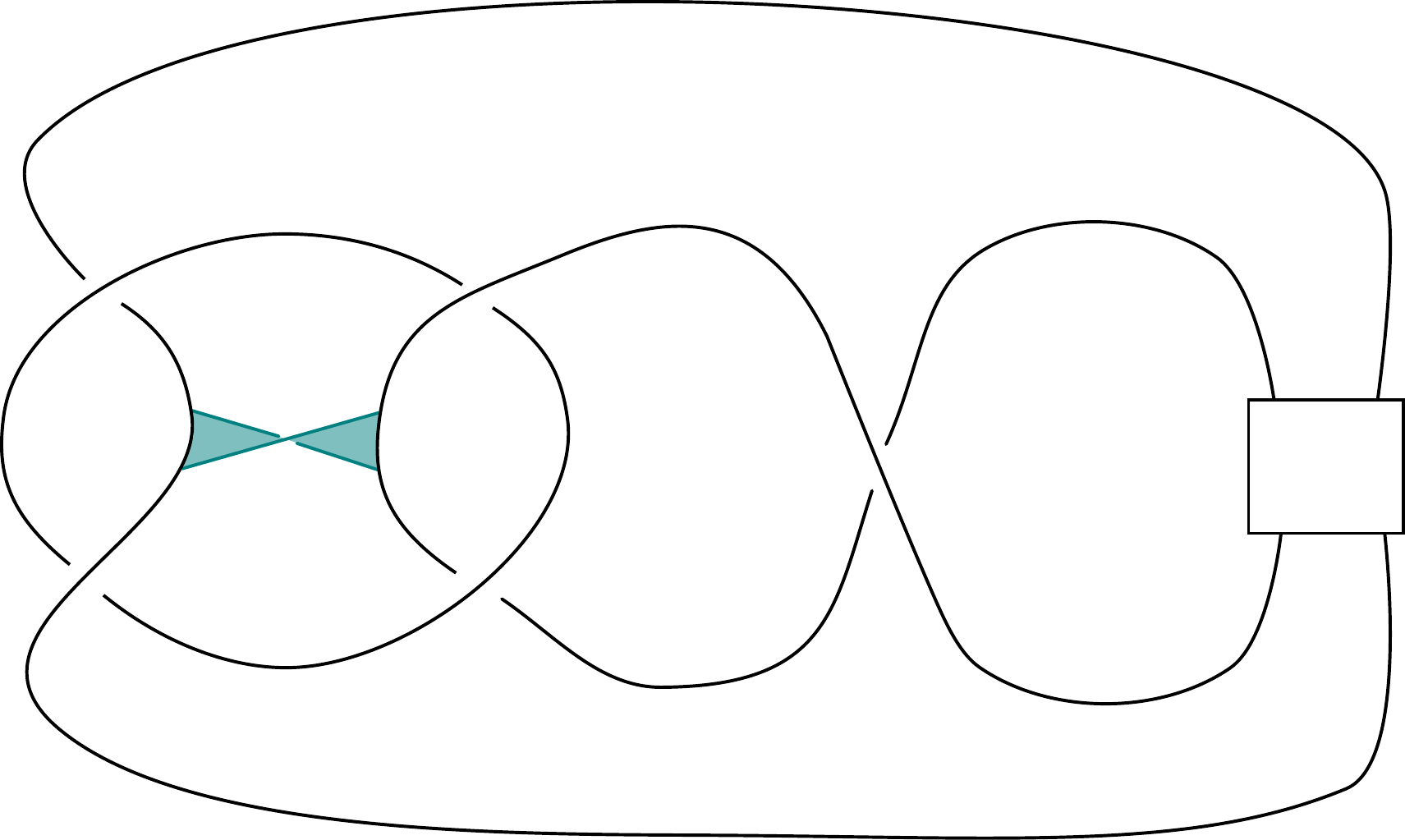}
        \put(92.5,25.2){$a$}
      \end{overpic}
  }
     \vspace{30pt}
    \qquad
    \subfloat[Subfigure C][Obtaining the split union of an unknot and $T(2,4)$ from $P(3,-a,a,1)$.\label{fig:ribbonmove:c}]{%
      \begin{overpic}[
      height=0.17 \textheight]{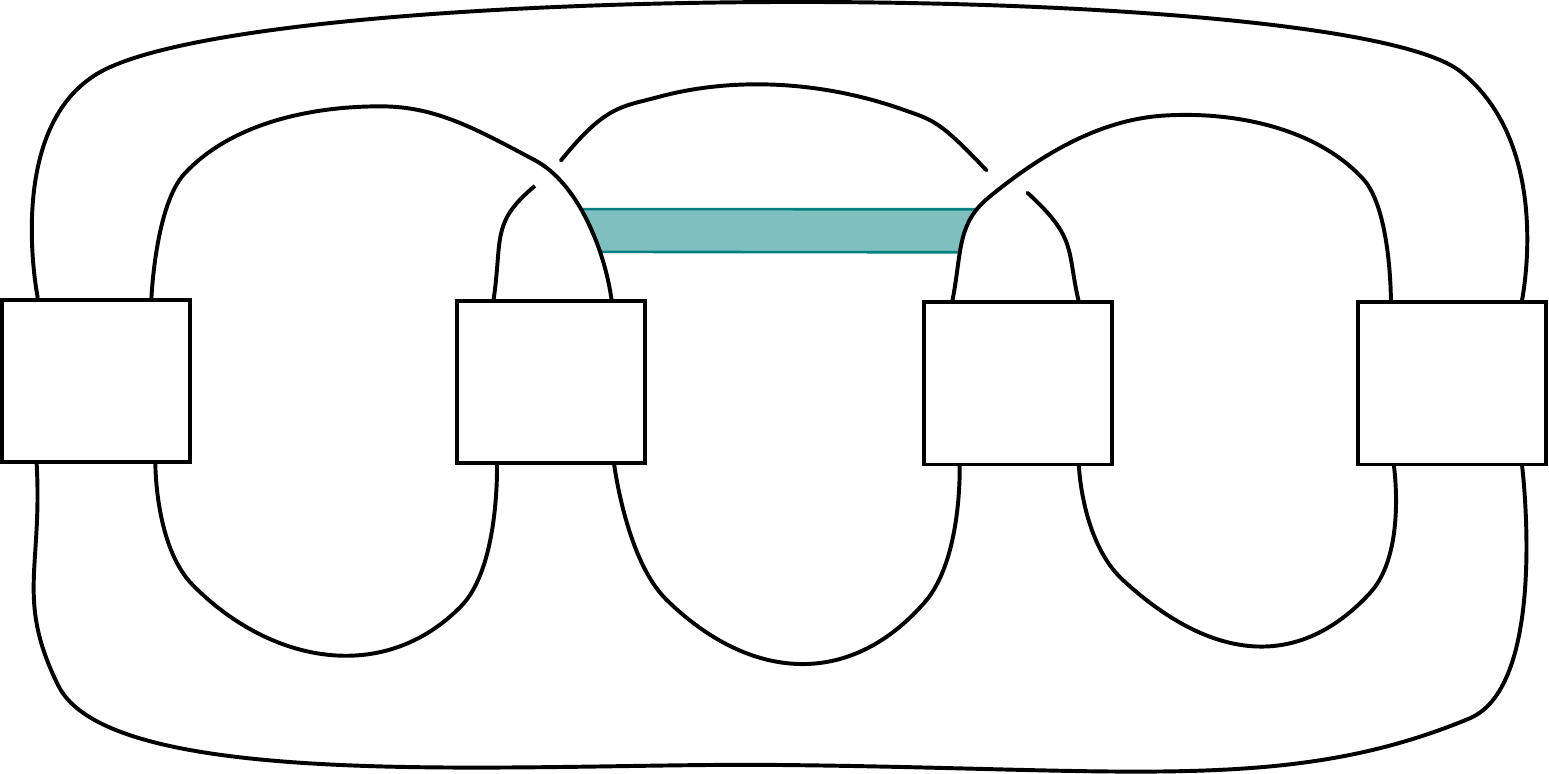}
        \put(34,24){$a$}
        \put(61,24){$-a$}
        \put(5,24){$3$}
        \put(92.5,24){$1$}
      \end{overpic}
    }
    \quad
    \subfloat[Subfigure D][Obtaining the split union of an unknot and $T(2,\pm k)$ from $P(a,-b,b,-(a\pm k))$.\label{fig:ribbonmove:d}]{%
      \begin{overpic}[
      height=0.17 \textheight]{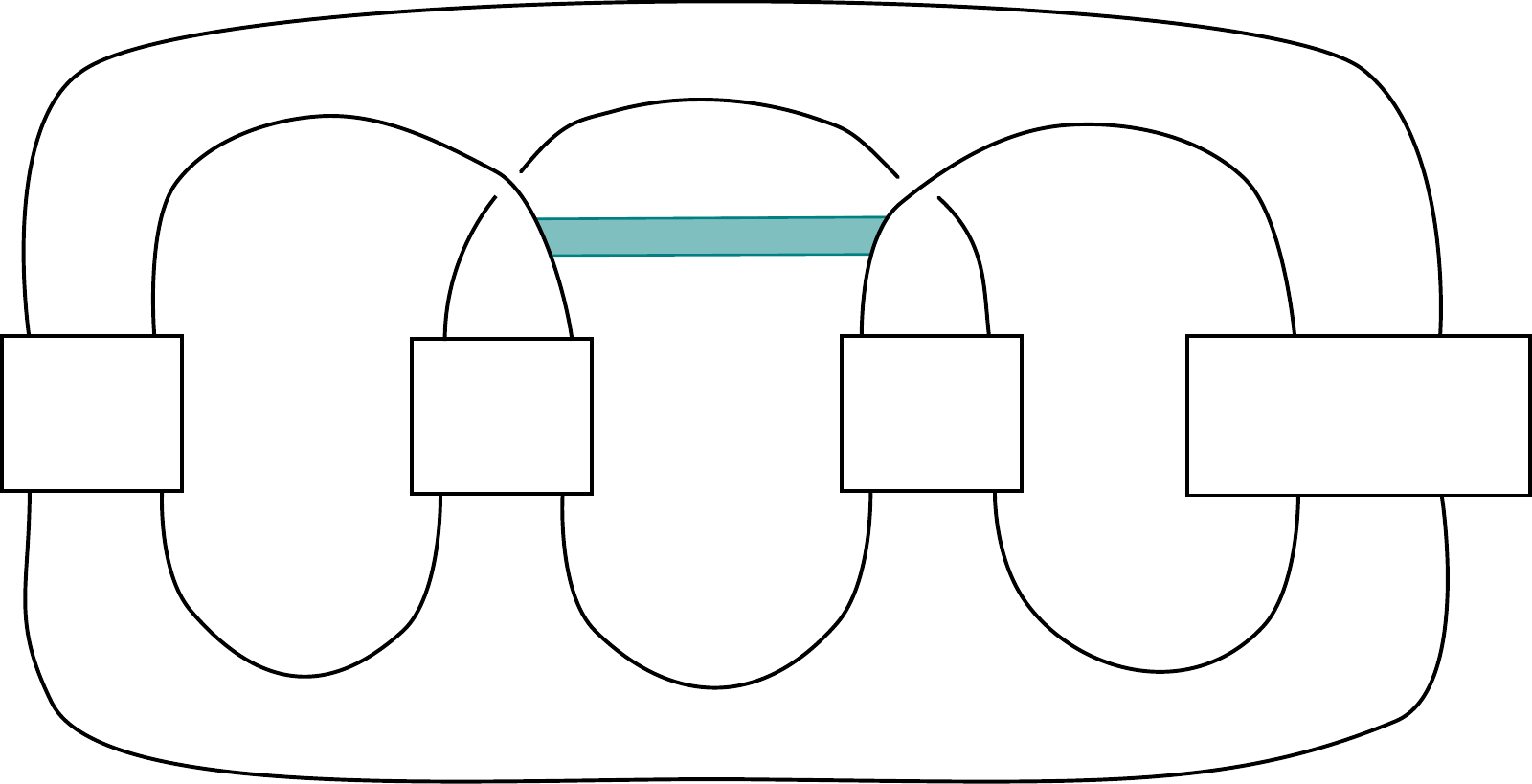}
        \put(5,22.5){$a$}
        \put(29,22.5){$-b$}
        \put(60,22.5){$b$}
        \put(78,22.5){$-(a\pm k)$}
      \end{overpic}
    }
    \vspace{30pt}
    
    \qquad
    \subfloat[Subfigure E][Obtaining the split union of an unknot and $P(a,1,-(a\pm k) $ from $P(a,2,2,-(a\pm k))$. \label{fig:ribbonmove:e}]{%
      \begin{overpic}[
      height=0.2 \textheight]{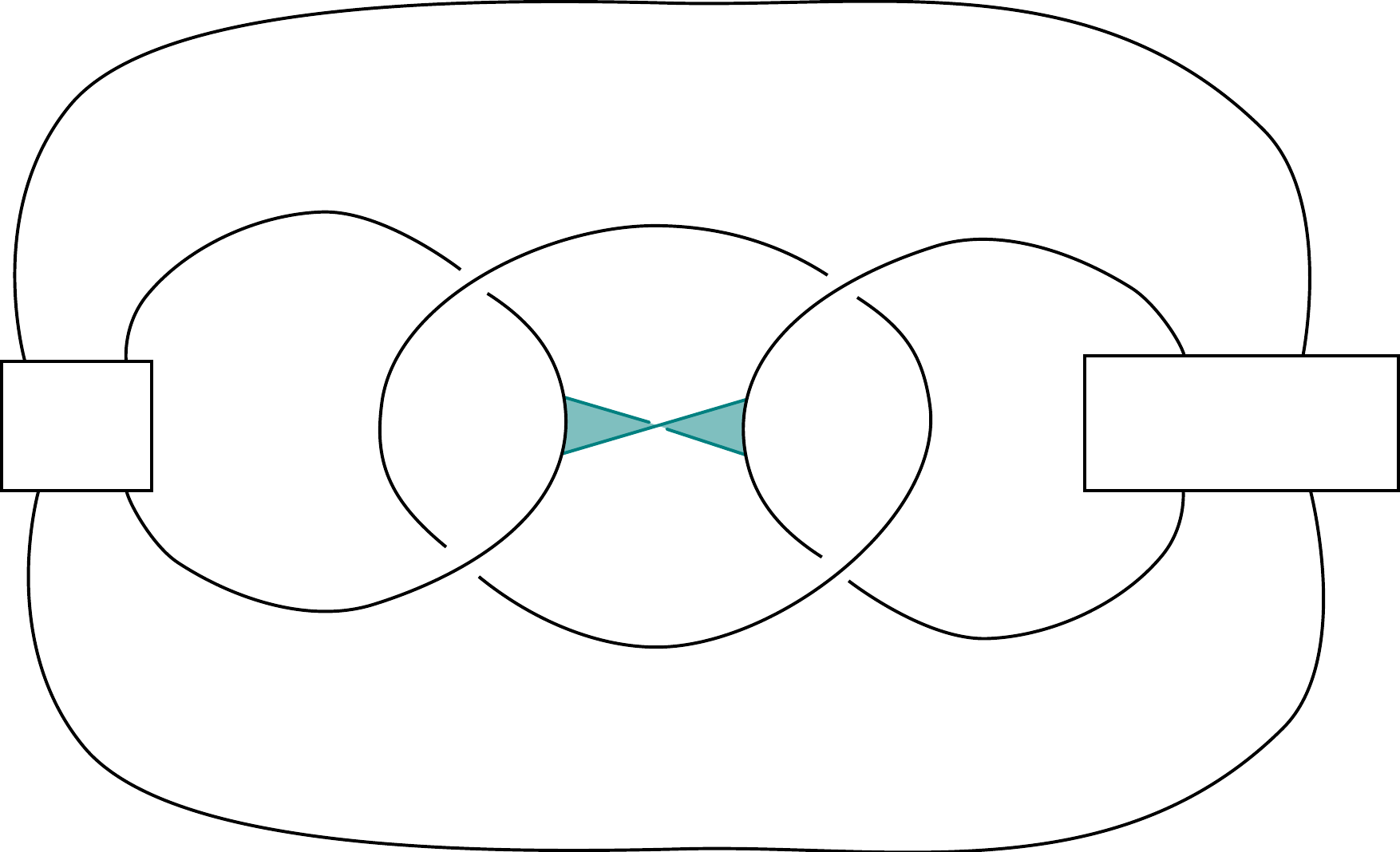}
        \put(4.5,29){$a$}
        \put(78,29){$-(a\pm k)$}
      \end{overpic}
    }
    \quad
    \subfloat[Subfigure F][Obtaining the unlink of two components from $P(1,3,a,-(a+3))$. \label{fig:ribbonmove:f}]{%
      \begin{overpic}[
      height=0.2 \textheight]{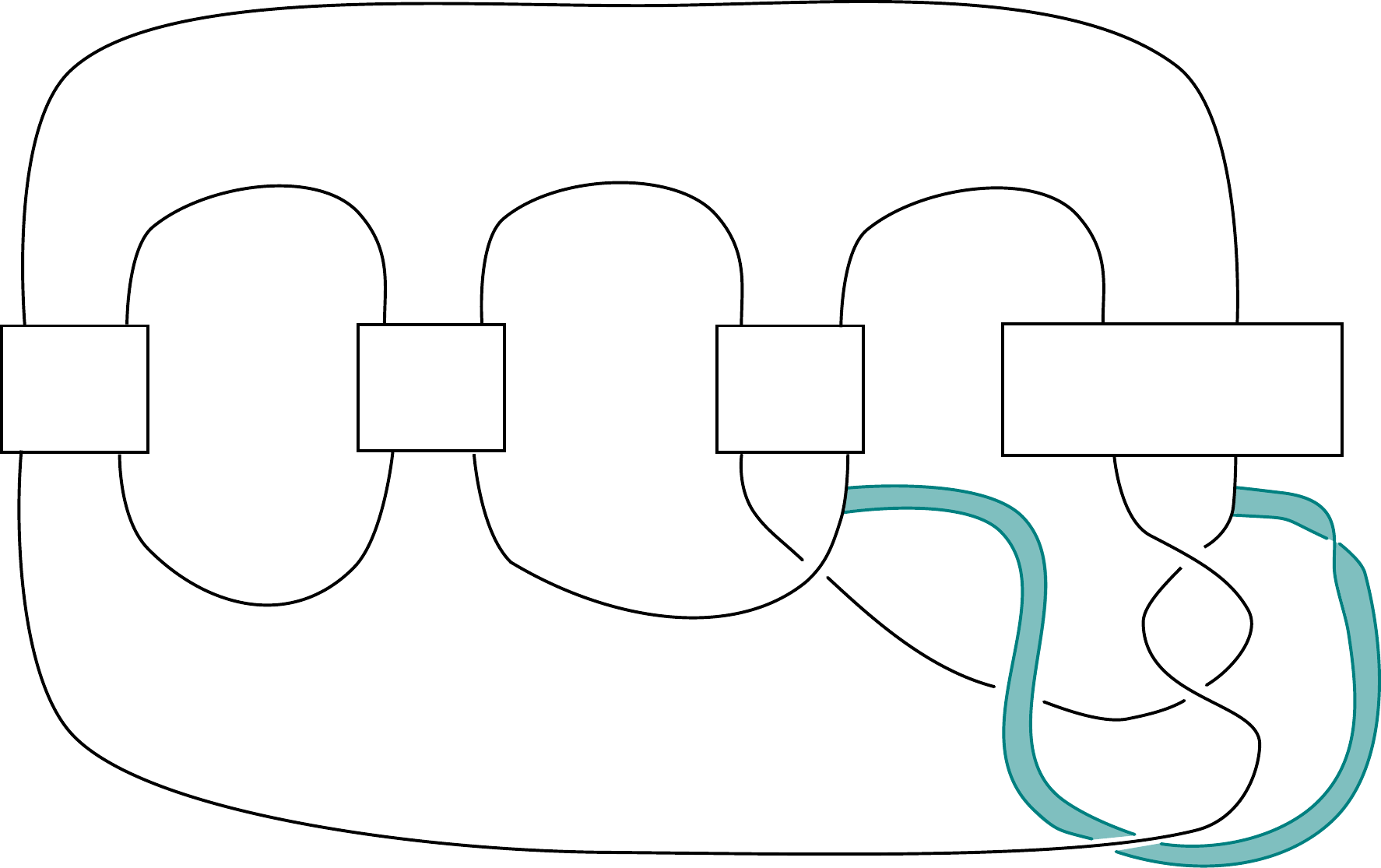}
        \put(4,33){$1$}
        \put(30,33){$3$}
        \put(56,33){$a$}
        \put(75,33){$-(a+3)$}
      \end{overpic}
    }    
  \vspace{20pt}
    \caption{Ribbon moves on families of pretzel links.}
    \label{fig:ribbonmove}

\end{figure}

\fourstrandedchislice*

\begin{proof}
The links $P(1,1,1,1)$ and $P(1,1,1,5)$ are $\chi-$ribbon by Lemma \ref{lem:2-bridgechi1}. 

Let $L=P(1,1,-2,-6)$; then $-L$ isotopic to the the two-bridge link $K(3,5)$. Since $[3,5]^+=\frac{9}{5}=[2,5]^-$ and the dual of $(2,5)$ is $(3,2,2,2)$, $-L$ (and hence $L$) is $\chi-$ribbon by Theorem \ref{thm:2bridgechislice}$(a)$. 

Performing the ribbon move to $P(2,2,-3,-6)$ shown in Figure \ref{fig:ribbonmove:a} yields  the split union of an unknot and the pretzel link $P(1,-3,-6)$. Since $P(1,-3,-6)$ is $\chi-$ribbon by Proposition \ref{prop:3-strand}, $P(2,2,-3,-6)$ is also $\chi-$ribbon by Proposition \ref{prop:ribbonmoves}.

For the links $P(2,2,-1,a)$ and $P(1,3,a, -(a+3))$,  we apply the ribbon moves indicated in Figure \ref{fig:ribbonmove:b} and \ref{fig:ribbonmove:f}, respectively, to obtain the two-component unlink, for any $a\neq 0$. 

Applying the ribbon move indicated in Figure \ref{fig:ribbonmove:c} to $P(3,-a,a,1)$ yields the split union of an unknot and $T(2,4)$ for any $a\neq 0$. Since $T(2,4)$ is $\chi-$ribbon by Corollary \ref{cor:torus}, the result follows.

For the links $P(a,-b,b,-(a\pm k))$ we apply the ribbon move indicated in Figure \ref{fig:ribbonmove:d} to obtain the split union of an unknot and the link $T(2,\pm k)$, which is $\chi-$ribbon if $k\in\{\pm1,\pm4\}$. It follows that the links in the final bullet point are $\chi-$ribbon.

Finally, applying the ribbon move indicated in Figure \ref{fig:ribbonmove:e} to the link $P(a,2,2,-(a\pm k))$ yields the split union of an unknot and the link $P(a,1,-(a\pm k))= P(1,a, -(a\pm k))$. Since $P(1,a, -(a+ k))$ is $\chi-$ribbon by Proposition \ref{prop:3-strand}, $P(a,2,2,-(a+ k))$ is also $\chi-$ribbon when $k\in \{0,4\}$.
\end{proof}

\pagebreak


\section{Obstructions}
\label{sec:background}

\subsection{Geometric Obstructions}\label{subsec:geometry}
For links with a small number of components, it is often possible to obstruct $\chi-$sliceness geometrically. 
For example, by the classification of surfaces, if $L$ is a $\chi-$slice link with two components, then $L$ must bound the disjoint union of a disk and Mobius band, implying that one of the link components must be slice. Consequently, if $L$ is a 2-component link and neither component is slice, then $L$ cannot be $\chi-$slice.
We will find this line of reasoning convenient for 3- and 4-stranded pretzel links. Another result in this direction is the following.

\begin{lem} Let $L\subset S^3$ be an oriented link that bounds the disjoint union of two surfaces $F_1$ and $F_2$ in $B^4$. Then the total linking number of $\partial F_1=L_1$ and $\partial F_2=L_2$ is even, i.e.

\[\emph{lk}(L_1,L_2):=\!\!\sum_{\substack{K_{\alpha} \emph{ in } L_1\\J_{\beta} \emph{ in } L_2}} \emph{lk}(K_{\alpha},J_{\beta})\equiv 0\pmod 2.\]

\noindent Moreover, if $F_1$ and $F_2$ are orientable, then the total linking number of $L_1$ and $L_2$ is zero.
\label{lem:linking}
\end{lem}

\begin{proof}
Let $L_1=\partial F_1$ and $L_2=\partial F_2$. The first part of the lemma is due to Lemma 2.4 in \cite{donaldowens}. We will now adapt their proof to prove the second part of the lemma. We may assume that the radial distance function $r$ on $B^4$ restricts to a Morse function on $F=F_1\sqcup F_2$ with values in $[.25,1]$ such that the index 0 critical values occur in the interval $[.25,.5)$, the index 1 critical values occur in the interval $(.5,.75)$, and the index 2 critical values occur in the interval $(.75,1)$. Let $(F_i)_t=r|_{F_i}^{-1}([0,t])$ and $(L_i)_t=r|_{F_i}^{-1}(t)$ for $i\in\{1,2\}$ and $t\in[0,1]$; similarly set $(F)_t=r|_{F}^{-1}([0,t])$ and $(L)_t=r|_{F}^{-1}(t)$.  Note that $\partial(F_i)_t=(L_i)_t$ and $(L_i)_1=L_i$. The level set $(L)_{.5}$ is an unlink and so $\text{lk}((L_1)_{.5},(L_2)_{.5})=0$. Since $F$ is oriented, all 1-handles of $F$, which are contained in the region $r^{-1}(.5,.75)$ are attached in an orientation-preserving way. In particular, $(F)_{.75}$ is orientable and ribbon. Consequently, $\text{lk}((L_1)_{.75},(L_2)_{.75})=0$; see Figure \ref{fig:ribbon-singularity}. Finally, it is easy to see that index 2 critical points do not change the linking number and so $\text{lk}(L_1,L_2)=0$.
\end{proof}

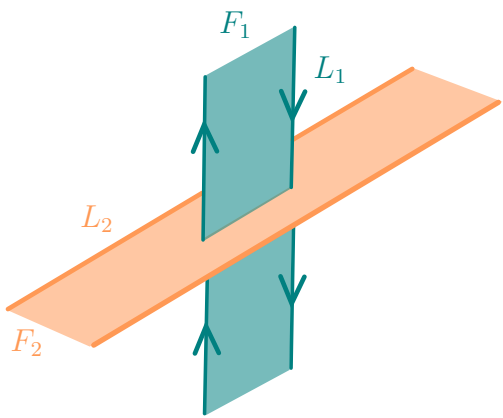
\begin{figure}[h!]
   \begin{overpic}
    [
width=.4\textwidth]{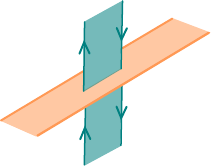}
  \put(62,68){\textcolor{specialteal}{$L_1$}}
   \put(15,38){\textcolor{specialorange}{$L_2$}}
    \put(43,77){\textcolor{specialteal}{$F_1$}}
     \put(1,13){\textcolor{specialorange}{$F_2$}}
\end{overpic}
    \caption{The surface $F$ before 2-handles are added; i.e. all level sets up to $t=.75$. In $S^3$ the only intersections between the two surfaces are ribbon. In particular, $(L_1)_{.75}$ intersects $(F_2)_{.75}$ algebraically zero times at each ribbon singularity, and hence $\text{lk}((L_1)_{.75},(L_2)_{.75})=0$.}
    \label{fig:ribbon-singularity}
\end{figure}

We will now use these geometric arguments to show that particular 3- and 4-stranded pretzel links are not $\chi-$slice. These results will be used in the proofs of Theorems \ref{thm:main3strandeda} and \ref{thm:main4strandednotchislice}.

\begin{lem} $L=P(2,-3,2,-6)$ is not $\chi-$slice.
\label{lem:2326}
\end{lem}

\begin{proof}
    First note that $L$ has three components two of which are unknots and one of which is the trefoil. Let $K_1$ denote the trefoil and let $K_2$ and $K_3$ denote the two unknots. 

    Suppose $L$ bounds a surface $F$ in $B^4$ with $\chi(F)=1$. By the classification of surfaces either:
    \begin{itemize}
    \item $F=D\sqcup D\sqcup F'$ for some surface $F'$ with $\chi(F')=-1$;
    \item $F=D\sqcup M\sqcup M$, where $M$ is a Mobius band; or
    \item $F=D\sqcup A$, where $A$ is an annulus.
    \end{itemize}

    \noindent Since at least one of the components of $L$ necessarily bounds a disk, it is not $K_1$, since the trefoil is not slice. We assume throughout that $K_2$ bounds a disk; the same arguments apply if $K_3$ bounds a disk. 
    First assume that $F=D\sqcup D\sqcup F'$ for some surface $F'$ with $\chi(F')=-1$. Then it follows that $K_2$ and $K_3$ bound disks. However, $\text{lk}(K_2,K_3)\neq0$, which contradicts Lemma \ref{lem:linking}. 
    Next assume that $F=D\sqcup M\sqcup M$, where $M$ is a Mobius band. It follows that $K_1$ bounds a Mobius band. However, $\text{lk}(K_1,K_2)=\pm1$, which contradicts Lemma \ref{lem:linking}.
    Finally assume that $F=D\sqcup A$, where $A$ is an annulus. Then $K_1$ and $K_3$ cobound an annulus; 
    this impossible, however, since $\text{lk}(K_2,K_1\sqcup K_3)\neq0$, which contradicts Lemma \ref{lem:linking}.
\end{proof}

\begin{lem}
Suppose $p>1$, $q,r<-1$ and $\frac{1}{p}+\frac{1}{q}+\frac{1}{r}>0.$ If $P(p,q,r)$ has three components (i.e. $p,q,r$ are all even), then $P(p,q,r)$ is not $\chi-$slice
\label{lem:3componentobstruct}
\end{lem}

\begin{proof}
Suppose $P(p,q,r)$ bounds a surface $F$ in $B^4$ with $\chi(F)=1$. By the classification of surfaces either:
\begin{itemize}
\item $F=D\sqcup D\sqcup F'$ for some surface $F'$ with $\chi(F')=-1$;
\item $F=D\sqcup M\sqcup M$, where $M$ is a Mobius band; or
\item $F=D\sqcup A$, where $A$ is an annulus.
\end{itemize}

\noindent Note that if we remove any one of the three components of $P(p,q,r)$, we obtain the torus link $T(2,a)$, where $a\in\{p,q,r\}$. 

First assume $F=D\sqcup D\sqcup F'$ for some surface $F'$ with $\chi(F')=-1$, then it follows that $T(2,a)$ is slice for some integer $a\in\{p,q,r\}$. The only slice torus links are the unknot $T(2,\pm1)$ and the  two component unlink $T(2,0)$. But, $a\neq0,\pm1$, by assumption.

Next assume $F=D \sqcup M\sqcup M$, where $M$ is a Mobius band. Then it follows that $T(2,a)$ and $T(2,b)$ are $\chi-$slice for some $a,b\in\{p,q,r\}$.
By Corollary \ref{cor:torus}, $a,b\in\{\pm1,\pm4\}$. Once again, $a,b\neq\pm1$ by assumption; hence $a,b\in\{\pm4\}$. If $a=-b$, then since exactly two parameters are negative, we have $\frac{1}{p}+\frac{1}{q}+\frac{1}{r}<0$, which is a contradiction. Hence $a=b=\pm4$. Again, since two parameters are negative, it follows that $a=b=-4$ and the third parameter $c>0$. But since $c>1$, it follows that $\frac{1}{a}+\frac{1}{b}+\frac{1}{c}\le0$, which is a contradiction. 

Finally assume $F=D\sqcup A$, where $A$ is an annulus. Let $L_1=\partial D$ and $L_2\cup L_3=\partial A$. Then by Lemma \ref{lem:linking}, the total linking number of $L_1$ and $L_2\cup L_3$ is zero. It is clear that this is only possible if $a+b=0$ for some integers $a,b\in\{p,q,r\}$. But then as above, $\frac{1}{a}+\frac{1}{b}+\frac{1}{c}<0$, which is a contradiction.
\end{proof}

\subsection{Donaldson's Obstruction and Lattice Embeddings}\label{embeddings}

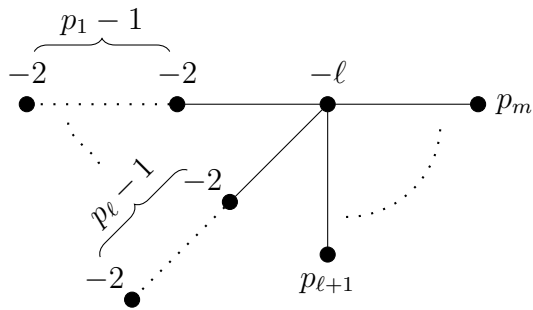
\begin{figure}[h!!]
    \centering
    \begin{tikzpicture}[dot/.style = {circle, fill, minimum size=1pt, inner sep=0pt, outer sep=0pt}]
\tikzstyle{smallnode}=[circle, inner sep=0mm, outer sep=0mm, minimum size=2mm, draw=black, fill=black];

\node[smallnode, label={90:$-\ell$}] (0) at (0,0) {};

\node[smallnode, label={90:$-2$}] (pi1a) at (-2,0) {};
\node[smallnode, label={90:$-2$}] (pi1b) at (-4,0) {};

\node[smallnode, label={[label distance=-.15cm]150:$-2$}] (pima) at (-1.3,-1.3) {};

\node[smallnode, label={[label distance=-.15cm]150:$-2$}] (pimb) at (-2.6,-2.6) {};

\node[smallnode, label={270:$p_{\ell+1}$}] (pim+1) at (0,-2) {};
\node (a) at (.1,-1.5) {};

\node[smallnode, label={0:$p_{m}$}] (pik) at (2,0) {};
\node (b) at (1.5,-.1) {};
\node (c) at (-3.45,-.1) {};
\node (d) at (-3.25,-.8) {};

\draw[-] (0) -- (pi1a);
\draw[loosely dotted, thick] (pi1a) -- (pi1b);
\draw[decorate,decoration={brace,amplitude=5pt,raise=.7cm,mirror},yshift=0pt] (pi1a) -- (pi1b) node [midway,yshift=1.1cm]{$p_{1}-1$};

\draw[-] (0) -- (pima);
\draw[loosely dotted, thick] (pima) -- (pimb);
\draw[decorate,decoration={brace,amplitude=5pt,raise=.7cm,mirror},yshift=0pt] (pima) -- (pimb) node [midway,sloped, yshift=1.1cm]{$p_{\ell}-1$};
\draw[-] (0) -- (pim+1);
\draw[-] (0) -- (pik);
\draw[loosely dotted, thick] (a) [out=0, in=-90] to (b);
\draw[loosely dotted, thick] (c) [out=-90, in=0] to (d);
\end{tikzpicture}
\caption{The plumbing $X(p_1,\ldots,p_m)$, where $p_{1},\ldots,p_{\ell}>0$ and $p_{{\ell+1}},\ldots,p_m<0$.}
    \label{fig:P}
\end{figure}

Recall that 
$Y(p_1,\ldots,p_m)$ denotes the double cover of $S^3$ branched over the pretzel link $P(p_1,\ldots,p_m)$, which can be realized as the boundary of the 4-dimensional plumbing shown in Figure \ref{fig:dbc}. Notice that $Y(p_1,\ldots,p_m)$ is diffeomorphic to $Y(p_{i_1},\ldots, p_{i_m})$, where the parameters $(p_{i_1},\ldots, p_{i_m})$ are any reordering of $(p_1,\ldots,p_m)$. Thus, on the level of the double branched cover, if $l$ is the number of positive parameters, we may assume that $p_1,\ldots,p_\ell>0$ and $p_{\ell+1},\ldots,p_m<0$. By performing suitable blowups and blowdowns at the vertices with weights $p_1,\ldots,p_\ell$, it is easy to see that $Y(p_1,\ldots,p_m)$ also bounds the plumbing $X(p_1,\ldots,p_m)$ shown in Figure \ref{fig:P}. 

The determinant of $P(p_1,\ldots,p_m)$ is given by $-\sum_{i=1}^mp_1\cdots p_{i-1}\,\widehat{p_i}\,p_{i+1}\cdots p_m$. Since $p_i\neq 0$ for all $i$, $\det P(p_1,\ldots,p_m)=0$ if and only if $\frac{1}{p_1}+\cdots+\frac{1}{p_m}=0$. Hence we are only concerned with pretzel links satisfying $\frac{1}{p_1}+\cdots+\frac{1}{p_m}\neq0$. Moreover, since the mirror of $P(p_1,\ldots,p_m)$ is $P(-p_1,\ldots,-p_m)$, we may restrict our attention to pretzel links satisfying $\frac{1}{p_1}+\cdots +\frac{1}{p_m}>0$. In this case, by Theorem 5.2 in \cite{neumannraymond}, $X(p_1,\ldots,p_m)$ has negative definite intersection form.

If $P(p_1,\ldots,p_m)$ is $\chi$-slice and $\det P(p_1,\ldots,p_m)\neq0$, then by Proposition \ref{prop:donaldowens}, $Y(p_1,\ldots,p_m)$ bounds a rational homology 4-ball $B$. Let $X:=X(p_1,\ldots,p_m)$ and assume that $\sum_i \frac{1}{p_i}>0$ so that $X$ is negative definite. Then $Z=X\cup (-B)$ is a smooth closed negative definite 4-manifold. By Donaldson's Diagonalization Theorem \cite{donaldson}, there exists a basis for $H_2(Z)\cong \ZZ^n$ for which the intersection form $Q_Z$ can be represented by the matrix $-I$. Consequently, there exists a lattice embedding $\phi:(H_2(X),Q_{X})\to (\ZZ^{\text{rank}(H_2(X))},-I)$; that is, a homomorphism $\phi:H_2(X)\to \ZZ^{\text{rank}(H_2(X))}$ satisfying $Q_X(a,b)=-I(\phi(a),\phi(b))$ for all elements $a,b\in H_2(X)$.

\begin{thm} In the above setup, if there does not exist a lattice embedding $(H_2(X),Q_{X})\to (\ZZ^{\text{rank}(H_2(X))},-I)$, then $Y(p_1,\ldots,p_m)$ does not bound a rational homology ball.
\label{thm:donaldson}
\end{thm}

Endow $H_2(X)$ with the basis given by the 2-handles of $X$ (i.e. the vertices of the plumbing diagram).  
We will follow the notation in \cite{greenejabuka} and use $\{f_1,\ldots ,f_n\}$ to refer to this basis on $H_2(X)\cong\ZZ^n$ and use $\{e_1, ..., e_n\}$ to refer to the standard orthonormal basis on the codomain of $\phi$ with respect to which the intersection form is $-I$. 
Given two vectors $v,w$ in the codomain of $\phi$, we write $v\cdot w:=-I(v,w)$. 
Note that $\phi(f_j) \cdot \phi(f_i)$ is the $(i,j)-$th entry of $Q_X$ since
\begin{center}
    $\phi(f_j) \cdot \phi(f_i) =-I(\phi(f_j), \phi(f_i)) =  Q_X(f_i, f_j).$
\end{center}  

Following the notation in \cite{lecuona}, let $C_i=\{f_{i_1},\ldots,f_{i_{p_{i}-1}}\}$ denote the basis vectors corresponding to the chain of $p_{i}-1$ vertices of weight $-2$ in Figure \ref{fig:P}, where $1\le i\le \ell$, and $f_{i_{p_{i}-1}}$ corresponds to the vertex of $C_i$ adjacent to the central vertex. Finally, let $U_{C_i}=\{e_\alpha\,|\,\phi(f_{i_j})\cdot e_\alpha\neq0\text{ for some }j\}$.

\begin{prop}[c.f. Lemma 5.1 in \cite{lecuona}]\label{prop:facts}\hfill
\begin{enumerate}
\item\label{1} The images of the basis vectors in $C_i$ are of the form $\phi(f_{i_j})=e_{i_j}-e_{i_{j+1}}$, for all $1\le j\le p_i-1$; consequently if $p_i\ge2$, then $|U_{C_{i}}|=p_i.$
\item\label{2} If $i\neq j$ and $|C_i|\ge2$ or $|C_j|\ge2$, then $|U_{C_i}\cap U_{ C_j}|=0$.
\item\label{3} If $i\neq j$ and $|C_i|=|C_j|=1$, then either $|U_{C_i}\cap U_{C_j}|=0$ or $|U_{C_i}\cap U_{ C_j}|=2$. In the second case, if $C_i=\{f_i\}$ and $C_j=\{f_j\}$, then up to the action of $\Aut\ZZ^n$, $\phi(f_i)=e_x-e_y$ and $\phi(f_j)=-e_x-e_y$ for some integers $x$ and $y$.
\item\label{4} If $|U_{C_i}\cap U_{C_j}|\neq0$, then $|U_{C_k}\cap U_{C_i}|=|U_{C_k}\cap U_{C_j}|=0$ for all $k\neq i,j$. Consequently, $\#\{i\,:\,|U_{C_i}\cap U_{C_j}|\neq 0\text{ for some }j\neq i\}$ is even.
\end{enumerate} 
\end{prop}

\begin{proof}

(\ref{1}): Up to the action of $\Aut\ZZ^n$, we may assume that $C_i=\{f_1,\ldots,f_{p_i-1}\}$. Let $f$ denote the basis vector corresponding to the central vertex in Figure \ref{fig:P}. Since $\phi(f_1)^2=-2$, it follows that $\phi(f_1)=\pm e_x\pm e_y$ for some integers $x$ and $y$; up to the action of $\Aut\ZZ^n$, we may assume $\phi(f_1)=e_1-e_2$. Since $\phi(f_2)^2=-2$ and $\phi(f_2)\cdot \phi(f_1)=1$, it follows that $\phi(f_2)=\pm e_x\pm e_y$, where $x\in\{1,2\}$ and $y\not\in\{1,2\}$. Once again, up to the action of $\Aut\ZZ^n$, we may assume  $\phi(f_2)= e_2-e_3$. Since $\phi(f_3)^2=-2$ and $\phi(f_3)\cdot \phi(f_2)=1$, it follows that $\phi(f_2)=\pm e_x\pm e_y$, where $x\in\{2,3\}$ and $y\not\in\{2,3\}$.
Assuming $x=2$; then since $\phi(f_3)\cdot\phi(f_1)=0$, it follows that $y=1$ and  $\phi(f_3)= -e_2-e_1$. Note that the vertex corresponding to $f_3$ in Figure \ref{fig:P} connects to two vertices: the vertex corresponding to $f_2$ and the vertex corresponding to either $f$ (if $p_i=4$) or the vertex corresponding to $f_4$ (if $p_i>4$).
Let $F=f$ if $p_i=4$ and $F=f_4$ if $p_i>4$. Then since $\phi(F)\cdot \phi(f_3)=1$, either $\phi(F)\cdot e_1\neq0$ or $\phi(F)\cdot e_2\neq0$. If $\phi(F)\cdot e_1\neq0$ (or similarly $\phi(F)\cdot e_2\neq0$), then since $\phi(F)\cdot \phi(f_1)=0$, it follows that $\phi(F)=ae_1+ae_2+g$ for some nonzero integer $a$ and some vector $g$. But then $\phi(F)\cdot \phi(f_3)=2a\neq1$, which is a contradiction. Hence $x=3$. Since $y\not\in\{2,3\}$ and $\phi(f_3)\cdot \phi(f_1)=0$, it follows that $y\not\in\{1,2,3\}$. Up to the action of $\Aut\ZZ^n$, we may assume that $\phi(f_3)=e_3-e_4$.
Continuing inductively, we obtain the result.

(\ref{2}): Without loss of generality, suppose that $|C_i|\ge2$. Up to the action of $\Aut\ZZ^n$, we may assume that $C_i=\{f_1,\ldots,f_{p_i-1}\}$ and $C_j=\{f_{p_i},\ldots,f_{p_i+p_j-1}\}$. By (\ref{1}), we have that $\phi(f_j)=e_j-e_{j+1}$ for all $1\le j\le p_i-1$. If there exists $e_\alpha\in U_{C_i}\cap U_{C_j}$, then by definition, $1\le\alpha\le p_i$ and there exists an integer $p_i\le k\le p_i+p_j-1$ such that $\phi(f_k)\cdot e_\alpha\neq0$. Consequently, since $\phi(f_k)^2=-2$, $\phi(f_k)=\pm e_\alpha\pm e_l$ for some integer $l$. Since $\phi(f_\alpha)=e_\alpha-e_{\alpha+1}$ and $\phi(f_\alpha)\cdot \phi(f_k)=0$, we necessarily have that $l=\alpha+1$ and so $\phi(f_k)=\pm(e_\alpha + e_{\alpha+1})$. If $\alpha>1$, then since $\phi(f_{\alpha-1})=e_{\alpha-1}-e_{\alpha}$, it follows that $\phi(f_{\alpha-1})\cdot \phi(f_k)\neq0$, which is a contradiction. Similarly, if $\alpha=1$, then since $\phi(f_{2})=e_{2}-e_{3}$, it follows that $\phi(f_{\alpha-1})\cdot \phi(f_k)\neq0$, which is a contradiction since $p_i-1\ge2$.

(\ref{3}): Assume $i\neq j$, $|C_i|=|C_j|=1$, and $|U_{C_i}\cap U_{C_j}|\neq 0$. 
Set $C_i=\{f_i\}$ and $C_j=\{f_j\}$. By (\ref{1}),
up to the action of $\Aut \ZZ^n$, we may assume that $\phi(f_i)=e_x-e_y$, for some integers $x$ and $y$. 
Since $|U_{C_i}\cap U_{C_j}|\neq 0$, either $e_x\in U_{C_i}\cap U_{C_j}$ or  $e_y\in U_{C_i}\cap U_{C_j}$. If $e_x\in U_{C_i}\cap U_{C_j}$ and $e_y\not\in U_{C_i}\cap U_{C_j}$ (or $e_y\in U_{C_i}\cap U_{C_j}$ and $e_x\not\in U_{C_i}\cap U_{C_j}$), then it follows that $\phi(f_i)\cdot\phi(f_j)=\pm1$, which is a contradiction. Hence $U_{C_i}\cap U_{C_j}=\{e_x,e_y\}$. Now since $\phi(f_i)\cdot\phi(f_j)=0$, up to the action of $\Aut\ZZ^n$, $\phi(f_j)=-e_x-e_y$ and $|U_{C_i}\cap U_{C_j}|=2$.

(\ref{4}): If $|U_{C_i}\cap U_{C_j}|\neq0$, then by (\ref{2}) and (\ref{3}), $|C_i|=|C_j|=1$ and up to the action of $\Aut\ZZ^n$, $\phi(f_i)=e_x-e_y$ and $\phi(f_j)=-e_x-e_y$, where $C_i=\{f_i\}$ and $C_j=\{f_j\}$. Assume there exists an integer $k$ such that $|U_{C_k}\cap U_{C_i}|\neq 0$. Then by (\ref{3}), $U_{C_k}=\{x,y\}$. Since $\phi(f_k)\cdot \phi(f_i)=0$, we have that $\phi(f_k)=\pm (e_x+e_y)$. But then $\phi(f_k)\cdot \phi(f_j)\neq0$, which is a contradiction. Hence $|U_{C_k}\cap U_{C_i}|=0$ and similarly $|U_{C_k}\cap U_{C_j}|=0$.
\end{proof}

\subsection{Greene-Jabuka obstruction} 
The next obstruction to $\chi-$sliceness we will use comes from Heegaard Floer $d-$invariants, and  builds on work of Greene-Jabuka in \cite{greenejabuka}. Since pretzel knots are the focus in their paper, $|H^2(Y(p_1,\ldots,p_k))|$ has no 2-torsion and consequently, the first Chern class is used to identify $\text{Spin}^c(Y(p_1,\ldots,p_k))$ with $H^2(Y(p_1,\ldots,p_k))$. Since we are considering links, it is not always the case that $|H^2(Y(p_1,\ldots,p_k))|$ has no 2-torsion. Hence a generalization of their obstruction to rational homology spheres with arbitrary $|H^2|$ is needed. A starting point for such a generalization is given in \cite{boyleissa}. The following is their result, adapted to our conventions.

\begin{prop}[\cite{greenejabuka}, Proposition 4.2 in \cite{boyleissa}]
Let $X$ be a simply connected negative definite smooth 4-manifold with boundary a rational homology sphere $Y$. Choose a basis for $H_2(X)$ and let $Q$ denote the intersection form of $X$. Suppose $Y$ bounds a rational homology ball $B$. 
Let $H^2(X)$ have the hom-dual basis of $H_2(X)$ and endow $H^2(X\cup (-B))$ with the basis for which the intersection form on $X\cup(-B)$ is $-I$ (the existence of which follows from Theorem \ref{thm:donaldson}). Let $A$ denote the matrix representing the map induced by inclusion $H^2(X\cup (-B))\to H^2(X)$. Then $A^T$ represents the lattice embedding map $H_2(X)\to H_2(X\cup(-B))\cong\ZZ^{\text{rank}(H_2(X))}$ and $Q=-AA^T$. Moreover, the spin-c structures on $Y$ that extend over $B$ are those of the form 
$$Av \pmod{2Q}\in Char(H^2(X))/\im(2Q)=\text{Spin}^c(Y)$$
where $v\in H^2(X\cup(-B))=\ZZ^n$ is any vector with all odd entries and $Char(H^2(X))$ is the set of characteristic elements of $H^2(X)$.
\label{prop:boyleissa}
\end{prop}

Continuing with the notation in Proposition \ref{prop:boyleissa}, each element in the image of $H^2(B)\to H^2(Y)$ corresponds to a spin-c structure on $Y$ that extends over $B$. By \cite{os-absolutelygraded}, for each such spin-c structure $\mathfrak{s}$, $d(Y,\mathfrak{s})=0$. It is known that the image of $H^2(B)\to H^2(Y)$ has order $\sqrt{|H_1(Y)|}$ (c.f. Lemma 3 in \cite{cassongordon}); hence there are $\sqrt{|H_1(Y)|}$ such vanishing $d-$invariants. Note that $\sqrt{|H_1(Y)|}=\sqrt{|\det Q|}=|\det A^T|$; moreover in light of Proposition \ref{prop:donaldowens}, if $Y$ is the double cover of $S^3$ branched along a nonzero determinant link $L$, then $\sqrt{|H_1(Y)|}=\sqrt{|\det L|}$.

Further suppose that $X$ is a pluming tree with $n$ vertices and with at most one vertex whose valence is greater than the absolute value of its weight. Then by \cite{ozsvathszaboplumbed} and Proposition \ref{prop:boyleissa}, for any $\mathfrak{s}\in\text{Spin}^c(Y)$,
$$d(Y,\mathfrak{s})=\max_{Av\in\text{Char}_{\mathfrak{s}}(Q)} \frac{n-(Av)^TQ^{-1}(Av)}{4}=\max_{v\in\ZZ^n_{odd}} \frac{n-|v|^2}{4}.$$
It follows that if $d(Y,\mathfrak{s})=0$, then $v=(v_1,\ldots,v_n)=(\pm1,\ldots, \pm1)$. 
Set $S=\{(x_1,\ldots,x_n)\in\ZZ^n\,|\,x_i=\pm1\}$.
Following as in \cite{greenejabuka}, by Proposition \ref{prop:boyleissa}, it is easy to see that if $Av$ and $Av'$ represent the same spin-c structure on $Y$, then $v-v'\in \im(2A^T)$. We use $S/\im(2A^T)$ to denote the set of vectors $v\in S$ up to the equivalence given by $v\sim v'$ if and only if $v-v'\in\im(2A^T)$.

\begin{thm} In the above setup, if 
$|S/\text{Im}(2A^T)|<|\det A^T|$, then $Y$ does not bound a rational homology 4-ball. In particular, if $Y$ is the double branched cover of $S^3$ branched along a nonzero determinant link $L$ and
$|S/\text{Im}(2A^T)|<\sqrt{|\det L|}$, then $L$ is not $\chi-$slice.
\label{thm:dinvts}
\end{thm}

\subsection{Unordered parameters}\label{sec:unordered} Recall that if $(p_{i_1},\ldots,p_{i_m})$ is a cyclic reordering and/or reversal of $(p_1,\ldots,p_m)$, then the pretzel links $P(p_1,\ldots,p_m)$ and $P(p_{i_1},\ldots,p_{i_m})$ are isotopic. 
On the 3- and 4-manifold level, however, more is true.
Note that if 
$(p_{i_1},\ldots,p_{i_m})$ is \textit{any} reordering of $(p_1,\ldots,p_m)$, then $Y(p_1,\ldots,p_m)$ is diffeomorphic to $Y(p_{i_1},\ldots,p_{i_m})$ (and $X(p_1,\ldots,p_m)$ is diffeomorphic to $X(p_{i_1},\ldots,p_{i_m})$).
Using the notation introduced in the introduction, every link in $P\langle p_1,\ldots,p_m\rangle$ has the same double branched cover---$Y(p_1,\ldots,p_m)$.
Hence two pretzel links with the same number of strands and whose parameters are the same as unordered sets admit diffeomorphic double branched covers. Consequently, the main obstructions to $\chi-$sliceness that we will employ (Donaldson's obstruction and the Greene-Jabuka obstruction) do not differentiate between pretzel links with the same unordered parameters. 

\begin{figure}[h!]
   \begin{overpic}[
   width=.9\textwidth]{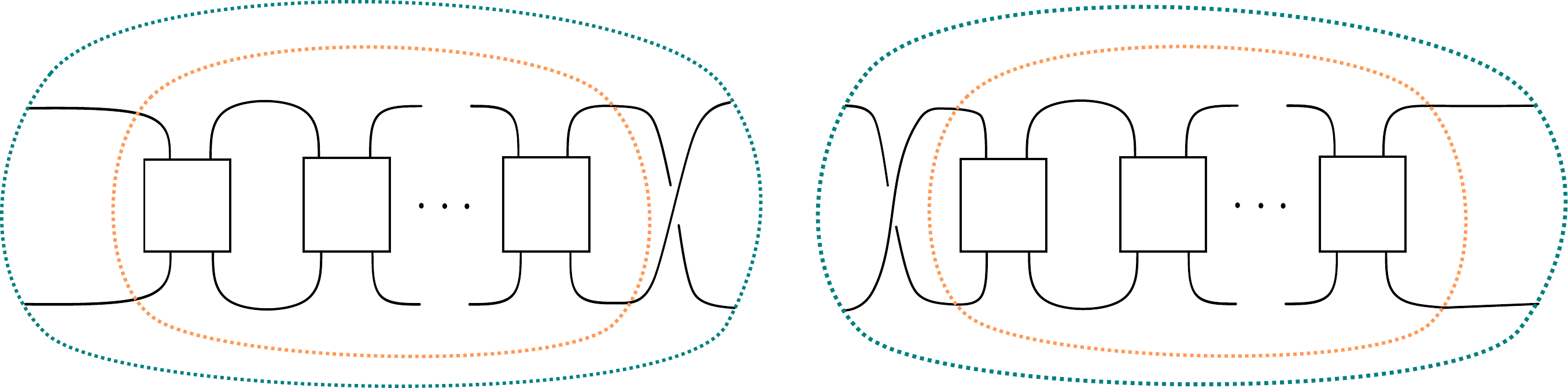}
        \put(1,21){\textcolor{specialteal}{$\widehat{T}$}}
        \put(7,20){\textcolor{specialorange}{$T$}}
        \put(52,21){\textcolor{specialteal}{$\widehat{T}$}}
        \put(59,4.5){\textcolor{specialorange}{\rotatebox{180}{\reflectbox{$T$}}}}
        \put(11.5,11){$x$}
        \put(21.5,11){$x$}
        \put(33.8,11){$x$}
        \put(63.2,12){\rotatebox{180}{\reflectbox{$x$}}}
        \put(73.5,12){\rotatebox{180}{\reflectbox{$x$}}}
        \put(86,12){\rotatebox{180}{\reflectbox{$x$}}}
    \end{overpic}
    \caption{A tangle $\widehat{T}$  in the pretzel link $L=P(\ldots,x^{[a]},1,\ldots)$. Performing a flype on $\widehat{T}$ gives an isotopic tangle, and rotates $T$ about a horizontal axis. In our situation, $T$ is symmetric about this axis so that this move gives an isotopy between $L$ and $P(\ldots, 1,x^{[a]},\ldots)$}
    \label{fig:flype}
\end{figure}
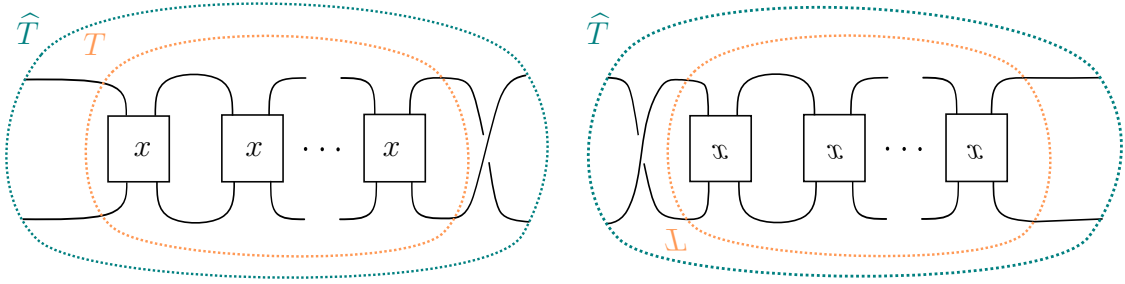

For 3-stranded pretzel links, this is not a problem; indeed, any reordering of three integers $p,q,r$ can be obtained by cyclic reordering and reversing and so any reordering of $p,q,r$ does not change the isotopy type of $P(p,q,r)$. For 4-stranded pretzel links, however, there are three possible orderings of $p,q,r,s$ up to cyclic reordering and reversal; these are represented by $(p,q,r,s)$, $(q,p,r,s)$, and $(q,r,p,s)$. In general, two distinct orderings of a pretzel link might give rise to isotopic pretzel links, as shown in the following result, which will be used in the proof of Theorem \ref{thm:mainpositivepretzellinks}.

\begin{lem} Let $\epsilon_i\in\{-1,1\}$ for all $1\le i\le k$ and let $a=a_1+\cdots+a_k$. Then for any nonzero integers $x$ and $y$, $P(\epsilon_{1},x^{[a_1]},\epsilon_2, x^{[a_2]}\ldots,\epsilon_k,x^{[a_k]},y)$ is isotopic to $P(\epsilon_1,\ldots,\epsilon_k,x^{[a]},y)$.
\label{lem:flype}
\end{lem}

\begin{proof}
    By performing a sequence of flypes to $L=P(\epsilon_{1},x^{[a_1]},\ldots,\epsilon_k,x^{[a_k]},y)$, we can arrange that all strands with a single $\pm1$ half-twist appear as consecutive strands of $L$; see Figure \ref{fig:flype}. The result follows.
\end{proof}


\section{Positive pretzel links}\label{sec:positive}

In this section we will consider the family of positive $m$-stranded pretzel links for $m\ge3$; that is the pretzel links $P(p_1,\ldots, p_m)$ satisfying $p_i > 0$ for all $i$. We will prove Theorem \ref{thm:mainpositivepretzellinks}.
   
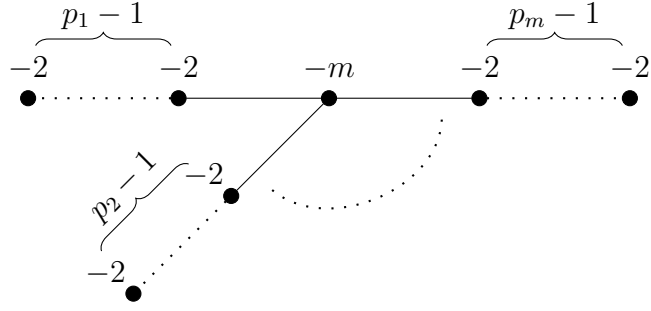
\begin{figure}[h!]
  \begin{tikzpicture}[dot/.style = {circle, fill, minimum size=1pt, inner sep=0pt, outer sep=0pt}]
\tikzstyle{smallnode}=[circle, inner sep=0mm, outer sep=0mm, minimum size=2mm, draw=black, fill=black];

\node[smallnode, label={90:$-m$}] (0) at (0,0) {};

\node[smallnode, label={90:$-2$}] (p1a) at (-2,0) {};
\node[smallnode, label={90:$-2$}] (p1b) at (-4,0) {};

\node[smallnode, label={[label distance=-.15cm]150:$-2$}] (p2a) at (-1.3,-1.3) {};

\node[smallnode, label={[label distance=-.15cm]150:$-2$}] (p2b) at (-2.6,-2.6) {};

\node[smallnode, label={90:$-2$}] (pma) at (2,0) {};
\node[smallnode, label={90:$-2$}] (pmb) at (4,0) {};

\node (a) at (-.9,-1.1) {};
\node (b) at (1.5,-.1) {};

\draw[-] (0) -- (p1a);
\draw[loosely dotted, thick] (p1a) -- (p1b);
\draw[decorate,decoration={brace,amplitude=5pt,raise=.7cm,mirror},yshift=0pt] (p1a) -- (p1b) node [midway,yshift=1.1cm]{$p_{1}-1$};

\draw[-] (0) -- (p2a);
\draw[loosely dotted, thick] (p2a) -- (p2b);
\draw[decorate,decoration={brace,amplitude=5pt,raise=.7cm,mirror},yshift=0pt] (p2a) -- (p2b) node [midway,sloped, yshift=1.1cm]{$p_{2}-1$};

\draw[-] (0) -- (pma);
\draw[loosely dotted, thick] (pma) -- (pmb);
\draw[decorate,decoration={brace,amplitude=5pt,raise=.7cm},yshift=0pt] (pma) -- (pmb) node [midway,yshift=1.1cm]{$p_{m}-1$};
\draw[loosely dotted, thick] (a) [out=-40, in=-90] to (b);
\end{tikzpicture}
\caption{The plumbing $X(p_1,\ldots,p_k)$, where $p_1,\ldots,p_m>0$.}
    \label{fig:positivepretzel}
\end{figure}

Consider the plumbing diagram in Figure \ref{fig:positivepretzel}. Set $X:=X(p_1,\ldots,p_m)$ and $n=p_1+\cdots+p_m-(m-1)$, and let $\phi:(\Z^n,Q_X)\to(\Z^n,-I)$ be a lattice embedding. Let $\{f_1,\dots,f_n\}$ be the standard basis of the domain of $\phi$ corresponding to the vertices in Figure \ref{fig:positivepretzel}. Let $\{e_1,\dots,e_n\}$ be the standard basis of the codomain of $\phi$. Let $C_i$ denote the set of basis vectors corresponding to the chain of $p_i-1$ spheres with weight $-2$ for $1\le i\le m$. Recall that $U_{C_i}=\{e_\alpha\,|\,\phi(f_{i_j})\cdot e_\alpha\neq0\text{ for some }j\}$. 
Finally, we define
$$k_{\phi}:=\frac{\#\{i\,:\,|U_{C_i}\cap U_{C_j}|\neq0\text{ for some }j\neq i\}}{2}\,\,\text{ and } \,\,z_\phi:=\#\{i\,:\,p_i=1\}.$$
Note that by Proposition \ref{prop:facts}(\ref{4}), $k_\phi$ is an integer.

\begin{lem}
If $k_{\phi}=0$ and $L$ is a $\chi-$slice positive pretzel link, then $L$ is isotopic to $P(1^{[z]},z\pm2)$ for some integer $z$ satisfying $z\ge2$ and $z\pm2\ge1$.
\label{lem:k=0}
\end{lem}

\begin{proof}
Set $z:=z_\phi$. By Proposition \ref{prop:facts}(\ref{1}), for fixed $1\le i\le m$, 
\[
|U_{C_i}|= 
\begin{cases}
p_i \text{ if } p_i>1\\
0 \text{ if } p_i=1.
\end{cases}
\]
Up to reordering, we may assume that $p_1,\ldots,p_z=1$ and $p_{z+1},\ldots,p_m>1$ (where if $z=0$, it is understood that $p_i>1$ for all $i$). Since $k_{\phi}=0$, we have that $|U_{C_i}\cap U_{C_j}|=0$ for all $i\neq j$. Thus, the number of basis vectors $n$ must be greater than or equal to $\sum_{i=1}^m|U_{C_i}|=p_{z+1}+\cdots+p_m$. On the other hand, $n$ is the number of vertices in the plumbing graph and hence $n=1+\sum_{i=1}^m(p_i-1)=p_{z+1}+\ldots+p_m-(m-z)+1$. Hence $p_{k+1}+\cdots+p_m-(m-z)+1\ge p_{k+1}+\cdots +p_{m}$, which implies that $z\ge m-1$. Since $m$ is the number of parameters in pretzel link $L$, it is clear that $z\le m$.
Thus $L$ is isotopic to $P(1^{[z]},a)$, where $a\ge 1$. By Lemma \ref{lem:2-bridgechi1}, $a=z\pm2$. Finally, since $L$ has at least three parameters, each of which is positive, $z\ge2$ and $z\pm2\ge1$.
\end{proof}

\begin{lem} If $k_{\phi}\neq0$, then
$m\in\{2k_\phi+z_\phi,2k_{\phi}+z_\phi+1\}$.
    \label{lem:pos1}
\end{lem}

\begin{proof}
Set $k:=k_{\phi}$ and $z:=z_\phi$. Since $m$ is the number of parameters in the pretzel link, it is clear that $m\ge 2k+z$. We will now show that $m\le 2k+z+1$ by interpreting $-m$ as the weight of the central vertex of the plumbing graph.  
Up to relabeling, we may assume that $|U_{C_{2i-1}}\cap U_{C_{2i}}|\neq0$ for all $1\le i\le k$ and by Proposition \ref{prop:facts}(\ref{4}), $|U_{C_j}\cap U_{C_{2i-1}}|=0$ for all $j\neq 2i$ and $|U_{C_j}\cap U_{C_{2i}}|=0$ for all $j\neq 2i-1$.
By Proposition \ref{prop:facts}(\ref{2}), $|C_\alpha|=1$ for all $1\le \alpha\le 2k$; let $f_\alpha$ denote the basis element corresponding to the unique vertex in $C_\alpha$. Then by Proposition \ref{prop:facts}(\ref{3}), up to the action of $\Aut\ZZ^n$, we have that $\phi(f_{2i-1})=e_{2i-1}-e_{2i}$ and $\phi(f_{2i})=-e_{2i-1}-e_{2i}$. It follows that the numer of vertices is $n=1+\sum_{i=1}^m(p_i-1)=1+2k+\sum_{i=2k+1}^m(p_i-1)=1-m+4k+\sum_{i=2k+1}^mp_i.$
Let $Z=\{j\,:\, p_j=1\}$ and note that $z=|Z|.$
By Proposition \ref{prop:facts}(\ref{1}), $|U_{C_j}|=p_j$ for all $j\not\in Z$ and $|U_{C_j}|=0$ for all $j\in Z$. Moreover, if $j>2k$, then $|U_{C_{j}}\cap U_{C_{i}}|=0$ for all $i\neq j$. It now follows from counting the number of basis vectors $e_\alpha$ in the sets $U_{C_i}$ that $\displaystyle n\ge2k+\sum_{\substack{i=2k+1\\i\not\in Z}}^mp_i.$ Hence $\displaystyle 2k+\sum_{\substack{i=2k+1\\i\not\in Z}}^mp_i\le1-m+4k+\sum_{i=2k+1}^mp_i,$ which implies that $m\le2k+z+1$. The result follows.
\end{proof}

\begin{prop}
If $k_{\phi}\neq0$, $m=2k_\phi+z_\phi$, and $L$ is a $\chi-$slice positive pretzel link, then $L$ is isotopic to $P(1^{[z]},2^{[2k]})$, where $z\ge0$ and $k\ge1$ are integers satisfying $z+k=4$.
\label{prop:knot-1}
\end{prop}

\begin{proof}
Let $k:=k_\phi$ and $z:=z_\phi$. Using Proposition \ref{prop:facts}(\ref{1} \& \ref{3}), we have $p_j=|U_{C_j}|=2$ for all $1\leq j\leq 2k$. Since $m=2k+z$, it follows that $L$ is isotopic to a link in the set $P\langle1^{[z]},2^{[2k]}\rangle$. By Lemma \ref{lem:flype}, $L$ is necessarily isotopic to $P(1^{[z]},2^{[2k]})$. Note that $n=2k+1$.
Using the conventions in the proof of Lemma \ref{lem:pos1}, we have that we have that $\phi(f_{2i-1})=e_{2i-1}-e_{2i}$ and $\phi(f_{2i})=-e_{2i-1}-e_{2i}$ for all $1\le i\le k$. 
Let $f_{2k+1}$ be the basis vector corresponding to the central vertex. Since $\phi(f_{2k+1})\cdot\phi(f_{\alpha})=1$ for all $1\le \alpha\le 2k$, it follows that up to the action of $\Aut\ZZ^n$,
$$\phi(f_{2k+1})=e_2+e_4+\ldots+e_{2k}+\lambda e_{2k+1}$$
for some integer $\lambda$. Since $(\phi(f_n))^2=-m=-2k-z$, we have that $k+\lambda^2=2k+z$, implying that $k+z=\lambda^2$.
Using the notation of Theorem \ref{thm:dinvts}, we have
\[ A^T=
    \left[\begin{matrix}
    1  &-1   &    &    && 0\\
    -1 & -1   &     &    & & 1\\
    & &  \ddots && & \vdots   \\
    && &1&-1& 0   \\
    &   &&-1&-1& 1 \\
    &  & & &&   \lambda  
    \end{matrix}\right].
\]
Let $S=\{x\in\Z^{2k+1}:x_i=\pm 1\}$. We claim that $|S/\im{2A^T}|\le 2^{k+1}$. 
Let $B$ be the subgroup of $\ZZ^{2k+1}$ generated by the first $2k$ columns of $2A^T$ and let $E=(\ZZ/2\ZZ)^{k}\oplus(\ZZ/4\ZZ)^k\oplus\ZZ$. 
Define the map $l:\ZZ^{2k+1}\to E$ by $l=(l_1,\ldots,l_{2k+1})$, where 
\begin{align*}
    l_{2i-1}(x)&=[x_{2i-1}]_2&\text{ for all } 1\le i \le k,\\
    l_{2i}(x)&=[x_{2i-1}+x_{2i}]_4&\text{ for all } 1\le i \le k,\\
    l_{2k+1}(x)&=x_{2k+1}.
\end{align*}

We claim that $\ker l=B$. Indeed, it is easy to see that $B\subseteq\ker l$. We now show the reverse inclusion. 
If $x=x_1e_1\cdots+ x_{2k+1}e_{2k+1}\in\ker l$, it follows that $x_{2k+1}=0$ and that there exist integers $\alpha_1,\ldots,\alpha_m$ satisfying $x_{2i-1}=2\alpha_{2i-1}$ and $x_{2i-1}+x_{2i}=4\alpha_{2i}$ for all $1\le i\le k$.
Hence $x_{2i-1}-x_{2i}=4\alpha_{2i-1}-4\alpha_{2i}$ for all $1\le i\le k$. It follows that
\begin{equation*}
\begin{split}
	x&= x_1e_1+\cdots +x_{2k+1}e_{2k+1}\\
        &=\sum_{i=1}^k\frac{x_{2i-1}-x_{2i}}{2}(e_{2i-1}-e_{2i})-\frac{x_{2i-1}+x_{2i}}{2}(-e_{2i-1}-e_{2i})\\
        &=\sum_{i=1}^k2(\alpha_{2i-1}-\alpha_{2i})(e_{2i-1}-e_{2i})-2\alpha_{2i}(-e_{2i-1}-e_{2i})\in B.
\end{split}
\end{equation*}

Note that when restricted to $S$, the functions $l_{2i-1}$, where $1\le i\le k$, are constant. On the other hand, the functions $l_{2i}$ and $l_{2k+1}$, where $1\le i\le k$, each take on two values when restricted to $S$. It follows that when restricted to $S$, $l$ takes on $2^{k+1}$ distinct values; hence $|S/\im(2A^T)|\le|S/\im B|=2^{k+1}$.

Using the formula in Section \ref{sec:background}, the determinant of $L=P(1^{[z]},2^{[2k]},)$ is $(z+k)2^{2k}$. Hence $\sqrt{\det L}=\sqrt{z+k} 2^k$.
By Theorem \ref{thm:dinvts}, we have that $\sqrt{z+k} 2^k\le 2^{k+1}$, or $z+k\le 4$. Since $z+k=\lambda^2$, we have that $z+k\in\{1,4\}$. Moreover, since $m=2k+z\ge 3$, and $k\ge1$, it follows that $z+k=4$. 
The result follows.
\end{proof}

\begin{prop}
If $k_{\phi}\neq0$, $m=2k_\phi+z_\phi+1$, and $L$ is a $\chi-$slice positive pretzel link, then $L$ is isotopic to $P(1^{[z]},2^{[2k]},k+z\pm2)$, where $k\ge1$ and $z\ge0$ are integers satisfying $z+2k\ge2$ and $k+z\pm2\ge1$.
\label{prop:knot0}
\end{prop}

\begin{proof}
Set $k:=k_{\phi}$ and $z:=z_\phi$. Since $m=2k+z+1$, $L$ is isotopic to a link in the set $P\langle1^{[z]},2^{[2k]},t+1\rangle$, where $t=|C_m|\ge1$. By Lemma \ref{lem:flype}, $L$ is necessarily isotopic to $P(1^{[z]},2^{[2k]},t+1)$. Note that $n=2k+t+1$. 
Using the conventions in the proof of Lemma \ref{lem:pos1}, we have that we have that $\phi(f_{2i-1})=e_{2i-1}-e_{2i}$ and $\phi(f_{2i})=-e_{2i-1}-e_{2i}$ for all $1\le i\le k_{\phi}$. 
Let $C_m=\{f_{2k+1},\ldots,f_{2k+t}\}$, where $f_{2k+t}$ is the basis element corresponding to the vertex adjacent to the central vertex. Then $n=2k+t+1$ and the central vertex corresponds to the basis element $f_n$.

By Proposition \ref{prop:facts}(\ref{1}), up to the action of $\Aut\ZZ^n$, we have that $\phi(f_{2k+j})=e_{2k+j}-e_{2k+j+1}$ for $1\le j\le t$. Since $\phi(f_{n})\cdot \phi(f_i)=1$ for all $1\le i\le 2k$, up to the action of $\Aut \ZZ^n$, $\phi(f_n)=e_2+e_4+\cdots+e_{2k}+z$ for some vector $z\in\text{span}\{e_{2k+1},e_{2k+2},\ldots,e_{2k+t+1}\}$. Since $\phi(f_n)\cdot\phi(f_{2k+t})=1$, it follows that 
$$\phi(f_n)=e_2+e_4+\cdots+e_{2k}+\lambda e_{2k+1}+(\lambda+1)(e_{2k+2}+\cdots+e_{2k+t})$$
for some integer $\lambda$. Thus $(\phi(f_n))^2=-(k+\lambda^2+(\lambda+1)^2t)=-(\frac{m-z-1}{2}+\lambda^2+(\lambda+1)^2t)$. Since $(\phi(f_n))^2=-m$, it follows that $m=2\lambda^2+2(\lambda+1)^2t-z-1$.

Using the notation of Theorem \ref{thm:dinvts}, we have
\[ A^T=
    \left[\begin{matrix}
    1  &-1   &    &    && &&&&0\\
    -1 & -1   &     &    &&&&& & 1\\
    & &  \ddots && & &&&&\vdots   \\
    && &1&-1&& && & 0   \\
    &   &&-1&-1&&&&& 1 \\
    &  & & &&  1& &&&\lambda     \\
    && &&  &   -1    &  1  &&&\lambda+1   \\
    && && &        & -1  &&&\\
    && && &        &   &\ddots&&\vdots \\
    && && &       && & 1  &\\
    && && &        &&& -1  &\lambda+1\\
    \end{matrix}\right].
\]

Let $S=\{x\in\Z^{2k+t+1}:x_i=\pm 1\}$. We claim that $|S/\im{2A^T}|\le 2^{k}(t+2)$. 
Let $B$ be the subgroup of $\ZZ^{2k+t+1}$ generated by the first $2k+t$ columns of $2A^T$ and let $E= (\ZZ/2\ZZ)^{k}\oplus(\ZZ/4\ZZ)^k\oplus\ZZ$. Define the map $l:\ZZ^{2k+t+1}\to E$ by $l=(l_1,\ldots,l_{2k+t+1})$, where 
\begin{align*}
    l_{2i-1}(x)&=[x_{2i-1}]_2&\text{ for all } 1\le i \le k,\\
    l_{2i}(x)&=[x_{2i-1}+x_{2i}]_4&\text{ for all } 1\le i \le k,\\
    l_{2k+1}(x)&=x_{2k+1}+\cdots+x_{2k+t+1}.   
\end{align*}

Arguing as in the proof of Proposition \ref{prop:knot-1}, one can show that $\ker l=B$.
Note that when restricted to $S$, the functions $l_{2i-1}$, where $1\le i\le k$, and $l_{2k+j}$, where $1\le j\le t+1$, are constant. On the other hand, the functions $l_{2i}$, where $1\le i\le k$, each take on two values when restricted to $S$. Finally, the function $l_{2k+t+1}$ takes on $t+2$ values when restricted to $S$. It follows that when restricted to $S$, $l$ takes on $2^k(t+2)$ distinct values; hence $|S/\im(2A^T)|\le2^k(t+2)$.

Using the formula in Section \ref{sec:background}, the determinant of $L=P(1^{[z]},2^{[2k]},t+1)$ is given by 
\begin{equation*}
\begin{split}
    \det L &= z(t+1)2^{2k}+(2k)(t+1)2^{2k-1}+2^{2k}\\
           &= 2^{2k}((t+1)(z+k)+1)\\  
           &=2^{2k}\Big((t+1)(\lambda^2+(\lambda+1)^2t-1)+1\Big)\\
           &= (2^{k}(\lambda(t+1)+t))^2.
\end{split}
\end{equation*}

By Theorem \ref{thm:dinvts}, we have that $|2^k(\lambda(t+1)+t)|\le 2^k(t+2)$.
It is clear that  $-2\le \lambda\le 2$. 
If $\lambda=-1$, then $m\le1$, which is a contradiction.
Let $\lambda=0$. Then $m=2t-z-1$ and so $t+1=\frac{m+3}{2}=k+z+2$; hence $L$ is of the form given in the statement of the proposition.
Let $\lambda=-2$. Then $m=7+2t\ge7$, implying that $t+1=\frac{m-5}{2}=k+z-2$; hence $L$ is of the form given in the statement of the proposition.
Let $\lambda=1$; then it is clear that $t\in\{0,1\}$. If $t=0$, then $m=1-z\le 1$, which is a contradiction. Thus $t=1$ and so $m=9-z$. Since $m=2k+z+1$, we have that $2k+2z+1=9$ or $k+z=4$; hence $t+1=2=k+z-2$ and so $L$ is of the form given in the statement of the proposition. Finally, let $\lambda=2$. Then $t=0$ and so $m=7-z$. Since $m=2k+z+1$, we have that $k+z=3$; hence $t+1=1=k+z-2$ and so $L$ is of the form given in the statement of the proposition. Finally, since $L$ has at least three parameters, each of which is positive, we have that $z+2k\ge2$ and $k+z\pm2\ge1$.
\end{proof}

We are now read to prove Theorem \ref{thm:mainpositivepretzellinks}, which we recall here for convenience.

\positive*

\begin{proof}[Proof of Theorem \ref{thm:mainpositivepretzellinks}]
If $L$ is of the form given in the statement of Theorem \ref{thm:mainpositivepretzellinks}, then $L$ $\chi-$ribbon (and hence $\chi-$slice) by Proposition \ref{prop:pospretzelslice}. Now suppose that $L$ is an arbitrary positive pretzel link. Using the notation introduced at the beginning of Section \ref{sec:positive}, we have two cases to consider: $k_\phi=0$ and $k_\phi\neq0$. If $k_\phi=0$, then by Lemma \ref{lem:k=0}, $L$ is isotopic to $P(1^{[z]},z\pm2)$, which is of the form listed in the statement of Theorem \ref{thm:mainpositivepretzellinks} with $k=0$. If $k_\phi\neq0$, then by Lemma \ref{lem:pos1}, $m\in \{2k_\phi+z,2k_\phi+z+1\}$. If $m=2k_\phi+z$, then by  Proposition \ref{prop:knot-1}, $L$ is isotopic to $P(1^{[z']},2^{[2k]})$, where $k+z'=4$. If $z'=0$, we obtain $P(2,2,2,2,2,2,2,2)$; if $z'>0$, setting $z=z'-1$, we have that $P(1^{[z']},2^{[2k]})=P(1^{[z]},2^{[2k]},1)$ which is of the form listed in the statement of Theorem \ref{thm:mainpositivepretzellinks} with $z+k=3$. Finally, if $m=2k_\phi+z+1$, then by Proposition \ref{prop:knot0}, $L$ is isotopic to $P(1^{[z]},2^{[2k+1]},k+z\pm2)$, which is of the form listed in the statement of Theorem \ref{thm:mainpositivepretzellinks}.
\end{proof}

\section{3-stranded pretzel links}\label{sec:3}

Let $P(p,q,r)$ be a 3-stranded pretzel link.
The goal of this section is to prove Theorem \ref{thm:main3strandeda}.
Note that $P(p,q,r)$ is $\chi-$slice if and only if its mirror $-P(p,q,r)=P(-p,-q,-r)$ is $\chi-$slice. In light of this observation and the discussion in Section \ref{sec:background}, there are only three cases to consider:
\begin{enumerate}
    \item[(A)] $p,q,r>0$;
    \item[(B)] $p,q>0$, $r<0$, and $\frac{1}{p}+\frac{1}{q}+\frac{1}{r}>0$; and
    \item[(C)] $p>0$, $q,r<0$, and $\frac{1}{p}+\frac{1}{q}+\frac{1}{r}>0$.
\end{enumerate}

\begin{figure}[h]
    \centering

    \subfloat[Subfigure 1 list of figures text][$p,q,r>0$ ]{\begin{tikzpicture}[dot/.style = {circle, fill, minimum size=1pt, inner sep=0pt, outer sep=0pt}]
\tikzstyle{smallnode}=[circle, inner sep=0mm, outer sep=0mm, minimum size=2mm, draw=black, fill=black];

\node[smallnode, label={90:$-3$}] (0) at (0,0) {};

\node[smallnode, label={90:$-2$}] (p1) at (-1,0) {};
\node[smallnode, label={90:$-2$}] (p2) at (-2,0) {};

\node[smallnode, label={180:$-2$}] (q1) at (0,-1) {};
\node[smallnode, label={180:$-2$}] (q2) at (0,-2) {};

\node[smallnode, label={90:$-2$}] (r1) at (1,0) {};
\node[smallnode, label={90:$-2$}] (r2) at (2,0) {};
\draw[-] (0) -- (p1);
\draw[loosely dotted, thick] (p1) -- (p2);
\draw[decorate,decoration={brace,amplitude=5pt,raise=.7cm,mirror},yshift=0pt] (p1) -- (p2) node [midway,yshift=1.1cm]{$p-1$};
\draw[-] (0) -- (q1);
\draw[loosely dotted, thick] (q1) -- (q2);
\draw[decorate,decoration={brace,amplitude=5pt,raise=.9cm,mirror}] (q1) -- (q2) node [midway,xshift=-1.7cm]{$r-1$};
\draw[-] (0) -- (r1);
\draw[loosely dotted, thick] (r1) -- (r2);
\draw[decorate,decoration={brace,amplitude=5pt,raise=.7cm,mirror},yshift=0pt] (r2) -- (r1) node [midway,yshift=1.1cm]{$q-1$};
\end{tikzpicture}

    \label{fig:3stranded:a}
    }
    \quad
    \subfloat[Subfigure 2 list of figures text][$p,q>0$, $r<0$]{\begin{tikzpicture}[dot/.style = {circle, fill, minimum size=1pt, inner sep=0pt, outer sep=0pt}]
\tikzstyle{smallnode}=[circle, inner sep=0mm, outer sep=0mm, minimum size=2mm, draw=black, fill=black];

\node[smallnode, label={90:$-2$}] (0) at (0,0) {};

\node[smallnode, label={90:$-2$}] (p1) at (-1,0) {};
\node[smallnode, label={90:$-2$}] (p2) at (-2,0) {};

\node[smallnode, label={180:$r$}] (q1) at (0,-1) {};

\node[smallnode, label={90:$-2$}] (r1) at (1,0) {};
\node[smallnode, label={90:$-2$}] (r2) at (2,0) {};

\draw[-] (0) -- (p1);
\draw[loosely dotted, thick] (p1) -- (p2);
\draw[decorate,decoration={brace,amplitude=5pt,raise=.7cm,mirror},yshift=0pt] (p1) -- (p2) node [midway,yshift=1.1cm]{$p-1$};

\draw[-] (0) -- (q1);

\draw[-] (0) -- (r1);
\draw[loosely dotted, thick] (r1) -- (r2);
\draw[decorate,decoration={brace,amplitude=5pt,raise=.7cm,mirror},yshift=0pt] (r2) -- (r1) node [midway,yshift=1.1cm]{$q-1$};
\end{tikzpicture}
    \label{fig:3stranded:b}}
    \vspace{15pt}
    \qquad
    \subfloat[Subfigure 3 list of figures text][$p>0$, $q,r<0$]{\begin{tikzpicture}[dot/.style = {circle, fill, minimum size=1pt, inner sep=0pt, outer sep=0pt}]
\tikzstyle{smallnode}=[circle, inner sep=0mm, outer sep=0mm, minimum size=2mm, draw=black, fill=black];

\node[smallnode, label={90:$-1$}] (0) at (0,0) {};

\node[smallnode, label={90:$-2$}] (p1) at (-1,0) {};
\node[smallnode, label={90:$-2$}] (p2) at (-2,0) {};

\node[smallnode, label={180:$r$}] (q1) at (0,-1) {};

\node[smallnode, label={90:$q$}] (r1) at (1,0) {};

\draw[-] (0) -- (p1);
\draw[loosely dotted, thick] (p1) -- (p2);
\draw[decorate,decoration={brace,amplitude=5pt,raise=.7cm,mirror},yshift=0pt] (p1) -- (p2) node [midway,yshift=1.1cm]{$p-1$};

\draw[-] (0) -- (q1);
\draw[-] (0) -- (r1);
\end{tikzpicture}
    \label{fig:3stranded:c}}

    \caption{The plumbing $X(p,q,r)$}
    \label{fig:X(p,q,r)}

\end{figure}
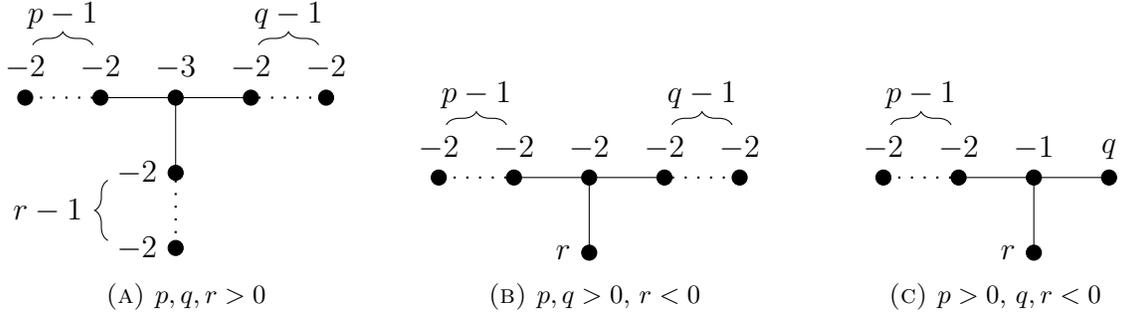

Figure \ref{fig:X(p,q,r)} shows the plumbing diagrams of the negative definite plumbing $X(p,q,r)$ in these three cases.
In light of Lemma \ref{lem:2-bridgechi2}, we will assume that $|p|,|q|,|r|>1$.

\subsection{Case A}
The case of $p,q,r>1$ follows from Theorem \ref{thm:mainpositivepretzellinks}.

\begin{lem}
Let $p,q,r>1$. If $L=P(p,q,r)$ is $\chi-$slice, then $L$ is isotopic to $P(2,2,3).$
    \label{lem:3pos}
\end{lem}

\subsection{Case B} 
In this case, we consider pretzel links $P(p,q,r)$ with $p,q>1$ and $r<-1$.  Consider Figure \ref{fig:3stranded:b}. Set $X:=X(p,q,r)$. 
Put $n=p+q$ and let $\phi:(\Z^n,Q_X)\to(\Z^n,-I)$ be a lattice embedding. Let $\{f_1,\dots,f_n\}$ be the standard basis of the domain of $\phi$ corresponding to the vertices in Figure \ref{fig:3stranded:b} such that $f_1,\dots,f_{p-1}$ correspond to the $p-1$ vertices with weight $-2$ going left-to-right, $f_p,\dots,f_{p+q-2}$ correspond to the $q-1$ vertices with weight $-2$ going right-to-left, $f_{p+q-1}$ corresponds to the vertex with weight $r$, and $f_{n}$ corresponds to the central vertex. Let $C_1$ denote the chain of the $p-1$ vertices of weight $-2$, let $C_2$ denote the chain of the $q-1$ vertices of weight $-2$.
Let $\{e_1,\dots,e_n\}$ be the standard basis of the codomain of $\phi$. We consider two cases: $|U_{C_1}\cap U_{C_2}|\neq 0$ and $|U_{C_1}\cap U_{C_2}|= 0$.

\begin{prop}
If $|U_{C_1}\cap U_{C_2}|\neq 0$, then $p=q=2$ and $r\in\{-2,-5\}$.
\label{prop:intersect}
\end{prop}

\begin{proof}
Since $|U_{C_1}\cap U_{C_2}|\neq 0$, by Proposition \ref{prop:facts}(\ref{3}), $p=q=2$ and up to the action of $\Aut \ZZ^n$, we may assume that $\phi(f_1)=e_1-e_2$ and $\phi(f_2)=-e_1-e_2$. Note that $n=4$. Since $\phi(f_4)\cdot\phi(f_1)=\phi(f_4)\cdot\phi(f_2)=1$ and $(\phi(f_4))^2=-2$, up to the action of $\Aut \ZZ^n$, $\phi(f_4)=e_2-e_3$. Finally, since $\phi(f_3)\cdot\phi(f_1)=\phi(f_3)\cdot\phi(f_2)=0$, it follows that $\phi(f_3)=e_3+\lambda e_4$ for some integer $\lambda\neq0$. Hence $r=-(1+\lambda^2)$.

Using the notation of Theorem \ref{thm:dinvts}, we have that
\[ A^T=
    \left[\begin{matrix}
    1 & -1 & 0 & 0\\
    -1 & -1 &0 & 1\\
    0 & 0 & 1& -1\\
    0 & 0 & \lambda& 0\\
    \end{matrix}\right]
\]
Note that $|\det A^T|=2|\lambda|$. Let $B$ be the subgroup of $\ZZ^n$ generated by the first, second, and last columns of $2A^T$ and let $E=\ZZ/2\ZZ\oplus\ZZ/2\ZZ\oplus\ZZ/4\ZZ\oplus\ZZ$. Define the map $l:\ZZ^{4}\to E$ by 
$$l(x_1,x_2,x_3,x_4)=([x_3]_2,[x_2+x_3]_2,[x_1+x_2+x_3]_4,x_4).$$
We claim that $\ker l=B$. Indeed, it is easy to see that $B\subseteq\ker l$. We now show the reverse inclusion. 
If $x=x_1e_1+x_2e_2+x_3e_3+x_4e_4\in\ker l$, then $x_4=0$, $x_3=2k_3$, $(x_2+x_3)=2k_2$, and $x_1+x_2+x_3=4k_1$, for some integers $k_1,k_2,k_3\in\ZZ$. Hence 
\begin{equation*}
\begin{split}
	x&= x_1e_1+x_2e_2+x_3e_3+x_4e_4\\
 &=-x_3(e_2-e_3)-(x_2+x_3)(e_1-e_2)+(x_1+x_2+x_3)e_1\\
	&=-2k_3(e_2-e_3)-2k_2(e_1-e_2)+4k_1e_1\\
	&=-2k_3(e_2-e_3)-(2k_2-2k_1)(e_1-e_2)-2k_1(-e_1-e_2)\in B.
\end{split}
\end{equation*}

Set $S=\{(x_1,x_2,x_3,x_4)\in\ZZ^4\,|\,x_i=\pm1\}$. Note that when restricted to $S$: the first two components of $l$ are constant and the last two components of $l$ takes on 2 distinct values each. Thus $l|_S$ takes on 4 distinct values. Using the notation of Theorem \ref{thm:dinvts}, it follows that $|S/\im(2A^T)|\le 4$.  Hence by Theorem \ref{thm:dinvts}, $|\det A^T|=2|\lambda|\le 4$, or $|\lambda|\in\{1,2\}$. The result follows.
\end{proof}

\begin{prop}[Proposition 3.7 in \cite{greenejabuka}]
If $|U_{C_1}\cap U_{C_2}|=0$, then either $p+q=0$, $q+r=0$, or $p+r=0$.
\label{prop:notintersect}
\end{prop}

\begin{remark}
Proposition 3.7 in \cite{greenejabuka} considers the case in which $p,q>1$, $r<-1$, and $p,q,r$ are odd. However, their proof holds for all $p,q>1$ and $r<-1$. Note that their analysis uses a version of Theorem \ref{thm:dinvts} that applies to knots (or more generally, links with odd determinant). In particular, they consider $|S/\im(A^T)|$ instead of $|S/\im(2A^T)|$. This does not go against Theorem \ref{thm:dinvts}, however; when the determinant of a link is odd (so the determinant of $A^T$ is odd), the sets $S/\im(A^T)$ and $S/\im(2A^T)$ coincide.
\end{remark}

To summarize, we have proven the following:

\begin{prop}
Let $p,q>1$ and $r<-1$. If $L=P(p,q,r)$ is $\chi-$slice, then $L$ is isotopic to either $P(2,2,-2)$, $P(2,2,-4)$, or $P(a,-a,b)$, where $a,b\ge1$.
    \label{prop:3,2pos}
\end{prop}

\subsection{Case C}
In this case, let $p>0$ and $q,r<0$, where $\frac{1}{p}+\frac{1}{q}+\frac{1}{r}>0.$ The plumbing diagram for $X:=X(p,q,r)$ is shown in Figure \ref{fig:3stranded:c}. Set $n=p+2$ and let $\phi:(\Z^n,Q_X)\to(\Z^n,-I)$ be a lattice embedding. Let $\{f_1,\dots,f_n\}$ be the standard basis of the domain of $\phi$ corresponding to the vertices in Figure \ref{fig:3stranded:c} such that $f_1,\dots,f_{p-1}$ correspond to the $p-1$ vertices with weight $-2$ going left-to-right, $f_{p}$ and $f_{p+1}$ correspond to the vertices of weight $q$ and $r$, respectively, and $f_{p+2}$ corresponds to the central vertex. 
As before, let $\{e_1,\dots,e_n\}$ be the standard basis of the codomain of $\phi$.

Recall the following set:
\[\mathcal{F}=\left\{\begin{array}{c c}
    \!\!\multirow{2}{*}{$P(a,-a-x_1^2-x_2^2,-a-y_1^2-y_2^2)\,\,:$ }   & \!\!\Big|\det\begin{bmatrix}x_1&y_1\\x_2&y_2\end{bmatrix}\Big|\le 4, \,\,\, x_1y_1+x_2y_2=-a,\\ 
    & \text{and exactly two parameters are even}
\end{array}\right\}\cup \mathcal{E}.\]

\begin{prop}
Suppose $p>1$, $q,r<-1$ and $\frac{1}{p}+\frac{1}{q}+\frac{1}{r}>0.$ If $P(p,q,r)$ is $\chi-$slice, then $P(p,q,r)$ belongs to the set $\mathcal{F}$.
\label{prop:c}
\end{prop}

\begin{proof}
By Proposition \ref{prop:facts}(\ref{1}), it is easy to see that (up to the action of $\Aut \ZZ^n$) $\phi(f_i)=e_i-e_{i+1}$ for all $1\le i\le p-1$, $\phi(f_{p+2})=e_{p}$, $\phi(f_{p})=-e_1-\cdots-e_{p}+x_1e_{p+1}+x_2e_{p+2}$, and $\phi(f_{p+1})=-e_1-\cdots-e_{p}+y_1e_{p+1}+y_2e_{p+2}$ for some integers $x_i,y_i,z_i\in\ZZ$ where $x_1y_1+x_2y_2=-p.$
Using the notation of Theorem \ref{thm:dinvts}, we have 

\[ A^T=
    \left[\begin{matrix}
    1  &    &    &    &&-1 &-1&      \\
    -1 & 1   &     &    &&& & \\
    & -1 &\ddots  &  && \vdots& \vdots&   \\
    &&\ddots & \ddots     &&\vdots& \vdots&   \\
    & &     &\ddots  &  1&& & \\
    &  &     & & -1 &-1 &-1  &  1 \\
    &&& &&          x_1  & y_1&  \\
    &&& &&           x_2  & y_2 &\\
    \end{matrix}\right]
\]
Note that $\det A^T=\det\begin{bmatrix}
x_1  & y_1  \\
     x_2  & y_2 
\end{bmatrix}$. Let $S=\{(z_1,\ldots,z_n)\in\ZZ^n\,|\,z_i=\pm1\}$. Let $B$ be the subgroup of $\ZZ^n$ generated by the first $p-1$ and last columns of $2A^T$. It is easy to see that the first $p$ columns of $A^T$ form a basis for $\ZZ^{p}=\langle e_1,\ldots,e_p\rangle\subset \ZZ^{n}$. Consequently, if $(x_1,\ldots,x_n),(y_1,\ldots,y_n)\in S$ such that $x_{n-1}=y_{n-1}$ and $x_n=y_n$, then $x-y\in B$. It follows that $|S/\im(2A^T)|\le 4$. By Theorem \ref{thm:dinvts}, $|\det A^T|\le 4$.

Finally, by Lemma \ref{lem:3componentobstruct}, $P(p,q,r)$ must have at most 2 components, implying that at most two of the parameters are even. If exactly one parameter is even, then $P(p,q,r)$ is a knot and by Theorem \ref{thm:3strandedknots}, $P(p,q,r)$ belongs to the set $\mathcal{E}\subset\mathcal{F}$. If exactly two of the parameters are even, then $P(p,q,r)$ belongs to the first set comprising $\mathcal{F}$.
\end{proof}

\subsection{Proof of Theorem \ref{thm:main3strandeda}}
We are now ready to prove Theorem \ref{thm:main3strandeda}. For convenience, we recall it here.

\threestrandeda*
\begin{proof}
Let $L$ be a 3-stranded pretzel link.
If $L$ is $\chi-$slice, then by Lemmas \ref{lem:2-bridgechi2} and \ref{lem:3pos}, and Propositions \ref{prop:3,2pos} and \ref{prop:c}, $L$ or $-L$ belongs to one of the sets listed in the statement of Theorem \ref{thm:main3strandeda} or $\mathcal{F}$. Consequently, if $L\not\in\mathcal{F}$ and $-L\not\in\mathcal{F}$, then $L$ or $-L$ belongs to one of the sets listed in the statement of the theorem. Conversely, all pretzel links listed are $\chi-$ribbon by Proposition \ref{prop:3-strand}. 
\end{proof}


\section{4-stranded pretzel links}\label{sec:4}

Let $P(p, q, r, s)$ be a 4-stranded pretzel link. Once again, $P(p,q,r,s)$ is $\chi-$slice if and only if its mirror $-P(p,q,r,s)=P(-p,-q,-r,-s)$ is $\chi-$slice. In light of the discussion in Section \ref{sec:unordered}), we have four cases to consider:
\begin{enumerate}[(A)]
    \item $p,q,r,s > 0$;
    \item $p,q,r > 0$, $s < 0$ and $\frac{1}{p} + \frac{1}{q} + \frac{1}{r} + \frac{1}{s} > 0$; \item $p,q > 0$, $r,s < 0$ and $\frac{1}{p} + \frac{1}{q} + \frac{1}{r} + \frac{1}{s} > 0$; and
    \item $p,q,r < 0$, $s > 0$ and $\frac{1}{p} + \frac{1}{q} + \frac{1}{r} + \frac{1}{s} > 0$.
\end{enumerate}

\subsection{Case A}
The case of $p,q,r,s>0$ follows from Theorem \ref{thm:mainpositivepretzellinks}.

\begin{lem}
Suppose $p,q,r,s>0$. If $L=P(p,q,r,s)$ is $\chi-$slice, then $L$ is isotopic to either $P(1,1,1,1)$ or $P(1,1,1,5).$
\label{lem:4pos}
\end{lem}

\subsection{Case B}

In this case, we consider pretzel links $P(p,q,r,s)$ with $p,q,r>0$ and $s<0$. The plumbing $X(p,q,r,s)$ is shown in Figure \ref{fig:3pos1neg}.

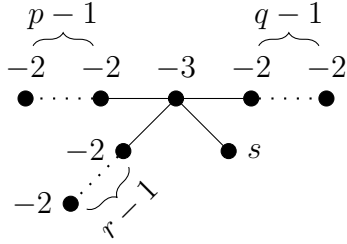
\begin{figure}[ht]
    \centering
    \begin{tikzpicture}[dot/.style = {circle, fill, minimum size=1pt, inner sep=0pt, outer sep=0pt}]
\tikzstyle{smallnode}=[circle, inner sep=0mm, outer sep=0mm, minimum size=2mm, draw=black, fill=black];

\node[smallnode, label={90:$-3$}] (-3) at (0,0) {};
\node[smallnode, label={90:$-2$}] (p1) at (-2,0) {};
\node[smallnode, label={90:$-2$}] (p2) at (-1,0) {};
\node[smallnode, label={90:$-2$}] (q1) at (2,0) {};
\node[smallnode, label={90:$-2$}] (q2) at (1,0) {};
\node[smallnode, label={180:$-2$}] (r1) at (-1.4,-1.4) {};
\node[smallnode, label={180:$-2$}] (r2) at (-0.7,-0.7) {};
\node[smallnode, label={0:$s$}] (s) at (0.7,-0.7) {};

\draw[-] (-3) -- (p2);
\draw[loosely dotted, thick] (p2) -- (p1);
\draw[decorate,decoration={brace,amplitude=5pt,raise=.7cm,mirror},yshift=0pt] (p2) -- (p1) node [midway,yshift=1.1cm]{$p-1$};
\draw[-] (-3) -- (q2);
\draw[loosely dotted, thick] (q2) -- (q1);
\draw[decorate,decoration={brace,amplitude=5pt,raise=.7cm,mirror},yshift=0pt] (q1) -- (q2) node [midway,yshift=1.1cm]{$q-1$};
\draw[-] (-3) -- (r2);
\draw[loosely dotted, thick] (r2) -- (r1);
\draw[decorate,decoration={brace,amplitude=5pt,raise=.2cm,mirror},yshift=0pt] (r1) -- (r2) node [midway, sloped, yshift=-0.6cm]{$r-1$};
\draw[-] (-3) -- (s);

\end{tikzpicture}

\caption{The plumbing $X(p,q,r,s)$, where $p,q,r>0$ and $s<0$}
    \label{fig:3pos1neg}
\end{figure}

In light of Lemma \ref{lem:4strands}, we will assume that at most one parameter is equal to 1. Hence, up to relabeling, we may assume that $p,q\ge2$ and $r\ge1$.
Set $n=p+q+r-1$ and assume $\phi:(\Z^n,Q_X)\to(\Z^n,-I)$ be a lattice embedding. Let $\{f_1,\dots,f_n\}$ be the standard basis of the domain of $\phi$ corresponding to the vertices in Figure 4 such that $f_1,\dots,f_{p-1}$ correspond to the $p-1$ vertices with weight $-2$ going left-to-right, $f_p,\dots,f_{p+q-2}$ correspond to the $q-1$ vertices with weight $-2$ going right-to-left, $f_{p+q-1},\dots,f_{n-2}$ correspond to the $r-1$ vertices with weight $-2$ going from bottom-to-top, $f_{n-1}$ corresponds to the vertex of weight $s$, and $f_{n}$ corresponds to the central vertex of weight $-3$. Let $C_1$ denote the chain of the $p-1$ vertices of weight $-2$, let $C_2$ denote the chain of the $q-1$ vertices of weight $-2$, and let $C_3$ denote the chain of the $r-1$ vertices of weight $-2$.
As before, let $\{e_1,\dots,e_n\}$ be the standard basis of the codomain of $\phi$.

The next three lemmas provide restrictions on the parameters under the standing assumption that there is a lattice embedding.

\begin{lem}
If $|U_{C_i}\cap U_{C_j}|=0$ for all $i\neq j$, then $r=1$
\label{lem:step_zero}
\end{lem}

\begin{proof}
Suppose $|U_{C_i}\cap U_{C_j}|=0$ for all $i\neq j$.
By Proposition \ref{prop:facts}(\ref{1}), we may assume, up to the action of $\Aut\ZZ^n$ that: $\phi(f_i)=e_i-e_{i+1}$ for each $i=1,\dots,p-1$; $\phi(f_j)=e_{j+1}-e_{j+2}$ for each $j=p,\dots,p+q-2$; and $\phi(f_k)=e_{k+2}-e_{k+3}$ for each $k=p+q-1,\dots,n-2$. In particular, $\phi(f_{n-2})=e_n-e_{n+1}$. But this is not possible since $\phi(f_{n-2})\in \ZZ^n=\langle e_1,\dots,e_n\rangle$.
Thus one of $p,q$ or $r$ must equal 1. Since our standing assumption is that $p,q\ge2$, we have that $r=1$.
\end{proof}

\begin{lem}
\label{lem:step_four}
Suppose $r=1$ (and $p,q\ge2$). Then up to relabeling, either
\begin{itemize}
    \item $p=q=2$ and $s=-\lambda^2-(\lambda+1)^2$ for some $\lambda\in\Z$, or
    \item $q=3$ and $s=-3\lambda^2-p(2\lambda-1)^2$ for some $\lambda\in\Z$.
\end{itemize}
\end{lem}

\begin{proof}
Note that $n=p+q$.
First suppose $|U_{C_1}\cap U_{C_2}|\neq 0$. Then by Proposition \ref{prop:facts}(\ref{2}), $|U_{C_1}|=|U_{C_2}|=1$. Consequently $p=q=2$ and so $n=4$. By Proposition \ref{prop:facts}(\ref{3}), $\phi(f_1)=e_1-e_2$ and $\phi(f_2)=-e_1-e_2$. The conditions $\phi(f_1)\cdot\phi(f_4)=\phi(f_2)\cdot\phi(f_4)=1$ and $\phi(f_4)\cdot\phi(f_4)=-3$ imply that (up to the action of $\Aut\ZZ^n$) $\phi(f_4)=e_2+e_3+e_4$. Finally, the conditions $\phi(f_1)\cdot\phi(f_3)=\phi(f_2)\cdot\phi(f_3)=0$ and $\phi(f_3)\cdot\phi(f_4)=1$ imply that $\phi(f_3)=\lambda e_3-(\lambda+1)e_4$ for some $\lambda\in\Z$. Consequently, $s=-\lambda^2-(\lambda+1)^2$.

Next suppose that $|U_{C_1}\cap U_{C_2}|=0$. Then by Proposition \ref{prop:facts}(\ref{1}), we may assume  that $\phi(f_i)=e_i-e_{i+1}$ for each $i=1,\dots,p-1$, and $\phi(f_j)=e_{j+1}-e_{j+2}$ for each $j=p,\dots,n-2$.
Since $\phi(f_{p-1})=e_{p-1}-e_p$, $\phi(f_{n-2})=e_{n-1}-e_{n}$, $\phi(f_{n})\cdot\phi(f_{p-1})=\phi(f_{n})\cdot\phi(f_{n-2})=1$, and $\phi(f_n)\cdot \phi(f_j)=0$ for all $j\not\in\{p-1,n-2\},$ it follows that $$\phi(f_{n})=\sum_{i=1}^{p-1}ae_i+(a+1)e_p+\sum_{i=p+1}^{n-1}be_i+(b+1)e_n$$
for some integers $a,b\in\ZZ$. Thus $-3=(\phi(f_{n}))^2=-(p-1)a^2-(a+1)^2-(q-1)b^2-(b+1)^2$. It follows that either $a=0$ and $b=-1$ or $a=-1$ and $b=0$. In the first case, we necessarily have that $q=3$ and in the latter case, we have that $p=3$. Assume the former case (the latter case is analogous). Then $q=3$, $n=p+3$, and $\phi(f_{p+3})=e_p-e_{p+1}-e_{p+2}$. Note that $f_{p+2}$ corresponds to the vertex with weight $s$. The conditions on $\phi(f_i)\cdot\phi(f_{p+2})$ for all $i$ imply that $\phi(f_{p+2})=(2\lambda-1)(e_1+\dots+e_p)+\lambda(e_{p+1}+e_{p+2}+e_{p+3})$ for some $\lambda\in\Z$. Consequently, $s=-3\lambda^2-p(2\lambda-1)^2.$
\end{proof}

\begin{lem}
\label{lem:step_six}
If $r\ge2$, then, up to relabeling, $q=r=2$ and $s=-p\lambda^2-(\lambda+1)^2$ for some $\lambda\in\Z$.
\end{lem}

\begin{proof}
Since $p,q,r\ge2$, by Lemma \ref{lem:step_zero}, $|U_{C_i}\cap U_{C_j}|\neq0$ for some $i\neq j\in\{1,2,3\}$. It then follows from Proposition \ref{prop:facts}(\ref{2}) that $|C_i|=|C_j|=1$. Up to relabeling, we may assume that $i=2$ and $j=3$ so that $q=r=2$. Note that $n=p+3$.
By Proposition \ref{prop:facts}(\ref{1}), $\phi(f_i)=e_i-e_{i+1}$ for each $i=1,\dots,p-1$. Moreover, by Proposition \ref{prop:facts}(\ref{3}), we may assume that $\phi(f_p)=e_{p+1}-e_{p+2}$ and $\phi(f_{p+1})=-e_{p+1}-e_{p+2}$. The conditions on $\phi(f_{p+3})$ imply that $\phi(f_{p+3})=e_p+e_{p+2}-e_{p+3}$ (unless $p=3$, in which case $\phi(f_6)=-e_1-e_2+e_5$ is also valid, but this contradicts the conditions on $\phi(f_5)$). The conditions on $\phi(f_{p+2})$ imply that $\phi(f_{p+2})=\lambda(e_1+\dots+e_p)+(\lambda+1)e_{p+3}$. Consequently, $s=-p\lambda^2-(\lambda+1)^2$.
\end{proof}

The next two propositions apply Theorem \ref{thm:dinvts} to put further restrictions on  the parameters under the assumption that $L=P(p,q,r,s)$ is $\chi-$slice. Recall that our standing assumption is that $p,q\ge2$ and $r\ge1$.

\begin{prop}\label{prop:d2}
Suppose $r=1$. If $L=P(p,q,r,s)$ is $\chi-$slice, then $L$ is isotopic to $P( 2,2,1,-1)$, $P(a,3,1,-a)$, or $P(a,3,1,-(a+3))$, where $a\ge2$.
\end{prop}

\begin{proof}
By Lemma \ref{lem:step_four}, we have two cases to consider. 
Consider the first case: $p=q=2$ and $s=-\lambda^2-(\lambda+1)^2$. From the proof of Lemma \ref{lem:step_four} that the columns of $A^T$ (as defined in Theorem \ref{thm:dinvts}) $e_1-e_2$, $-e_1-e_2$, $\lambda e_3-(\lambda+1)e_4$, and $e_2+e_3+e_4$. That is 
\[ A^T=
    \left[\begin{matrix}
    1  &  -1 & 0& 0      \\
    -1 & -1 & 0 & 1\\
    0 & 0 &\lambda& 1\\
    0 & 0 & -(\lambda+1)& 1
    \end{matrix}\right].
\]

Note that $|\det A^T|=|4\lambda+2|$. Let $B$ be the subgroup of $\ZZ^4$ generated by the first, second, and last columns of $2A^T$ and let $E=\ZZ/2\ZZ\oplus\ZZ/2\ZZ\oplus\ZZ/4\ZZ\oplus\ZZ$. Define the map $l:\ZZ^{4}\to E$ by 
$$l(x_1,x_2,x_3,x_4)=([x_3]_2,[x_2-x_3]_2,[x_1+x_2-x_3]_4,x_3-x_4).$$
We claim that $\ker l=B$. Indeed, it is easy to see that $B\subseteq\ker l$. We now show the reverse inclusion. 
If $x=x_1e_1+x_2e_2+x_3e_3+x_4e_4\in\ker l$, then $x_4=x_3=2k_3$, $x_2-x_3=2k_2$, and $x_1+x_2-x_3=4k_1$, for some integers $k_1,k_2,k_3\in\ZZ$. Hence 
\begin{equation*}
\begin{split}
	x&= x_1e_1+x_2e_2+x_3e_3+x_4e_4\\
 &=x_3(e_2+e_3)-(x_2-x_3)(e_1-e_2)+(x_1+x_2-x_3)e_1+x_4e_4\\
	&=2k_3(e_2+e_3)-2k_2(e_1-e_2)+4k_1e_1+2k_3e_4\\
	&=2k_3(e_2+e_3+e_4)-(2k_2-2k_1)(e_1-e_2)-2k_1(-e_1-e_2)\in B.
\end{split}
\end{equation*}

Set $S=\{(x_1,x_2,x_3,x_4)\in\ZZ^4\,|\,x_i=\pm1\}$. Note that when restricted to $S$: the first two components of $l$ are constant and the last two components of $l$ takes on 2 distinct values each. Thus $l|_S$ takes on 4 distinct values. Using the notation of Theorem \ref{thm:dinvts}, it follows that $|S/\im(2A^T)|\le 4$.  Hence by Theorem \ref{thm:dinvts}, $|\det A^T|=|4\lambda+1|\le 4$, implying that $\lambda\in\{-1,0\}$. It follows that $r=-1$ and $L$ belongs to the set $P\langle 2,2,1,-1\rangle$. Now by Lemma \ref{lem:flype}, it follows that $L$ is isotopic to $P(2,2,1,-1)$.

Next consider the second case: $q=3$ and $s=-3\lambda^2-p(2\lambda-1)^2$ for some $\lambda\in\Z$.
From the proof of Lemma \ref{lem:step_four} that the columns of $A^T$ (as defined in Theorem \ref{thm:dinvts}) are $e_1-e_2,\dots,e_{p-1}-e_p,e_{p+1}-e_{p+2},e_{p+2}-e_{p+3},(2\lambda-1)(e_1+\dots+e_p)+\lambda(e_{p+1}+e_{p+2}+e_{p+3}),e_p-e_{p+1}-e_{p+2}$. That is 
\[ A^T=
    \left[\begin{matrix}
    1  &    &    &    && &&  2\lambda-1  &       \\
    -1 & 1   &     &    && & &&\\
       & -1 &  &  &&& & \vdots &  \\
       && & \ddots     && & &\vdots&   \\
       & &     &  & 1  &&  &\\
       &  &     & & -1 &&& 2\lambda-1& 1      \\
    &&& && 1 &        &     \lambda &-1\\
    &&& && -1&1        &     \lambda&-1\\
      &&& && &-1        &      \lambda&\\
    \end{matrix}\right].
\]

Let $B$ be the subgroup of $\Z^{p+3}$ generated by the first $p+1$ and last columns of $2A^T$ and let $E=(\Z/2\Z)^{p+2}\oplus\Z$. Define $l:\Z^{p+3}\to(\Z/2\Z)^{p+2}\oplus\Z$ by
\begin{equation*}
	l(x)=(l_1(x),\dots,l_{p+3}(x)),
\end{equation*}
with $l_j(x)=[x_1+\dots+x_j]_2$ for each $j=1,\dots,p+1$, $l_{p+2}(x)=[x_{p+1}+x_{p+2}]_2$, and $l_{p+3}(x)=2(x_1+\dots+x_p)+x_{p+1}+x_{p+2}+x_{p+3}$. 

We claim that $B=\ker l$. It is clear $B\subseteq\ker l$. We now prove the reverse inclusion.
Suppose that $x=\sum_{j=1}^{p+3}x_je_j\in\ker l$, so that for each $j=1,\dots,p+1$, there is a $k_j\in\Z$ such that $x_1+\dots+x_j=2k_j$. Also note that $x_{p+1}+x_{p+2}=2k_{p+2}$ for some $k_{p+2}\in\Z$, and that $2(x_1+\dots+x_p)+x_{p+1}+x_{p+2}+x_{p+3}=0$. 
Let $c_j=\sum_{i=1}^jx_i$ for each $j=1,\dots,p+1$, $c_{p+2}=2\sum_{i=1}^px_i+x_{p+1}+x_{p+2}$, and $c_{p+3}=2\sum_{i=1}^px_i+x_{p+1}+x_{p+2}+x_{p+3}$. Then
\begin{equation*}
\begin{split}
	x=&\sum_{j=1}^{p-1}c_j(e_j-e_{j+1})+c_p(e_p-e_{p+1}-e_{p+2})\\
	&+c_{p+1}(e_{p+1}-e_{p+2})+c_{p+2}(e_{p+2}-e_{p+3})+c_{p+3}e_{p+3}\\
    =&\sum_{j=1}^{p-1}2k_j(e_j-e_{j+1})+2k_p(e_p-e_{p+1}-e_{p+2})\\
	&+2k_{p+1}(e_{p+1}-e_{p+2})+2(x_1+\cdots+x_p+k_{p+2})(e_{p+2}-e_{p+3})\in B.
\end{split}
\end{equation*}

Set $S=\{x\in\Z^{p+3}:x_i=\pm1\}$, and note that when restricted to $S$, the first $p+2$ components of $l$ are constant. Observe that when restricted to $S$, $l_{p+3}$ takes values between $-(2p+3)$ and $2p+3$, and that these values are all odd. Hence when restricted to $S$, $l$ takes at most $2p+4$ distinct values.

Using the formula in Section \ref{sec:background}, it is routine to check that $|\det L|=((4p+3)\lambda-2p)^2$.
By Theorem \ref{thm:dinvts}, $|(4\lambda-2)p+3\lambda|\le2p+4,$ implying that $\lambda=0,1$. It follows that $r=-p$ or $r=-(p+3)$. Hence $L$ belongs to either $P\langle a,3,1,-a\rangle$ or $P\langle a,3,1,-(a+3)\rangle$, where $a=p\ge2$. Now by Lemma \ref{lem:flype}, it follows that $L$ is isotopic to either $P( a,3,1,-a)$ or $P( a,3,1,-(a+3))$.
\end{proof}

\begin{prop}\label{prop:d1}
Suppose $r\ge2$. If $L=P(p,q,r,s)$ is $\chi$-slice, then $L$ is isotopic to a link in $P\langle 2,2,a,-a\rangle$ or $P\langle 2,2,a,-(a+4)\rangle$.
\end{prop}

\begin{proof}
Since $r\ge2$, by Lemma \ref{lem:step_six}, $q=r=2$ and $s=-p\lambda^2-(\lambda+1)^2$ for some $\lambda\in\Z$. Recall from the proof of Lemma \ref{lem:step_six} that the the columns of $A^T$ (as defined in Theorem \ref{thm:dinvts}) are $e_1-e_2,\dots,e_{p-1}-e_p,e_{p+1}-e_{p+2},e_{p+1}+e_{p+2},\lambda(e_1+\dots+e_p)+(\lambda+1)e_{p+3},e_p-e_{p+1}-e_{p+3}$. That is 
\[ A^T=
    \left[\begin{matrix}
    1  &    &    &    && &&  \lambda &        \\
    -1 & 1   &     &    &&& & &\\
       & -1 &  &  &&&  &\vdots &  \\
       && & \ddots     &&&  &\vdots   \\
       & &     &  & 1  &&  &\\
       &  &     & & -1 && &\lambda  & 1     \\
    &&& &&1  & 1            &  &-1\\
    &&& &&-1 &1             &&\\
    &&&   & &  &   & \lambda + 1&-1 \\
    \end{matrix}\right]
\]

Let $B$ be the subgroup of $\Z^{p+3}$ generated by the first $p+1$ and last columns of $2A^T$ and let $E=\Z^{p+3}/B\cong(\Z/2\Z)^{p+1}\oplus\Z/4\Z\oplus\Z$. Define $l:\Z^{p+3}\to E$ by
\begin{equation*}
	l(x)=(l_1(x),\dots,l_{p+3}(x)),
\end{equation*}
with $l_j(x)=[x_1+\dots+x_j]_2$ for each $j=1,\dots,p+1$, $l_{p+2}(x)=[x_1+\dots+x_{p+2}]_4$, and $l_{p+3}(x)=x_1+\dots+x_p+x_{p+3}$.

We claim that $\ker l=B$. It is clear that $B\subseteq\ker l$. We now show the reverse inclusion.
Suppose that $x=\sum_{j=1}^{p+3}x_je_j\in\ker l$, so that for each $j=1,\dots,p+1$, there is a $k_j\in\Z$ such that $x_1+\dots+x_j=2k_j$. Also note that $x_1+\dots+x_{p+2}=4k_{p+2}$ for some $k_{p+2}\in\Z$, and $x_1+\dots+x_p+x_{p+3}=0$. 
Let $c_j=\sum_{i=1}^jx_i$ for each $j=1,\dots,p+2$ and $c_{p+3}=x_{p+3}+\sum_{i=1}^px_i$. Then
\begin{equation*}
\begin{split}
        x=&\sum_{j=1}^{p-1}c_j(e_j-e_{j+1})+c_p(e_p-e_{p+1}-e_{p+3})+c_{p+1}(e_{p+1}-e_{p+2})\\
	&+c_{p+2}e_{p+2}+c_{p+3}e_{p+3}\\
	=&\sum_{j=1}^{p-1}2k_j(e_j-e_{j+1})+2k_p(e_p-e_{p+1}-e_{p+3})\\
	&+2k_{p+1}(e_{p+1}-e_{p+2})+4k_{p+2}e_{p+2}\\
	=&\sum_{j=1}^{p-1}2k_j(e_j-e_{j+1})+2k_p(e_p-e_{p+1}-e_{p+3})\\
	&+(2k_{p+1}-2k_{p+2})(e_{p+1}-e_{p+2})+2k_{p+2}(e_{p+1}+e_{p+2})\in B.
\end{split}
\end{equation*}

Set $S=\{x\in\Z^{p+3}:x_i=\pm 1\}$, and note that when restricted to $S$, the first $p+1$ components of $l$ are constant. Also note that when restricted to $S$, the last component of $l$ takes $p+2$ distinct values. Lastly, observe that when restricted to $S$, the second-to-last component of $l$ takes precisely 2 distinct values. It follows that when restricted to $S$, $l$ takes at most $2(p+2)$ distinct values.

Using the formula in Section \ref{sec:background}, it is routine to check that $|\det L|=4((p+1)\lambda+1)^2.$
By Theorem \ref{thm:dinvts}, we must have that
\begin{equation*}
	2|(p+1)\lambda+1|\le 2(p+2).
\end{equation*}
Since $p\ge 2$, we have that $\lambda\in\{-1,0,1\}$; the result follows.
\end{proof}

Combining the above results, we have the following.

\begin{thm}
Suppose exactly three parameters are positive and $\frac{1}{p}+\frac{1}{q}+\frac{1}{r}+\frac{1}{s}>0$. If $L=P(p,q,r,s)$ is $\chi-$slice then $L$ or $-L$ belongs to
 $$\{P(a,2,2,-1),P(a,3,1,-a),(a,3,1,-(a+3))\}\cup P\langle a,2,2,-a\rangle\cup P\langle a,2,2,-(a+4)\rangle$$

\label{thm:pqr>0}
\end{thm}

\begin{proof}
If at least two of the parameters are equal to 1, then by Lemma \ref{lem:4strands}, $L$ or $-L$ is isotopic to $P(1,1,-1,3)$ or $P(1,1,3,-4)$ (which, up to flyping (c.f. Lemma \ref{lem:flype}, are of the form $P(a,3,1,-a)$ and $P(a,3,1,-(a+3))$ with $a=1$). Next assume exactly one of the parameters is equal to 1. Then by Proposition \ref{prop:d2}, $L$ or $-L$ is isotopic to either $P(1,2,2,-1)$ (which belongs to $P\langle a,2,2,-a\rangle$ with $a=1$), $P(a,3,1,-a)$, or $(a,3,1,-(a+3))$. Finally, suppose none of the parameters equal 1. Then by Proposition \ref{prop:d1}, $L$ or $-L$ is isotopic to a link in $P\langle a,2,2,-a\rangle\cup P\langle a,2,2,-(a+4)\rangle$.
\end{proof}

\subsection{Case C}

In this case, let $p, q > 0$ and $r, s < 0$. Recall that in order to apply Donaldson's theorem, the intersection form must be negative definite so we will assume that $\frac{1}{p} + \frac{1}{q} + \frac{1}{r} + \frac{1}{s} > 0$. 
Figure \ref{fig:p,q>0} shows the plumbing diagram associated to $X:=X(p,q,r,s)$. 

\begin{figure}[h]
    \centering
    \begin{tikzpicture}[dot/.style = {circle, fill, minimum size=1pt, inner sep=0pt, outer sep=0pt}]
\tikzstyle{smallnode}=[circle, inner sep=0mm, outer sep=0mm, minimum size=2mm, draw=black, fill=black];

\node[smallnode, label={90:$-2$}] (-2) at (0,0) {};
\node[smallnode, label={90:$-2$}] (p1) at (-2,0) {};
\node[smallnode, label={90:$-2$}] (p2) at (-1,0) {};
\node[smallnode, label={90:$-2$}] (q1) at (2,0) {};
\node[smallnode, label={90:$-2$}] (q2) at (1,0) {};
\node[smallnode, label={180:$r$}] (r) at (-.7,-.7) {};
\node[smallnode, label={0:$s$}] (s) at (0.7,-0.7) {};

\draw[-] (-3) -- (p2);
\draw[loosely dotted, thick] (p2) -- (p1);
\draw[decorate,decoration={brace,amplitude=5pt,raise=.7cm,mirror},yshift=0pt] (p2) -- (p1) node [midway,yshift=1.1cm]{$p-1$};
\draw[-] (-3) -- (q2);
\draw[loosely dotted, thick] (q2) -- (q1);
\draw[decorate,decoration={brace,amplitude=5pt,raise=.7cm,mirror},yshift=0pt] (q1) -- (q2) node [midway,yshift=1.1cm]{$q-1$};
\draw[-] (-3) -- (r);
\draw[-] (-3) -- (s);

\end{tikzpicture}
\caption{$X=X(p,q,r,s)$, where $p,q>0$ and $r,s<0$}\label{fig:p,q>0}
\end{figure}
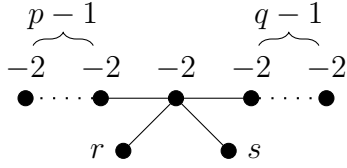

Let $X:=X(p,q,r,s)$ and suppose $Y(p,q,r,s)$ bounds a rational homology 4-ball.
By Theorem \ref{thm:donaldson}, there exists a lattice embedding $\phi : (\ZZ^n, Q_X) \to (\ZZ^n, -I)$, where $n = p + q +  1$. 
Let $C_1$ (resp. $C_2$) denote the chain of $p-1$ (resp. $q-1$) $-2$-weighted vertices in the plumbing diagram for $X$ shown in Figure \ref{fig:p,q>0}). 
By Proposition \ref{prop:facts}(\ref{2}), $U_{C_1}\cap U_{C_2}\neq\varnothing$ only if $p=q=2$.

\begin{prop} If $U_{C_1}\cap U_{C_2}\neq\varnothing$, then, up to switching $r$ and $s$, either:
\begin{itemize}
    \item $r=-6$ and $s=-3$;
    \item $r=-2$ and $s=-3$; or 
    \item $r=-2$ and $s=-6$.
\end{itemize}
\label{prop:p=q=2}
\end{prop}

\begin{proof}
Let $f_1$ and $f_2$ denote the basis vectors of $H_2(X)$ corresponding to the single vertices in $C_1$ and $C_2$. By Proposition \ref{prop:facts}(\ref{1}), up to the action of $\Aut\ZZ^n$, we may assume $\phi(f_1)=e_1-e_2$ and $\phi(f_2)=-e_1-e_2$.
Let $f_3$ and $f_4$ denote the basis vectors for $H_2(X)$ corresponding to the vertices with weights $r$ and $s$, respectively, and let $f_5$ denote basis vector for $H_2(X)$ corresponding to the central vertex of $X$. Since $\phi(f_5)\cdot\phi(f_1)=\phi(f_5)\cdot\phi(f_2)=1$, up to the action of $\Aut\ZZ^5$, we may assume that $\phi(f_5)=e_2-e_3$. Similarly, since $\phi(f_5)\cdot \phi(f_3)=\phi(f_5)\cdot \phi(f_2)=1$ and $\phi(f_1)\cdot \phi(f_3)=\phi(f_2)\cdot \phi(f_3)=\phi(f_1)\cdot \phi(f_4)=\phi(f_2)\cdot \phi(f_4)=0$, we necessarily have that $\phi(f_3)=e_3+ae_4+be_5$ and $\phi(f_4)=e_3+ce_4+de_5$ for some integers $a,b,c$ and $d$. Then $r = \phi(f_3) \cdot \phi(f_3) = -1-a^2-b^2$ and $s = \phi(f_4) \cdot \phi(f_4) = -1-c^2-d^2$. Finally, since $\phi(f_3)\cdot \phi(f_4)=1$, it follows that $ac+bd=-1$. 

Note that
the embedding matrix $H_2(X)\to \ZZ^n$ is given by 
\[ A^T=
    \left[\begin{matrix}
    1  & -1 & 0 & 0 & 0\\
    -1 & -1 & 0 & 0 & 1\\
    0 & 0 & 1 & 1 & -1\\
    0 & 0 & a & c & 0\\
    0 & 0 & b & d & 0\\
    \end{matrix}\right]
\]
It is easy to verify that $\det A^T=2(ad-bc)$.
We will now work towards applying Theorem \ref{thm:dinvts}. Let $B$ denote the matrix whose columns are the first, second, and last columns of $2A^T$ and let $S=\{(x_1,\ldots,x_5)\in\ZZ^5\,|\,x_i=\pm1\}$. It can be verified that $|S/\im(B)|=8$. It follows that  $|S/\im(2A^T)|\le 8$. Now by Theorem \ref{thm:dinvts}, $2|ad-bc|\le 8$. 

First assume that $a,b,c,d\neq0$. Since $ac+bd=-1$, it follows that either exactly one of $a,b,c$ or $d$ is positive or exactly one of $a,b,c$ or $d$ is negative. Hence $2|ad-bc|=2|ad|+2|bd|\le8$; consequently $|a|,|b|,|c|,|d|\le 3$. If any one of the variables is equal to $\pm3$, then the rest of the variables must be $\pm1$ which makes it impossible to satisfy $ac+bd=-1$. This implies that $|a|,|b|,|c|,|d|\le 2$. Once again since $ac+bd=-1$, we may assume without loss of generality that $a=-2$, and $b=c=d=1$. Consequently $r=-6$ and $s=-3$.

Next assume that $a=0$. Then since $ac+bd=-1$, we have $b=-d\in\{\pm1\}$. It can once again be checked that $|S/\im(2A^T)|=4$. Hence by Theorem \ref{thm:dinvts}, $2|ad-bc|=2|c|\le4$, or $|c|\le 2$. Consequently, $r=-2$ and $s\in\{-3,-6\}$. A similar argument shows similar results if we assume $b,c$ or $d$ equal 0 (with the values of $r$ and $s$ possibly swapped).
\end{proof}

\begin{lem}
\label{lem:embedding_rsab}
Suppose $U_{C_1}\cap U_{C_2}=\varnothing$.
Then
\begin{align}
    r &= -p\lambda^2 - q(\lambda + 1)^2 - a^2 \text{ and} \label{lem:Case4r}\\
    s &= -p\mu^2 - q(\mu + 1)^2 - b^2 \label{lem:Case4s}.
\end{align}
for some $\lambda, \mu, a, b \in \ZZ$ such that
\begin{equation}
ab = -p\lambda\mu - q(\lambda + 1)(\mu + 1) \label{lem:Case4ab}.
\end{equation}
\end{lem}
\begin{proof}

Let $f_1,\ldots,f_{p-1}$ denote the basis vectors corresponding to the vectors in $C_1$ and let $f_{p}\cdots f_{p+q-2}$ denote the basis vectors corresponding to the vectors in $C_2$. Let $f_{p+q-1}$ and $f_{p+q}$ denote the basis vectors corresponding to the vertices with weights $r$ and $s$, respectively. Finally, let $f_{p+q+1}=f_n$ denote the basis vector corresponding to the central vertex of $X$. 
Using Proposition \ref{prop:facts}(\ref{1}), we may assume that $\phi(f_i)=e_i-e_{i+1}$ for all $1\le i\le p-1$, $\phi(f_j)=e_{j+1}-e_{j+2}$ for all $p\le j\le p+q-2$, and $\phi(f_n)=e_p-e_{p-1}$.

Next, label $\phi(f_{n-2}) = \Sigma_{i=1}^n \lambda_i e_i$ and $\phi(f_{n-1}) = \Sigma_{i=1}^n \mu_i e_i$. In order to figure out what $\phi(f_{n-2})$ and $\phi(f_{n-1})$ are equal to, we will use information from the plumbing diagram:
\begin{enumerate}
    \item $\phi(f_{n-2}) \cdot \phi(f_i) = 0$ for any $0 \le i \le p + q - 2$ and
    \item $\phi(f_{n-2}) \cdot \phi(f_n) = 1$.
\end{enumerate}

Since $\phi(f_i) = e_i - e_{i+1}$ for $1 \leq i \leq p - 1$, we can simplify (1) to say that $-(\lambda_i - \lambda_{i+1}) = 0$ so $\lambda_i = \lambda_{i+1}$. Then $\lambda_1 = \lambda_2 = \dots = \lambda_{p}$. This can be done in a similar fashion for $p \leq j \leq p + q - 1$ so that $\lambda_{p+1} = \dots = \lambda_{p+q-1}$ Next, (2) tells us that $-(\lambda_p - \lambda_{p+1}) = 1$, so $\lambda_{p+1} = \lambda_p + 1$. Note that we can derive the same equalities for $\phi(f_{n-1})$ so that $\mu_1 = \dots = \mu_{p}$, $\mu_{p+1} = \dots = \mu_{p+q-1}$, and $\mu_{p+1} = \mu_p + 1$.

If we let $\lambda = \lambda_1$ and $\mu = \mu_1$, we obtain a more specific form for $\phi(f_{n-2})$ and $\phi(f_{n-1})$. For convenience, we will also label $a = \lambda_{p+q+1}$ and $b = \mu_{p+q+1}$. This tells us that
\[
    \phi(f_{n-1}) = \sum_{i=1}^{p} \lambda e_i + \sum_{i=p+1}^{n-1} (\lambda + 1)e_i + ae_n.
\]
Using the notation of Theorem \ref{thm:dinvts},
\[ A^T=
    \left[\begin{matrix}
    1  &        &    &&& & \lambda     & \mu     & \\
    -1 &        &    &&& &&& \\
       & \ddots &    &&& & \vdots      & \vdots  & \\
       &        & 1  &&& &&& \\
       &        & -1 &&& & \lambda     & \mu     & 1  \\
    &&& 1  &        &    & \lambda + 1 & \mu + 1 & -1 \\
    &&& -1 &        &    &&& \\
    &&&    & \ddots &    & \vdots      & \vdots  & \\
    &&&    &        & 1  &&& \\
    &&&    &        & -1 & \lambda + 1 & \mu + 1 & \\
    &&&&&& a & b &
    \end{matrix}\right]
\]

Now, by the plumbing diagram again, $\phi(f_{n-2}) \cdot \phi(f_{n-2}) = r$ and $\phi(f_{n-1}) \cdot \phi(f_{n-1}) = s$. Then
\begin{align*}
    r &= -p\lambda^2 - q(\lambda + 1)^2 - a^2 \text{ and} \\
    s &= -p\mu^2 - q(\mu + 1)^2 - b^2.
\end{align*}
Finally, we know that $\phi(f_{n-2}) \cdot \phi(f_{n-1}) = 0$ so we get one last equation:
\[
    ab=-p\lambda\mu - q(\lambda + 1)(\mu + 1) .
\]
\end{proof}

Next, we have a lemma that gives us constraints for the $\lambda, \mu$ mentioned in Lemma \ref{lem:embedding_rsab}

\begin{lem}
\label{lem:lambdamu_slice}

Suppose $U_{C_1}\cap U_{C_2}=\varnothing$. Then, using the notation in Lemma \ref{lem:lambdamu_slice},
\begin{equation}
p(|b\lambda|+|a\mu|) + q(|b(\lambda + 1)|+|a(\mu + 1)|) \leq 2(p + q + 1). \label{eq:dbound}
\end{equation}
Further assume that $\lambda=0$. Then the right side of Equation (\ref{eq:dbound}) can be improved in the following cases: 
\begin{itemize}
    \item if $a=0$, then we can improve to $2q$; 
    \item if $|a|=1$, the we can improve to $p+2q+1$.
\end{itemize}
Similarly, if $\mu=-1$, then the right side of Equation (\ref{eq:dbound}) can be improved in the following cases: 
\begin{itemize}
    \item if $b=0$, then we can improve to $2p$; 
    \item if $|b|=1$, the we can improve to $2p+q+1$.
\end{itemize}
\end{lem}

\begin{proof}

Following \cite{greenejabuka} and applying Theorem \ref{thm:dinvts}, we will compute the $d-$invariant to further narrow down the possible values of $\lambda$ and $\mu$. 
As in the previous cases, we will apply Theorem \ref{thm:dinvts}; hence we aim to compute an upper bound of $|S/\text{Im}(2A^T)|$, where $S=\{(x_1,\ldots,x_{p+q+1})\in\ZZ^{p+q+1}:x_i=\pm1\}$. 

Consider the linear map $l : \ZZ^{p+q+1} \to \ZZ^2$ where
\[
l(x_1, \dots, x_{p+q+1}) = (x_1 + \dots + x_{p+q}, x_{p+q+1}).
\]

Note that the kernel of $l$ is contained in the image of $2A^T$. It is easy to check that the difference of two vectors $(x_1,\ldots,x_{p+q+1}),(y_1,\ldots,y_{p+q+1})\in S$, where $x_{p+q+1}=y_{p+q+1}$, is in $\text{ker}(l)\subset\im(2A^T)$ if $\#\{i\,|\,x_i<0\}=\#\{i\,|\,y_i<0\}$. Hence there are at most $2(p+q+1)$ equivalence classes of vectors in $S/\text{Im}(2A^T)$; consequently, $|S/\text{Im}(2A^T)|\le2(p+q+1)$

On the other hand, using Equation \ref{lem:Case4ab}, the determinant of $L$ is
\begin{equation*}
    \begin{split}
        |\det(L)| &= |pqr + pqs + prs + qrs|\\
                &= |pq(-p\lambda^2-q(\lambda+1)^2-a^2) + pq(-p\mu^2-q(\mu+1)^2-b^2)\\
                &\hspace{.5cm}+p(-p\lambda^2-q(\lambda+1)^2-a^2)(-p\mu^2-q(\mu+1)^2-b^2)\\
                &\hspace{.5cm}+q(-p\lambda^2-q(\lambda+1)^2-a^2)(-p\mu^2-q(\mu+1)^2-b^2)|\\
                &=|p(b\lambda - a\mu) + q(b(\lambda + 1) - a(\mu + 1))|^2.
    \end{split}
\end{equation*}

\noindent Thus by Theorem \ref{thm:dinvts}, we have
\[
|p(b\lambda - a\mu) + q(b(\lambda + 1) - a(\mu + 1))| \leq 2(p + q + 1).
\]

In light of Equation (\ref{lem:Case4ab}) and the fact that $p,q>0$, if $\lambda,\mu\not\in\{-1,0\}$, it is routine to check that the terms $pb\lambda$, $-pa\mu$, $qb(\lambda+1)$ and $-qa(\mu+1)$ all have the same sign. Equation (\ref{eq:dbound}) follows by distributing the absolute value. On the other hand, if $\lambda\in\{-1,0\}$ or $\mu\in\{-1,0\}$, a similar case-by-case analysis proves Equation (\ref{eq:dbound}).

Suppose $\lambda=0$. 
Analogous arguments also work in the case $\mu=-1$. Let $a=0$. 
We claim that we can improve the right side of the Equation (\ref{eq:dbound}) to $2q$.
Notice that the third to last column of $A^T$ is $(\underbrace{0, \ldots,0,}_{p} \underbrace{1, \ldots, 1}_q,0)^T$. Let $x_1,\ldots,x_p,z\in\{\pm1\}$. Thus $(x_1,\ldots,x_p,1,\ldots,1,z)^T-(x_1,\ldots,x_p,-1,\ldots,-1,z)^T\in\text{Im}(2A^T)$. Notice that the difference in the numbers of negative entries in these vectors is $q$.
Recall that any two vectors with the same last entry and containing the same number of negative entries represent the same equivalence class in $S/\text{Im}(2A^T)$. It follows that any two vectors in $S$ of the form $(x_1,\ldots,x_{p+q},z)$ and $ (y_1,\ldots,y_{p+q},z)$, where $\#\{i\,|\,x_i<0\}=\#\{i\,|\,y_i<0\}\pm q$ represent the same equivalence class in $S/\text{Im}(2A^T)$. In particular, for all $0\le i\le p$, any vector $v_1$ in $S$ with $i$ negative entries among its first $p+q$ entries represents the same class of any vector $v_2$ in $S$ with $p+i$ negative entries among its first $p+q$ entries and whose final entry is identical to the final entry of $v_1$. It follows that $|S/\text{Im}(2A^T)|\le 2(p+q+1)-2(p+1)=2q$.

Next let $|a|=1$. We claim that we can improve the right side of the Equation (\ref{eq:dbound}) to $p+2q+1$. Notice that the third to last column of $A^T$ is $(\underbrace{0, \ldots,0,}_{p} \underbrace{1, \ldots, 1}_q,a)^T$. Thus $(x_1,\ldots,x_p,1,\ldots,1,a)^T-(x_1,\ldots,x_p,-1,\ldots,-1,-a)^T\in\text{Im}(2A^T)$. Arguing as above, we see that any two vectors in $S$ of the form $(x_1,\ldots,x_{p+q},a)$ and $ (y_1,\ldots,y_{p+q},-a)$, where $\#\{i\,|\,x_i<0\}=\#\{i\,|\,y_i<0\}\pm q$ represent the same equivalence class in $S/\text{Im}(2A^T)$. It follows that $|S/\text{Im}(2A^T)|\le 2(p+q+1)-(p+1)=p+2q+1$.
\end{proof}

\begin{lem} Suppose $U_{C_1}\cap U_{C_2}=\varnothing$. Then, using the notation in Lemma \ref{lem:lambdamu_slice}, either $\lambda\in \{-1, 0\}$ or $\mu\in \{-1, 0\}$.
\label{lem:-1,0}
\end{lem}

\begin{proof}
First assume that $\lambda,\mu\not\in\{-1,0\}$. Then by Equation (\ref{lem:Case4ab}), all terms have the same sign so $a,b\neq0$ and by Lemma \ref{lem:lambdamu_slice}, Equation (\ref{eq:dbound}) holds.
Consequently, we cannot have that both $|b\lambda|+|a\mu|\ge3$ and $|b(\lambda+1)|+|a(\mu+1)|\ge3$. Moreover note that since $a,b\neq0$, $|b\lambda|+|a\mu|\ge2$ and $|b(\lambda+1)|+|a(\mu+1)|\ge2$. Hence either $|b\lambda|+|a\mu|=2$ or $|b(\lambda+1)|+|a(\mu+1)|=2$.
First assume that $|b\lambda|+|a\mu|=2$. Then since neither term is 0, we have that $\lambda=\mu=1$ and $|a|=|b|=1$, but this violates Equation (\ref{lem:Case4ab}). 
If $|b(\lambda+1)|+|a(\mu+1)|=2$, we arrive to a similar contradiction. Hence either $\lambda\in\{-1,0\}$ or $\mu\in\{-1,0\}$.
\end{proof}

In light of Lemma \ref{lem:-1,0} and Equations (\ref{lem:Case4r}), (\ref{lem:Case4s}), and (\ref{lem:Case4ab}) up to reversing the roles of $r$ and $s$, we may assume that $\lambda\in\{-1,0\}$. Moreover, up to reversing the roles of $p$ and $q$, we may assume that $\lambda=0$.

\begin{thm} 
Suppose exactly two parameters are positive and $\frac{1}{p} + \frac{1}{q} + \frac{1}{r} + \frac{1}{s} > 0$. If $L=P(p,q,r,s)$ is $\chi-$slice, then $L$ or $-L$ is isotopic to a link in one of the following sets:
\begin{enumerate}[(a)]
\item\label{label:-1} $\{P(2,2,-3,-6), P(1,1,-2,-6)\}$;
\item\label{label:2a} $P\langle p,q,-q,-(p+1)\rangle$;
\item\label{label:2b} $P\langle p,q,-q,-(p+4)\rangle$; or
\item\label{label:1} $P\langle p,q,-q-a^2,-q-b^2\rangle,$ where $ab=-q$ and $p\ge\frac{(|a|+|b|-2)q-2}{2}$.

\end{enumerate}
\label{thm:2neg}
\end{thm}

\begin{proof} 
Without loss of generality, assume that $p,q>0$ and $r,s<0$.
First suppose $U_{C_1}\cap U_{C_2}\neq\varnothing$. Then by Proposition \ref{prop:p=q=2}, $p=q=2$ and $(r,s)\in\{(-6,-3),(-2,-3),(-2,-6)\}$. 
The case $(2,2,-2,-3)$ falls under Theorem \ref{thm:2neg}\ref{label:2a} and the case $(2,2,-2,-6)$ falls under Theorem \ref{thm:2neg}\ref{label:2b}.
In case $(2,2,-6,-3)$, $L$ or $-L$ belongs to the set $P\langle 2,2,-3,-6\rangle=\{P(2,2,-3,-6), P(2,-3,2,-6)\}$. By Lemma \ref{lem:2326}, $P(2,-3,2,-6)$ is not $\chi-$slice. Hence $L$ or $-L$ is isotopic to $P(2,2,-3,-6)$, giving the first link in the set listed in Theorem \ref{thm:2neg}\ref{label:-1}.

Next suppose that $U_{C_1}\cap U_{C_2}=\varnothing$.
By the above remarks, we may assume that $\lambda=0$.
By Equation (\ref{eq:dbound}) from Lemma \ref{lem:lambdamu_slice} and Equation (\ref{lem:Case4ab}),
\begin{align}
    |a\mu|p +(|b|+|a(\mu+1)|)q &\le 2p+2q+2 \text{ and } \label{eq:lam0eq7}\\
    ab &=-q(\mu+1). \label{eq:lam0eq6}
\end{align}
It is easy to check that $\mu\in\{-2,-1,0,1\}$, otherwise one of the two above equations is necessarily violated. \\

\underline{$\mu=1$:}
If $\mu=1$, then Equations (\ref{eq:lam0eq7}) and (\ref{eq:lam0eq6}) reduce to $|a|p +(|b|+2|a|)q\le 2p+2q+2$ and $ab=-2q$. Consequently $a,b\neq0$. It is clear that if $|a|\ge2$, then the inequality if violated. Hence $|a|=1$ and $|b|=2q$. By Lemma \ref{lem:lambdamu_slice}, the inequality simplifies to $p +(2q+2)q\le p+2q+1$, which is not possible since $q>0$.\\

\underline{$\mu=-2$:}
If $\mu=-2$, then $2|a|p+(|b|+|a|)q\le 2p+2q+2$ and $ab=q$. Again, $a,b\neq0$. If $|a|\ge2$, then $2p+2q+2\ge 2|a|p+(|b|+|a|)q\ge 4p+3q$, which is not possible since $p,q>0$. Thus $|a|=1$ so that by Lemma \ref{lem:lambdamu_slice}, $2p+(|b|+1)q\le p+2q+1$ and $|b|=q$. This is only possible when $|a|=|b|=q=p=1$; hence $(p,q,r,s)=(1,1,-2,-6)$. It follows that $L$ or $-L$ belongs to the set $P\langle 1,1,-2,-6\rangle$. By Lemma \ref{lem:flype}, this set contains exactly one pretzel link; hence $L$ or $-L$ is isotopic to $P(1,1,-2,-6)$, giving the second link in the set listed in Theorem \ref{thm:2neg}\ref{label:-1}.\\

\underline{$\mu=0$:}
If $\mu=0$, then $ab=-q$ and $r=-q-a^2$ and $s=-q-b^2$; moreover, by Lemma \ref{lem:lambdamu_slice}, $p\ge\frac{(|a|+|b|-2)q-2}{2}$. This yields the family in Theorem \ref{thm:2neg}\ref{label:1}.\\

\underline{$\mu=-1$:}
If $\mu=-1$, then either $a=0$ or $b=0$ by Equation (\ref{eq:lam0eq7}) (but not both, since $\det A^T\neq0$). If $a=0$, by Lemma \ref{lem:lambdamu_slice}, $|b|q\le 2q$, implying that $|b|\le 2$, $r=-q$ and $s=-p-b^2$. If $b=0$, then by Lemma \ref{lem:lambdamu_slice}, $|a|p\le 2p$, implying that $|a|\le 2$, $r=-q-a^2$ and $s=-p$. Up to permuting and relabeling, both cases fall under the families in Theorem \ref{thm:2neg}\ref{label:2a} and \ref{label:2b}.\\ 
\end{proof}

\subsection{Case D} 

In this case, let $p>0$ and $q,r,s < 0$. In order to apply Donaldson's theorem, the intersection form must be negative definite so we will assume that $\frac{1}{p} + \frac{1}{q} + \frac{1}{r} + \frac{1}{s} > 0$. 
Figure \ref{fig:p>0} shows the plumbing diagram associated to $X:=X(p,q,r,s)$. 
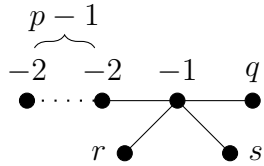
\begin{figure}[h]
    \centering
    \begin{tikzpicture}[dot/.style = {circle, fill, minimum size=1pt, inner sep=0pt, outer sep=0pt}]
\tikzstyle{smallnode}=[circle, inner sep=0mm, outer sep=0mm, minimum size=2mm, draw=black, fill=black];

\node[smallnode, label={90:$-1$}] (-3) at (0,0) {};
\node[smallnode, label={90:$-2$}] (p1) at (-2,0) {};
\node[smallnode, label={90:$-2$}] (p2) at (-1,0) {};
\node[smallnode, label={90:$q$}] (q) at (1,0) {};
\node[smallnode, label={180:$r$}] (r) at (-.7,-.7) {};
\node[smallnode, label={0:$s$}] (s) at (0.7,-0.7) {};

\draw[-] (-3) -- (p2);
\draw[loosely dotted, thick] (p2) -- (p1);
\draw[decorate,decoration={brace,amplitude=5pt,raise=.7cm,mirror},yshift=0pt] (p2) -- (p1) node [midway,yshift=1.1cm]{$p-1$};
\draw[-] (-3) -- (q);
\draw[-] (-3) -- (r);
\draw[-] (-3) -- (s);

\end{tikzpicture}
\caption{$X=X(p,q,r,s)$, where $p>0$ and $q,r,s<0$}\label{fig:p>0}
\end{figure}

\begin{prop}
Suppose exactly one parameter is positive and $\frac{1}{p}+\frac{1}{q}+\frac{1}{r}+\frac{1}{s}>0$. If $L=P(p,q,r,s)$ is $\chi-$slice, then $L$ or $-L$ is isotopic to a link in the set
$$P\langle a,-a-x_1^2-x_2^2-x_3^2,-a-y_1^2-y_2^2-y_3^2,-a-z_1^2-z_2^2-z_3^2\rangle,$$
where $$\Big|\det\begin{bmatrix}
x_1  & y_1 & z_1 \\
     x_2  & y_2 & z_2\\
    x_3 & y_3 & z_3\end{bmatrix}\Big|\le 8$$
and 
\begin{equation}
    x_1y_1+x_2y_2+x_3y_3=x_1z_1+x_2z_2+x_3z_3=y_1z_1+y_2z_2+y_3z_3=-p.
\end{equation}
\label{prop:p>0}
\end{prop}

\begin{proof}
The result follows from analogous arguments as in the proof of Proposition \ref{prop:c}.
\end{proof}

\subsection{Proof of Theorem \ref{thm:main4strandednotchislice}} For convenience, we recall the theorem here.

\fourstrandednotchislice*
\begin{proof}
The result follows from Lemma \ref{lem:4pos}, Theorems \ref{thm:pqr>0} and \ref{thm:2neg}, and Proposition \ref{prop:p>0}.
\end{proof}

\printbibliography
\end{document}